\documentclass{amsart}
\usepackage{amsmath}
\usepackage{amsfonts}
\usepackage{graphics}
\usepackage{epsfig}
\usepackage{amssymb}
\usepackage{amscd}
\usepackage[all]{xy}
\usepackage{latexsym}
\usepackage{graphicx}
\usepackage{multirow}
\usepackage{geometry}

\newtheorem{thm}{Theorem}[section]
\newtheorem{cor}[thm]{Corollary}
\newtheorem{lem}[thm]{Lemma}
\newtheorem{prop}[thm]{Proposition}
\newtheorem{conj}[thm]{Conjecture}
\newtheorem{quest}[thm]{Question}

\theoremstyle{definition}
\newtheorem{Def}[thm]{Definition}
\newtheorem{rem}[thm]{Remark}
\newtheorem*{ack}{Acknowledgement}

\newtheorem{case}{Case}

\numberwithin{equation}{section}
\numberwithin{figure}{section}

\def\rchi{{\hbox{\raise1.5pt\hbox{$\chi$}}}}
\def\Aut{{\text{\rm{Aut}}}}

\def\isom{\cong}
\def\tensor{\otimes}
\def\dsum{\oplus}

\def\a{\alpha}
\def\b{\beta}
\def\lam{\lambda}

\def\gam{\gamma}
\def\Gam{\Gamma}

\newcommand{\Mbar}{{\overline{\mathcal{M}}}}
\newcommand{\Qbar}{{\overline{\mathbb{Q}}}}

\newcommand{\bP}{{\mathbb{P}}}
\newcommand{\bC}{{\mathbb{C}}}

\newcommand{\bE}{{\mathbb{E}}}

\newcommand{\bQ}{{\mathbb{Q}}}
\newcommand{\bR}{{\mathbb{R}}}

\newcommand{\bZ}{{\mathbb{Z}}}
\newcommand{\cA}{{\mathcal{A}}}

\newcommand{\cM}{{\mathcal{M}}}

\newcommand{\cK}{{\mathcal{K}}}

\newcommand{\cO}{{\mathcal{O}}}

\newcommand{\cX}{{\mathcal{X}}}
\newcommand{\la}{{\langle}}
\newcommand{\ra}{{\rangle}}
\newcommand{\half}{{\frac{1}{2}}}
\newcommand{\bp}{{\mathbf{p}}}
\newcommand{\bx}{{\mathbf{x}}}

\newcommand{\rar}{{\rightarrow}}
\newcommand{\lrar}{{\longrightarrow}}

\newcommand{\hxi}{{\hat{\xi}}}

\begin{document}
\large
\setcounter{section}{0}

\title[Spectral curve of Eynard-Orantin
recursion and Laplace transform]
{The spectral curve of 
the Eynard-Orantin  recursion via the
Laplace transform}

\author[O.\ Dumitrescu]{Olivia Dumitrescu}
\address{
Department of Mathematics\\
University of California\\
Davis, CA 95616--8633, U.S.A.}
\email{dolivia@math.ucdavis.edu}

\author[M.\ Mulase]{Motohico Mulase}
\address{
Department of Mathematics\\
University of California\\
Davis, CA 95616--8633, U.S.A.}
\email{mulase@math.ucdavis.edu}

\author[B.\ Safnuk]{Brad Safnuk}  
\address{
Department of Mathematics\\
Central Michigan University\\
Mount Pleasant, MI 48859}
\email{brad.safnuk@cmich.edu}

\author[A.\ Sorkin]{Adam Sorkin}
\address{
Department of Mathematics\\
University of California\\
Davis, CA 95616--8633, U.S.A.}
\email{asorkin@math.ucdavis.edu}

\begin{abstract}

The Eynard-Orantin
recursion formula provides an effective 
tool for certain enumeration problems in
geometry. The formula requires a
\emph{spectral curve} and the recursion kernel.
We present a uniform construction 
of the spectral curve and the recursion kernel
from the unstable geometries
of the original counting problem. 
We examine this construction
 using four concrete examples:
Grothendieck's
dessins d'enfants (or higher-genus analogue
of the Catalan numbers), the intersection numbers of
tautological cotangent classes on
 the moduli stack of stable pointed curves, 
single Hurwitz numbers,  
and the stationary Gromov-Witten invariants
of the complex projective line.
\end{abstract}


\maketitle

\allowdisplaybreaks

\tableofcontents

\section{Introduction}
\label{sect:Intro}

\emph{What is the mirror dual object of the 
Catalan numbers?} 
We wish to make sense  of this question in 
the present paper. Since the homological mirror
symmetry is a categorial equivalence, it does
not require the existence of 
underlying spaces to which the
categories are associated. By identifying  
the Catalan numbers with a counting problem
similar to Gromov-Witten theory, we come up
with an \emph{equation}
\begin{equation}
\label{eq:x=z+1/z}
x = z +\frac{1}{z}
\end{equation}
as their mirror dual. It is not a coincidence
that (\ref{eq:x=z+1/z}) is the Landau-Ginzburg
model in one variable \cite{Ballard, Mirror}.
Once the mirror dual object is identified,
we can calculate the \emph{higher-genus
analogue} of the Catalan numbers
using the \emph{Eynard-Orantin topological 
recursion formula}. This recursion 
therefore provides a mechanism of
calculating the higher-order 
quantum corrections term by term.

The purpose of this paper is to present a
systematic construction of genus $0$ spectral 
curves of the Eynard-Orantin recursion 
formula \cite{EO1, EO3}. 
Suppose we have a symplectic space $X$ on 
the A-model side. If the Gromov-Witten theory
of $X$ is controlled by an integrable system,
then the homological mirror dual of $X$ is
expected to be 
a family of \emph{spectral curves} $\Sigma$. 
Let us consider the descendant Gromov-Witten
invariants of $X$ as a function in integer variables.
The \emph{Laplace transform} of these
functions are symmetric meromorphic functions 
defined on the products of $\Sigma$.
We expect that 
they satisfy the {Eynard-Orantin
topological recursion}  on the B-model side
defined on the curve $\Sigma$.

More specifically, we construct the
spectral curve using the Laplace transform
of the descendant Gromov-Witten type
invariants for the \emph{unstable} geometries
$(g,n)=(0,1)$ and $(0,2)$.
We give four concrete examples in this paper:
\begin{itemize}
\item The number of  dessins d'enfants 
of Grothendieck, which can be thought of
as higher-genus analogue of the Catalan numbers.
\item The $\psi$-class intersection numbers
$\la \tau_{d_1}\cdots \tau_{d_n}\ra_{g,n}$
on the moduli space $\Mbar_{g,n}$ of pointed
stable curves \cite{CMS,DVV,EO1,K1992,W1991}.
\item Single Hurwitz numbers \cite{BM, EMS,MS}.
\item The stationary Gromov-Witten invariants
of $\bP^1$ \cite{NS2,OP2}.
\end{itemize}
The spectral curves we construct are 
listed in Table~1. 
The Eynard-Orantin recursion formula for
the single Hurwitz numbers \cite{BEMS, BM, 
EMS, MZ} and the $\psi$-class intersection
numbers \cite{EO1} are known. 
Norbury and Scott conjecture that 
the stationary Gromov-Witten invariants 
of $\bP^1$ also satisfy the Eynard-Orantin recursion
\cite{NS2}. A similar 
statement for the number of dessins d'enfants 
does not seem to be known. We give a full proof
of this fact in this paper. 

\begin{table}[htb]
  \centering
  
  \begin{tabular}{|c||c|}

\hline 

Grothendieck's Dessins & $\begin{cases}
x = z +\frac{1}{z}\\
y = -z
\end{cases}$  \tabularnewline
\hline 
$\la \tau_{d_1}\cdots \tau_{d_n}\ra_{g,n}$ & $\begin{cases}
x = z^2\\
y = -z
\end{cases}$  \tabularnewline
\hline 
Single Hurwitz Numbers& 
$\begin{cases}
x = z e^{1-z}\\
y = e^{z-1}
\end{cases}$  \tabularnewline
\hline 
 Stationary GW
 Invariants
of $\bP^1$ & $\begin{cases}
x = z +\frac{1}{z}\\
y = -\log (1+z^2)
\end{cases}$  \tabularnewline
\hline 
\end{tabular}
\bigskip

  \caption{Examples of spectral curves.}
\end{table}

Let $D_{g,n}(\mu_1,\dots,\mu_n)$ denote the
weighted count of clean Belyi morphisms 
of  smooth connected algebraic curves of genus 
$g$ with $n$ poles of order $(\mu_1,\dots,\mu_n)$. 
We first prove

\begin{thm}
\label{thm:Dgn intro}
For $2g-2+n\ge 0$ and $n\ge 1$, the number of 
{clean} Belyi morphisms 
satisfies the following equation:
\begin{multline}
\label{eq:Dgn intro}
\mu_1 D_{g,n}(\mu_1,\dots,\mu_n)
=
\sum_{j=2}^n (\mu_1+\mu_j-2)
D_{g,n-1}(\mu_1+\mu_j-2, \mu_{[n]\setminus
\{1,j\}})
\\
+
\sum_{\a+\b=\mu_1-2}\a\b
\Bigg[
D_{g-1,n+1}(\a,\b,\mu_{[n]\setminus\{1\}})
+
\sum_{\substack{g_1+g_2=g\\
I\sqcup J=\{2,\dots,n\}}}
D_{g_1,|I|+1}(\a,\mu_I)D_{g_2,|J|+1}(\b,\mu_J)
\Bigg],
\end{multline}
where $\mu_I=(\mu_i)_{i\in I}$
for a subset $I\subset [n]=\{1,2,\dots,n\}$.
\end{thm}

The simplest case
$$
D_{0,1}(2m) = \frac{1}{2m} \;C_m 
$$
is  given by the Catalan number 
$C_m=\frac{1}{m+1}\binom{2m}{m}$.
The next case $D_{0,2}(\mu_1,\mu_2)$ is 
calculated in \cite{Kodama,KP}.
Note that the $(g,n)$-terms appears
also on the right-hand side of (\ref{eq:Dgn intro}).
Therefore, this  is merely an 
equation, not an effective recursion formula. 

Define the Eynard-Orantin differential form
by
$$
W_{g,n}^D(t_1,\dots,t_n)
=d_1\cdots d_n 
\sum_{\mu_1,\dots,\mu_n> 0}
D_{g,n}(\mu_1,\dots,\mu_n)
e^{-(\mu_1w_1+\cdots+\mu_n w_n)},
$$
where the $w_j$-coordinates and $t_j$-coordinates
are related by
$$
e^{w_j} = \frac{t_j+1}{t_j-1}+\frac{t_j-1}{t_j+1}.
$$
Then 
\begin{thm}
\label{thm:DEO intro}
The Eynard-Orantin differential forms 
for $2g-2+n>0$ satisfy
the following topological recursion formula
\begin{multline}
\label{eq:DEO intro}
W_{g,n}^D(t_1,\dots,t_n)
\\
=
-\frac{1}{64} \; 
\frac{1}{2\pi i}\int_\gam
\left(
\frac{1}{t+t_1}+\frac{1}{t-t_1}
\right)
\frac{(t^2-1)^3}{t^2}\cdot \frac{1}{dt}\cdot dt_1
\Bigg[
W_{g-1,n+1}^D(t,{-t},t_2,\dots,t_n)
\\
+
\sum_{j=2}^n
\bigg(
W_{0,2}^D(t,t_j)W_{g,n-1}(-t,t_2,\dots,\widehat{t_j},
\dots,t_n)
+
W_{0,2}^D(-t,t_j)W_{g,n-1}(t,t_2,\dots,\widehat{t_j},
\dots,t_n)
\bigg)
\\
+
\sum^{\text{stable}} _
{\substack{g_1+g_2=g\\I\sqcup J=\{2,3,\dots,n\}}}
W_{g_1,|I|+1}^D(t,t_I) W_{g_2,|J|+1}^D({-t},t_J)
\Bigg].
\end{multline}
\end{thm}

This is now a recursion formula, since the 
topological type $(g',n')$ of the Belyi morphisms
appearing on the right-hand side satisfies
$$
2g'-2+n' = (2g-2+n)-1,
$$
counting the contributions from the disjoint union
of the domain curves additively.
A corollary to the recursion formula is
a combinatorial identity between the 
number of clean Belyi morphisms and
the number   of
lattice points on the
moduli space $\cM_{g,n}$ that has been 
studied in \cite{CMS,MP2010,N1,N2,NS1}. 

\begin{cor}
\label{cor:intro}
\begin{equation}
\label{eq:DinN intro}
D_{g,n}(\mu_1,\dots,\mu_n)
=
\sum_{\ell_1>\frac{\mu_1}{2}}
\cdots
\sum_{\ell_n>\frac{\mu_n}{2}}
\prod_{i=1}^n
\frac{2\ell_i-\mu_i}{\mu_i}
\binom{\mu_i}{\ell_i}
N_{g,n}(2\ell_1-\mu_i,\cdots,2\ell_n-\mu_n),
\end{equation}
where $N_{g,n}(\mu_1,\dots,\mu_n)$ is defined
by (\ref{eq:Ngn}).
\end{cor}

The  recursion formula (\ref{eq:DEO intro})
is a typical example of the
\emph{Eynard-Orantin recursion} we discuss in this
paper. 
We establish this theorem by taking the Laplace
transform of (\ref{eq:Dgn intro}).
This is indeed a general theme. For every known 
case of the Eynard-Orantin recursion, we
establish its proof by taking the Laplace transform
of the counting formula like (\ref{eq:Dgn intro}).
For example, for the cases of single Hurwitz numbers
\cite{EMS,MZ} and open Gromov-Witten invariants
of $\bC^3$ \cite{Zhou3, Zhou4},
the counting formulas similar to 
(\ref{eq:Dgn intro}) are called the \emph{cut-and-join}
equations \cite{GJ,V,LLZ, Zhou1, Zhou2}.

The Laplace transform plays a mysterious role
in Gromov-Witten theory. We notice its
appearance in Kontsevich's work \cite{K1992}
that relates the Euclidean volume of $\cM_{g,n}$ 
and the intersection numbers on $\Mbar_{g,n}$,
and also in the work of Okounkov-Pandharipande
\cite{OP1} that relates the single Hurwitz numbers and
the enumeration of topological  graphs.
It has been  proved that in these two cases
the Laplace transform of the quantities in question
satisfies the Eynard-Orantin recursion 
\cite{CMS,EMS,EO2,MP2010,MZ}
for a particular choice of the spectral curve.

\emph{Then what is the role of the 
Laplace transform here?}
The answer we propose in this paper is that
\emph{the Laplace transform defines
the spectral curve}.  Since the spectral curve
is a B-model object, 
\textbf{the Laplace transform
plays the role of mirror symmetry}.

The Eynard-Orantin 
recursion formula is  an  
effective tool in  certain geometric enumeration.
The formula is originated  in
random matrix theory as a machinery to compute
the expectation value of a product of the 
resolvent of random matrices
(\cite{AMM}, \cite{E2004}).
In \cite{EO1, EO3}
Eynard and Orantin propose a novel point of view,
considering the recursion as a mechanism of 
\emph{defining}
 meromorphic symmetric differential forms
$W_{g,n}$
on the  product $\Sigma^n$ of 
a Riemann surface
$\Sigma$
 for every $g\ge 0$
and $n>0$. 
They derive in \cite{EO1, EO3}  many beautiful
properties that these quantities satisfy, including
modularity and
relations to integrable systems.

The effectiveness of the topological recursion in 
string theory is immediately noticed
\cite{DV,EMO,M2,OSY}.
A  remarkable discovery,
connecting the recursion formula and
geometry, is made
by Mari\~no \cite{M2} and
Bouchard, Klemm, Mari\~no and Pasquetti
\cite{BKMP}. It is formulated as the
\emph{Remodeling Conjecture}. 
This conjecture covers many aspects of
both closed and open Gromov-Witten invariants
of arbitrary toric Calabi-Yau threefolds. 
One of their statements    says
the following.
Let $X$ be an arbitrary toric Calabi-Yau threefold,
and $\Sigma$ its mirror curve. 
Apply the Eynard-Orantin 
recursion formula to $\Sigma$. Then
$W_{g,n}$ calculates the
open  Gromov-Witten invariants of $X$.
The validity of the topological recursion 
of \cite{EO1, EO3} is not limited to Gromov-Witten
invariants. It has been applied to  the HOMFLY
polynomials of torus knots \cite{BEM}, 
and understanding 
the role of quantum Riemann surfaces and
certain Seiberg-Witten invariants \cite{GS}.
A speculation also suggests its relation to colored
Jones polynomials and the hyperbolic volume 
conjecture of knot complements \cite{DFM}.

From the very beginning,  effectiveness of the 
Eynard-Orantin recursion in enumerative geometry
was suggested by physicists. 
Bouchard and Mari\~no conjecture in \cite{BM}
that particular generating functions of single
Hurwitz numbers satisfy the Eynard-Orantin topological 
recursion. They have come up to this conjecture
as the limiting case of the remodeling conjecture
for $\bC^3$ when the framing parameter tends to
$\infty$. The spectral curve for this scenario is
the \emph{Lambert curve}
$x = y e^{-y}$. The Bouchard-Mari\~no conjecture
is solved in \cite{BEMS, EMS, MZ}.
The work \cite{EMS} also influenced the solutions to  
 the remodeling conjecture for $\bC^3$
itself. The statement
on the open Gromov-Witten invariants
 was proved in
\cite{Chen, Zhou3, Zhou4}, and the closed case
 was proved in \cite{BCMS, Zhu}.

The Eynard-Orantin topological recursion 
\emph{starts} with
a spectral curve $\Sigma$. Thus it is reasonable
to propose the recursion formalism
whenever there is a natural curve in the problem
we study. Such curves may include the mirror curve
of a toric Calabi-Yau threefold \cite{BKMP, M2},
the zero locus of an A-polynomial \cite{DFM, GS},
the Seiberg-Witten curves \cite{GS},
the torus on which a knot is drawn \cite{BEM},
and the character variety of the fundamental group
of a knot complement relative to $SL(2,\bC)$ 
\cite{DFM}.
Now we ask the opposite question.

\begin{quest}
If an enumerative geometry problem is given, then
how do we find the spectral curve, with which 
the Eynard-Orantin formalism may provide a solution?
\end{quest}

In every work
of \cite{BCMS, CMS, Chen, EMS, 
EO1,EO2,EO3,MP2010, 
MZ,N2,NS2, Zhou3, Zhou4},
the spectral curve is considered to be {given}. 
\emph{How do we know that the particular 
choice of the
spectral curve is correct?}
Our proposal provides an answer to this question: 
\emph{the Laplace
transform of the unstable geometries 
$(g,n) = (0,1)$ and $(0,2)$ 
determines the spectral curve, and
the topological recursion formula itself.}
The key ingredients of the topological recursion 
are the spectral curve and the recursion 
kernel that is determined by the differential forms 
$W_{0,1}$ and $W_{0,2}$. In the literature
starting from \cite{EO1},
the word ``Bergman kernel'' is used for the
differential form $W_{0,2}$. 
But   it has indeed nothing to 
do with the classical \emph{Bergman kernel}
 in complex analysis. It is also treated as the
 universally given 2-form depending only on the
 geometry of the spectral curve. 
 We would rather emphasize in this paper
 that  this ``kernel'' is the Laplace transform
 of the \emph{annulus} amplitude, which should 
 be determined by the counting problem we start 
 with.

 Although it is still vague, our proposal is the following
 
 \begin{conj}[The Laplace transform conjecture]
 If the unstable geometries 
 $(g,n) = (0,1)$ and $(0,2)$ make sense in 
 a counting problem on the A-model
 side, then the Laplace transform
 of the solution to these cases determines the
 spectral curve and the recursion kernel
 of the Eynard-Orantin formalism, which is a 
 B-model theory. Thus the Laplace transform
 plays a role of mirror symmetry.
 The recursion then determines
  the solution to the original counting problem 
    for all $(g,n)$. 
 \end{conj}
 
The Eynard-Orantin recursion is a
 process of \emph{quantization}
 \cite{BEM, GS}. Thus the implication 
 of the conjecture is that 
 \emph{quantum invariants 
 are uniquely 
 determined by the disk and annulus amplitudes}.
 For example, single Hurwitz numbers
 $h_{g,\mu}$ are all determined by the first
 two cases $h_{0,(\mu_1)}$ and
 $h_{0,(\mu_1,\mu_2)}$. The present paper
 and our previous work \cite{EMS,MZ}
  establish this fact. The Lambert curve
  is the mirror dual of the number of 
  trees.

 The organization of this paper is the following.
In Section~\ref{sect:EO} we present the 
Eynard-Orantin recursion formalism 
for the case of genus $0$ spectral curve. 
Higher genus situations will be discussed elsewhere.
Sections~\ref{sect:dessin} and \ref{sect:LT}
deal with the
counting problem of Grothendieck's 
dessins d'enfants. We present our new results
on this problem, which are 
Theorem~\ref{thm:Dgn intro} and
Theorem~\ref{thm:DEO intro}.
We are inspired by
Kodama's beautiful talk \cite{Kodama} 
(that is based on \cite{KP}) to come
up with the generating function of the
Catalan numbers as the spectral curve for
this problem. We are grateful to 
G.~Gliner for drawing our attention to \cite{Kodama}.
The counting problem of the lattice points
on $\cM_{g,n}$ of \cite{CMS, MP2010,N1,N2}
is closely related to the counting 
of dessins, which is also treated in 
Section~\ref{sect:LT}.
The Eynard-Orantin recursion becomes 
identical to the Virasoro constraint condition
for the $\psi$-class intersection numbers
on $\Mbar_{g,n}$. We discuss this relation
in Section~\ref{sect:kont}, 
using Kontsevich's idea that the intersection 
numbers on $\Mbar_{g,n}$ are essentially
the same as Euclidean volume of $\cM_{g,n}$.
Section~\ref{sect:Hurwitz} is devoted to 
single Hurwitz numbers. In our earlier
work \cite{EMS,MZ} we used the Lambert
curve as given. Here we reexamine the 
Hurwitz counting problem and derive
the Lambert curve from the unstable
geometries.
We then consider the Norbury-Scott conjecture
\cite{NS2} in Section~\ref{sect:GWP1},
which states that the generating functions
of stationary Gromov-Witten invariants of
$\bP^1$ satisfy the Eynard-Orantin recursion. 
We are unable to prove this conjecture. What
we establish in this section is \emph{why}
the spectral curve of \cite{NS2} is the right choice
for this problem.

The subject of this paper is closely
related to random matrix theory. Since 
the matrix model side of the story has
been extensively discussed by the
original authors \cite{EO3}, we do not 
deal with that aspect in the current 
paper.

\section{The Eynard-Orantin differential forms
and the topological recursion}
\label{sect:EO}

We use the following mathematical definition
for the topological recursion of Eynard-Orantin
for a genus $0$ spectral curve. 
The differences between our definition and the
original formulation found in \cite{EO1,EO3}
are of the philosophical nature. Indeed, the original
formula and ours produce the exact same answer
in all examples we examine in this paper.

\begin{Def}
\label{def:EO}
We start with $\bP^1$ with a 
\emph{preferred  coordinate} $t$. 
Let $S\subset \bP^1$ be a finite collection of points
and compact real curves such that 
the complement $\Sigma = \bP^1\setminus S$ is 
connected.
The 
\emph{spectral curve} of genus $0$ is the 
data $(\Sigma, \pi)$ consisting
of a Riemann surface  $\Sigma$
 and a simply ramified
holomorphic map
\begin{equation}
\label{eq:xprojection}
\pi:\Sigma\owns t\longmapsto \pi(t) = x\in \bP^1
\end{equation}
so that its differential $dx$ has only simple 
zeros. Let us denote by $R=\{p_1,\dots,p_r\}\subset 
\Sigma$ 
 the 
ramification points, and by 
$$
U = \sqcup_{j=1} ^r U_j
$$
the disjoint union of small neighborhood
 $U_j$ around each $p_j$ such that 
 $\pi:U_j\rightarrow \pi(U_j)\subset \bP^1$ is
 a double-sheeted covering ramified only at $p_j$.
 We denote by $\bar{t}=s(t)$ the local Galois conjugate
of $t\in U_j$.
The canonical sheaf of $\Sigma$ is denoted by $\cK$.
Because of our choice of the preferred coordinate
$t$, we have a preferred 
basis $dt$ for $\cK$ and 
$\partial/\partial t$ for $\cK^{-1}$.
The meromorphic differential
forms $W_{g,n}(t_1,\dots,t_n)$,
$g=0, 1, 2, \dots, n= 1, 2, 3, \dots,$
are said to satisfy the \textbf{Eynard-Orantin
topological recursion} if the following conditions 
are satisfied:
\begin{enumerate}
\item $W_{0,1}(t) \in H^0(\Sigma,\cK)$.
\item $W_{0,2}(t_1,t_2) = \frac{dt_1\cdot dt_2}
{(t_1-t_2)^2} -\pi^*
\frac{dx_1\cdot dx_2}
{(x_1-x_2)^2}
\in  H^0(\Sigma\times \Sigma, \cK^{\tensor 2}(2\Delta))$, 
where $\Delta$ is the diagonal of $\Sigma\times 
\Sigma$.
\item The recursion kernel $K_j(t,t_1)
\in H^0(U_j\times C,(\cK_{U_j}^{-1}\tensor
\cK )(\Delta))$
for $t\in U_j$ and $t_1\in C$ is defined by
\begin{equation}
\label{eq:kernel}
K_j(t,t_1) = \half\; \frac{\int_t ^{\bar{t}}
W_{0,2}(\cdot,t_1)}
{W_{0,1}(\bar{t})-W_{0,1}(t)}.
\end{equation}
The kernel is an algebraic operator
that multiplies $dt_1$ while contracts
$dt$.

\item 
The general
differential forms
$W_{g,n}(t_1,\dots,t_n) \in
 H^0(\Sigma^n, \cK(*R)^{\tensor n})$
are meromorphic symmetric differential forms
with poles  only at the ramification points $R$ for
$2g-2+n>0$, and 
are given by the recursion formula
\begin{multline}
\label{eq:EO}
W_{g,n}(t_1,t_2,\dots,t_n)
= \frac{1}{2\pi i} 
\sum_{j=1} ^r 
\oint_{U_j} K_j(t,t_1)
\Bigg[
W_{g-1,n+1}(t,\bar{t},t_2,\dots,t_n)\\
+
\sum^{\text{No $(0,1)$ terns}} _
{\substack{g_1+g_2=g\\I\sqcup J=\{2,3,\dots,n\}}}
W_{g_1,|I|+1}(t,t_I) W_{g_2,|J|+1}(\bar{t},t_J)
\Bigg].
\end{multline}
Here the integration is taken with respect to
$t\in U_j$ along a positively oriented 
simple closed loop around $p_j$, and
$t_I = (t_i)_{i\in I}$
for a subset $I\subset \{1,2,\dots,n\}$.

\item The differential form
$W_{1,1}(t_1)$ requires a separate treatment
because  $W_{0,2}(t_1, t_2)$
is regular at the ramification points but
has poles elsewhere.
\begin{equation}
\label{eq:W11}
W_{1,1}(t_1) = 
 \frac{1}{2\pi i} 
\sum_{j=1} ^r 
\oint_{U_j} K_j(t,t_1) 
\left.\left[
W_{0,2}(u,v)+\pi^* \frac{dx(u)\cdot dx(v)}
{(x(u)-x(v))^2}
\right]
\right|_{\substack{u=t\\v=\bar{t}}}.
\end{equation}
Let  
 $y:\Sigma\lrar \bC$ 
 be a holomorphic function
defined by
the equation
\begin{equation}
\label{eq:ydx}
W_{0,1}(t) = y(t)dx(t).
\end{equation}
Equivalently, we can define the function 
by contraction $y = i_\cX W_{0,1}$, where 
$\cX$ is the vector field on $\Sigma$ dual to 
$dx(t)$ with respect to the coordinate $t$. 
Then we have an embedding
$$
\Sigma\owns t\longmapsto (x(t),y(t))\in\bC^2.
$$
\end{enumerate}
\end{Def}

\begin{rem}
The recursion (\ref{eq:EO}) 
also applies to
$(g,n)=(0,3)$, which gives $W_{0,3}$
in terms of $W_{0,2}$.
In \cite[Theorem~4.1]{EO1} 
an equivalent but often more
useful formula for $W_{0,3}$ is given:
\begin{equation}
\label{eq:W03}
W_{0,3}(t_1,t_2,t_3) =\frac{1}{2\pi i} 
\sum_{j=1} ^r 
\oint_{U_j}
\frac{W_{0,2}(t,t_1)W_{0,2}(t,t_2)W_{0,2}(t,t_3)}
{dx(t)\cdot dy(t)}.
\end{equation}
\end{rem}

\section{Counting Grothendieck's dessins d'enfants}
\label{sect:dessin}

The A-model side of the problem we 
consider in this section is the counting 
problem of Grothendieck's
\emph{dessins d'enfants} 
(see for example, \cite{Schneps,SL}) for a fixed 
topological type of Belyi morphisms \cite{Belyi}.
Gromov-Witten theory of an algebraic 
variety $X$ is an intersection theory of
naturally defined divisors on the moduli stack 
$\Mbar_{g,n}(X)$ of
stable morphisms from $n$-pointed algebraic 
curves of genus $g$  to the target variety $X$. 
Since we are considering tautological divisors,
their $0$-dimensional intersection points are
also  \emph{natural}. These points determine
a finite set on $\Mbar_{g,n}$ via the 
stabilization morphism.
If
we expect that the Gromov-Witten theory 
of $X$ satisfies the Eynard-Orantin recursion,
then we should also expect that the counting problem
of naturally defined finite sets of points on 
$\Mbar_{g,n}$ may satisfy the Eynard-Orantin 
recursion. 

Pointed curves defined over $\Qbar$ form a
dense subset of $\Mbar_{g,n}$. To specify
$n$, we need to use Belyi morphisms. 
When we identify a curve over $\Qbar$ with 
a Belyi morphism, a natural counting 
problem arises by considering the profile
of the Balyi morphism at the branched points.
In this way we arrive at  canonically defined
finite sets of points on $\Mbar_{g,n}$. 

More specifically, consider a Belyi morphism
\begin{equation}
\label{eq:Belyi}
b:C\lrar \bP^1
\end{equation}
of a smooth algebraic curve $C$ of genus $g$. 
This means $b$ is branched only
over $0,1,\infty\in\bP^1$. By Belyi's Theorem 
\cite{Belyi},
$C$ is defined over $\Qbar$. 
Let $q_1,\dots,q_n$ be poles of $b$ of
orders $(\mu_1,\dots,\mu_n)\in\bZ_+^n$.
This vector of positive integers is the \emph{profile}   of $b$   at $\infty$. In our enumeration we 
\emph{label} all poles of $b$. Therefore, 
an automorphism of a Balyi morphism  
 preserves the set of poles point-wise. 
 
 A \emph{clean} Belyi morphism is a special
 class of Belyi morphism 
 of even degree that has profile 
 $(2,2,\dots,2)$ over the branch point $1\in\bP^1$.
 We note that
 a complex algebraic curve is defined over 
 $\Qbar$ if and only if it admits a clean 
 Belyi morphism. 
 Let us denote by $D_{g,n}(\mu_1,\dots,\mu_n)$
 the number of genus $g$ clean Belyi morphisms
 of  profile $(\mu_1,\dots,\mu_n)$ at $\infty\in\bP^1$.
 This is the number we study in this section.

 We first derive a recursion equation among
 $D_{g,n}(\mu_1,\dots,\mu_n)$ for all $(g,n)$.
 This relation does not provide an effective
 recursion formula, because 
 $D_{g,n}(\mu_1,\dots,\mu_n)$ appears in
 the equation in a complicated manner. 
 We then compute the Laplace transform
 $$
 F_{g,n}^D(w_1,\dots,w_n)
 =\sum_{\mu_1,\dots,\mu_n>0}
 D_{g,n}(\mu_1,\dots,\mu_n)\;
 e^{-(\mu_1w_1+\cdots+\mu_n w_n)},
 $$
 and rewrite the recursion equation in terms
 of the Laplace transformed functions. 
 We then show that the symmetric differential forms
 $$
 W_{g,n}^D =d_1\cdots d_nF_{g,n}^D
 $$
 satisfy the Eynard-Orantin recursion formula.
 This time it is an effective recursion formula for
the 
 generating functions $W_{g,n}^D$ of the number
 $D_{g,n}(\mu_1,\dots,\mu_n)$
 of clean Belyi morphisms.

 Grothendieck 
 visualized the clean Belyi morphism
 by considering the inverse image
 \begin{equation}
 \label{eq:dessin}
 \Gam = b^{-1}([0,1])
 \end{equation}
 of the closed interval $[0,1]\subset\bP^1$
 by $b$
 (see his ``Esquisse
d'un programme'' reprinted
in \cite{SL}). This is what we call
\emph{dessin d'enfant}. It is a topological 
graph drawn on the algebraic curve $C$ being 
considered as a Riemann surface. 
We call each pre-image of $0\in\bP^1$ by $b$
a \emph{vertex} of $\Gam$. Since 
$b$ has profile $(2,\dots,2)$ over
$1\in\bP^1$, a pre-image of $1$ is the midpoint of
an \emph{edge} of $\Gamma$. The complement
$C\setminus \Gam$ of $\Gam$ in $C$ is the
disjoint union of $n$ disks centered at 
each $q_i$. By abuse of terminology we
call each disk a \emph{face} of $\Gam$. 
Then by Euler's formula we have
$$
2-2g = |b^{-1}(0)|-|b^{-1}(1)|+n.
$$
A dessin is a special kind of metric ribbon graph. 
A \emph{ribbon graph} of topological type $(g,n)$
is the $1$-skeleton of a cell-decomposition of a closed
oriented topological surface $C$ of genus $g$
that decomposes the surface into a disjoint union of
$0$-cells, $1$-cells, and $2$-cells. The number
of $2$-cells is $n$. 
Alternatively, a ribbon graph can be defined
as a graph with a cyclic order assigned to 
the incident half-edges at each vertex. 
When a positive real number, the  length,
is assigned to each edge of a ribbon graph,
we call it a \emph{metric} ribbon graph. 
A dessin is thus a metric ribbon graph with 
the same length given to each edge. We usually
consider this length to be $1$, so the
distance between $0$ and $1$ on $\bP^1$ is
measured as $\half$.

The concrete construction of \cite{MP1998}
gives a Belyi morphism to any given 
dessin. Thus the enumeration of clean Belyi 
morphism is equivalent to the enumeration of
ribbon graphs, where we assign length $1$ to 
every edge. The original interest of dessins 
lies in the fact that the 
absolute Galois group 
$Gal(\Qbar/\bQ)$ acts faithfully
on the set of dessins.

An alternative description of a Belyi morphism
is to use the dual graph 
\begin{equation}
\label{eq:dualgraph}
\check\Gam = b^{-1}([1,i\infty]),
\end{equation}
where 
$$
[1,i\infty]= \{1+iy\;|\;0\le y\le \infty\}\subset \bP^1
$$
is the vertical half-line on $\bP^1$
with real part $1$.
This time
the graph $\check\Gam$  has  $n$ labeled vertices
of degrees $(\mu_1,\dots,\mu_n)$.
Since we consider ribbon graphs in the context of
canonical cell-decomposition of the moduli
space $\cM_{g,n}$, we use the terminology
\emph{dessin} for a graph $\check\Gam$
dual to a ribbon graph $\Gam$. This distinction is
important, because when we count the number
of ribbon graphs, we consider the automorphism
of a graph that preserves each face, while 
the  automorphism group of the dual graph, i.e., a dessin,
preserves each vertex point-wise, but can 
permute faces.
In this dual picture, we define the number of
dessins with the automorphism factor by
\begin{equation}
\label{eq:dessin count}
D_{g,n}(\mu_1,\dots,\mu_n) = 
\sum_{\substack{\check\Gam \text{ dessin of}\\
\text{type } (g,n)}} \frac{1}{|\Aut_D
(\check\Gam)|},
\end{equation}
where $\check\Gam$ is a dessin of genus
$g$ with $n$ labeled vertices with prescribed degrees
$(\mu_1,\dots,\mu_n)$, and $\Aut_D
(\check\Gam)$ is the automorphism 
of $\check\Gam$
preserving each vertex point-wise.

Our theme is to find the spectral curve
of the theory by looking at the problem
for unstable curves $(g,n)=(0,1)$ and $(0,2)$. 
The dessins counted in
$D_{0,1}(\mu)$ for an integer
$\mu\in \bZ_+$ are spherical  graphs
that contain only one vertex of degree $\mu$. 
Since any edge of this graph has to start and end
with the same vertex, it is a loop, and thus $\mu$ is
even. So let us put $\mu=2m$. 
Each graph contributes with the weight 
$1/|\Aut_D (\check\Gam)|$
in the enumeration of the number $D_{0,1}(\mu)$.
This automorphism factor makes counting more
difficult.
Note that the automorphism group of 
a spherical dessin with a single vertex
is a subgroup of $\bZ/(2m)\bZ$ that 
preserves the graph. If we place an outgoing arrow
to one of the $2m$ half-edges incident to 
the unique vertex (see Figure~\ref{fig:ribbon01}),
 then we can kill the automorphism 
altogether. Since there are $2m$ choices of 
placing such an arrow, the number of arrowed graphs 
is $2m D_{0,1}(2m)$. This is now an integer.
By a simple bijection argument with the number of
arrangements of $m$ pairs of parentheses, we see that
\begin{equation}
\label{eq:Catalan}
2m D_{0,1}(2m) = C_m = \frac{1}{m+1}\binom{2m}{m},
\end{equation}
where $C_m$ is the $m$-th 
\emph{Catalan number}. We note that the Catalan
numbers appear in the same context of counting graphs
in \cite{HZ}.

\begin{figure}[htb]
\centerline{\epsfig{file=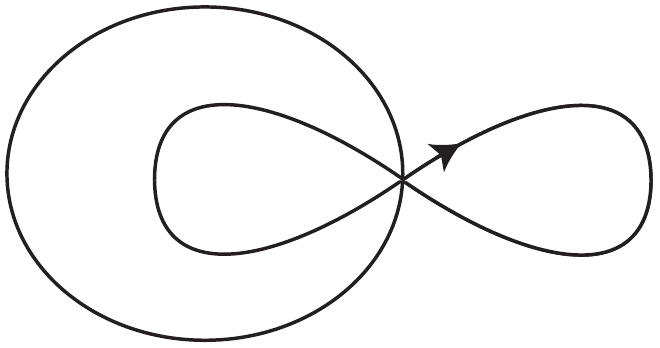, width=2in}}
\caption{An arrowed 
dessin d'enfant of genus $0$ with
one vertex.}
\label{fig:ribbon01}
\end{figure}

Define the Laplace transform of $D_{0,1}(\mu)$ by
\begin{equation}
\label{eq:F01Dtilde}
\widetilde{F}_{0,1}^D=\sum_{m=1} ^\infty D_{0,1}(2m) e^{-2mw}.
\end{equation}
Then the Eynard-Orantin differential
$$
\widetilde{W}_{0,1}^D = d \widetilde{F}_{0,1}^D = 
-\sum_{m=1} ^\infty 2m D_{0,1}(2m) e^{-2mw} dw
= -\sum_{m=1} ^\infty C_m e^{-2mw} dw
$$
is a generating function of the Catalan numbers.
Actually a better  choice is (see \cite{Kodama,KP})
\begin{equation}
\label{eq:z(x)}
z(x) = \sum_{m=0} ^\infty C_m \frac{1}{x^{2m+1}}
= \frac{1}{x}+\frac{1}{x^3}+\frac{2}{x^5}
+\frac{5}{x^7}
+\frac{14}{x^9}+\frac{42}{x^{11}}+\cdots.
\end{equation}
The radius of convergence of this infinite
Laurent series is $2$, hence the series 
converges absolutely for $|x|>2$.
The inverse function of $z=z(x)$ 
near $(x,z) = (\infty, 0)$ is given by
\begin{equation}
\label{eq:x=x(z)}
x = z+\frac{1}{z}.
\end{equation}
This can be easily seen by solving the quadratic
equation $z^2-xz+1=0$ with respect to $z$,
which is equivalent to the quadratic recursion
$$
C_{m+1} = \sum_{i+j = m} C_i \cdot C_j
$$
of Catalan numbers.
To take advantage of these
simple formulas, let us \emph{define} 
 \begin{equation}
 \label{eq:x=ew}
 x=e^w
 \end{equation}
 and allow the $m=0$ term in the Eynard-Orantin
 differential:
 \begin{equation}
 \label{eq:W01D}
 W_{0,1}^D = -\sum_{m=0} ^\infty C_m \;\frac{dx}
 {x^{2m+1}}.
 \end{equation}
 Accordingly the Laplace transform of $D_{0,1}(2m)$
 needs to be modified:
 \begin{equation}
 \label{eq:F01D}
 F_{0,1}^D = \sum_{m=1} ^\infty 
 D_{0,1}(2m)\; e^{-2mw}
 - w = \sum_{m=1} ^\infty D_{0,1}(2m)
\;  \frac{1}{x^{2m}}
 -\log x.
 \end{equation}
Although numerically $D_{0,1}(0) = 0$, 
its infinitesimal behavior is given by
$$
\lim_{m\rar 0}\frac{D_{0,1}(2m)}{x^{2m}}
= -\log x,
$$
which is consistent with 
$$
\lim_{m\rar 0}2mD_{0,1}(2m)
= C_0 = 1.
$$
From (\ref{eq:z(x)}) and (\ref{eq:W01D}), 
we obtain
\begin{equation}
\label{eq:W01D}
W_{0,1}^D = -z(x)\; dx.
\end{equation}
In light of (\ref{eq:ydx}), 
we have identified the spectral curve
for the  counting problem of dessins 
$D_{g,n}(\mu)$. It is given by
\begin{equation}
\label{eq:Dspectral}
\begin{cases}
x = z+\frac{1}{z}\\
y = -z
\end{cases}.
\end{equation}

To compute the recursion kernel of 
(\ref{eq:kernel}), we need to identify
$D_{0,2}(\mu_1,\mu_2)$ for the other
unstable geometry $(g,n) = (0,2)$. 
In  
the dual graph
picture, $D_{0,2}(\mu_1,\mu_2)$
counts the number of spherical  dessins $\check\Gam$
with
two vertices of degree $\mu_1$ and $\mu_2$,
counted with the weight of $1/|\Aut_D(\check\Gamma)|$.
The computation was done by Kodama and 
Pierce in \cite[Theorem~3.1]{KP}.
We also refer to a beautiful lecture by Kodama
\cite{Kodama}.

\begin{prop}[\cite{KP}]
\label{prop:N02}
The number of 
spherical dessinss $\check\Gam$
with
two vertices of degrees $j$ and $k$,
counted with the weight of $1/|\Aut_D(\check\Gamma)|$,
is given by the following formula.
\begin{equation}
\label{eq:D02}
D_{0,2}(\mu_1,\mu_2) = 
\begin{cases}
\frac{1}{2k}\;\binom{2k}{k}
\quad\qquad\qquad \mu_1=0, 
\mu_2=2k\ne 0\\
\\
\frac{1}{4}\;
\frac{1}{j+k}\;\binom{2j}{j}\binom{2k}{k}
\qquad \mu_1 = 2j \ne 0, \mu_2=2k\ne 0
\\
\\
\frac{1}{j+k+1}\;\binom{2j}{j}\binom{2k}{k}
\qquad \mu_1 = 2j+1 , \mu_2=2k+1
\end{cases}.
\end{equation}
All other cases $D_{0,2}(\mu_1,\mu_2) =0$.
Here the automorphism group 
$\Aut_D(\check\Gamma)$ is the topological 
graph automorphisms that fix each vertex, 
but may permute faces.
\end{prop}

\begin{rem}
The first case is irregular. For $\mu_1=0$, the
second vertex has an even degree, and hence 
we have $C_k/(2k)$ graphs. Note that this graph
has $k+1$ faces due to Euler's formula $2=1-k+(k+1)$.
The degree $0$ vertex has to be placed in one of these
faces, which makes the total number of graphs 
$$
\frac{k+1}{2k}\;C_k = \frac{1}{2k}\;\binom{2k}{k}.
$$
However, we are counting
 only connected graphs. Hence degree
$0$ vertices are not 
allowed in our counting.
\end{rem}

In general the number of dessins satisfies the
following:

\begin{thm}
\label{thm:Dgn}
For $g\ge 0$ and $n\ge 1$ 
subject to $2g-2+n\ge 0$, the number of 
dessins (\ref{eq:dessin count})
satisfies a recursion equation
\begin{multline}
\label{eq:Dgn recursion}
\mu_1 D_{g,n}(\mu_1,\dots,\mu_n)
=
\sum_{j=2}^n (\mu_1+\mu_j-2)
D_{g,n-1}\big(\mu_1+\mu_j-2, \mu_{[n]\setminus
\{1,j\}}\big)
\\
+
\sum_{\a+\b=\mu_1-2}\a\b
\Bigg[
D_{g-1,n+1}(\a,\b,\mu_{[n]\setminus\{1\}})
+
\sum_{\substack{g_1+g_2=g\\
I\sqcup J=\{2,\dots,n\}}}
D_{g_1,|I|+1}(\a,\mu_I)D_{g_2,|J|+1}(\b,\mu_J)
\Bigg],
\end{multline}
where $\mu_I=(\mu_i)_{i\in I}$
for a subset $I\subset [n]=\{1,2,\dots,n\}$.
The last sum is over all 
partitions of the genus
$g$ and the index set  $\{2,3,\dots,n\}$ into
two pieces. 
\end{thm}

\begin{rem}
Note that when $g_1=0$ and $I=\emptyset$, 
 $D_{g,n}$ appears in the right-hand side
of (\ref{eq:Dgn recursion}). Therefore, this is
an equation of the number of dessins, not a
recursion formula.
\end{rem}

\begin{proof}
Consider the collection  of genus $g$ dessins 
with $n$  vertices labeled by the 
index set $[n]=\{1,2,\dots,n\}$  and of degrees
$(\mu_1,\dots,\mu_n)$. 
The left-hand side of (\ref{eq:Dgn recursion}) is 
the number of dessins with an outward arrow
placed on one of the incident edges at the 
vertex $1$. The equation is based on the
removal of this edge. 
There are two cases. 
\begin{case}
The arrowed edge connects the vertex $1$ and
vertex $j>1$.
We then remove the edge and 
put the two vertices $1$ and $j$ together
as shown in  Figure~\ref{fig:case1}. 
This operation is better described as shrinking 
the arrowed edge to a point. 
The resulting dessin has one less vertices, but the
genus is the same as before. The degree of the
newly created vertex is $\mu_1+\mu_j-2$, 
while the degrees of all other vertices are unaffected.

\begin{figure}[htb]
\centerline{\epsfig{file=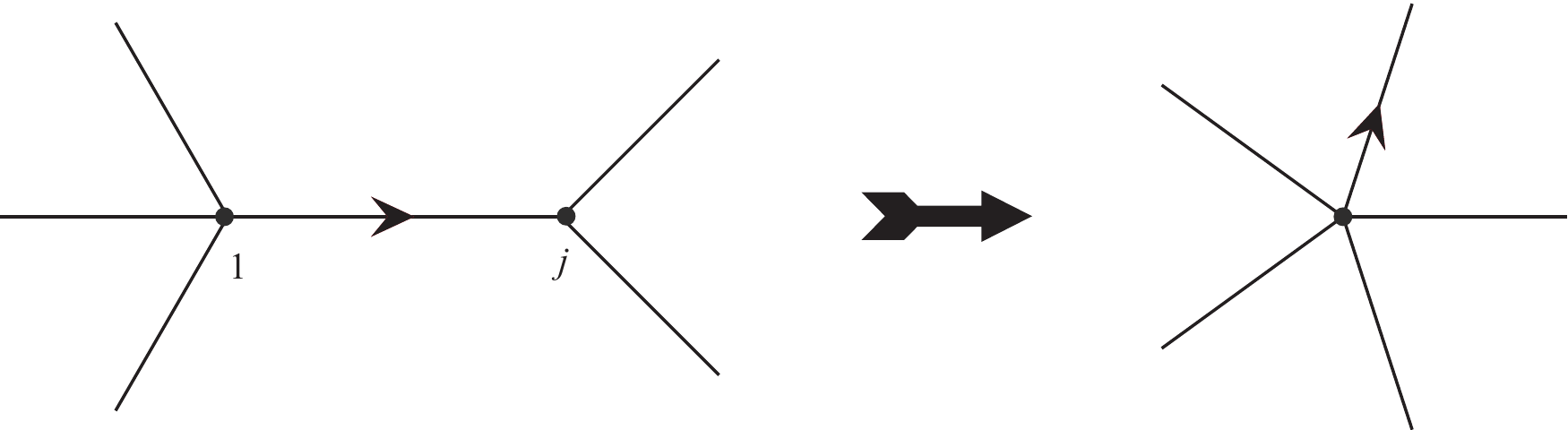, height=1in}}
\caption{The operation that shrinks
 the arrowed edge to a point
and joins two vertices labeled by $1$ and $j$
together.}
\label{fig:case1}
\end{figure}

To make the bijection argument, we
need to be able to reconstruct the original
dessin from the new one. Since both $\mu_1$ and
$\mu_j$ are given as the input value, we have to specify
which edges go to vertex $1$ and which go to $j$
when we separate the vertex of degree
$\mu_1+\mu_j-2$. 
For this purpose, 
what we need is a marker on one of the
incident edges. We group the  marked edge
and $\mu_i-2$ edges
following it according to the cyclic order. The rest of
the $\mu_j-1$ incident edges are also grouped.  
Then we insert an edge and separate the vertex into
two vertices, $1$ and $j$, 
so that the first group of edges are
incident to vertex $1$ and the second group is incident
to $j$, honoring their
cyclic orders (see Figure~\ref{fig:case1}). The contribution
from this case is therefore
$$
\sum_{j=2}^n (\mu_1+\mu_j-2)
D_{g,n-1}\big(\mu_1+\mu_j-2, \mu_{[n]\setminus
\{1,j\}}\big).
$$
\end{case}

\begin{case}
The arrowed edge forms a loop that is attached to 
vertex $1$. We remove this loop from the 
dessin, and separate the vertex into two vertices.
The loop classifies all incident 
half-edges, except for the loop itself,
into two groups: the ones that follow the 
arrowed half-edge in the cyclic order but 
before the incoming end of the loop, and all 
others (see Figure~\ref{fig:case2}). Let $\a$
be the number of half-edges in the first group, 
and $\b$ the rest. Then $\a+\b=\mu_1-2$, 
and we have created two vertices of degrees 
$\a$ and $\b$. 

To recover the original dessin from the 
new one, we need to mark a half-edge
from each vertex so that we can put the
loop back to the original place.
The number of choices of these markings is $\a\b$.

\begin{figure}[htb]
\centerline{\epsfig{file=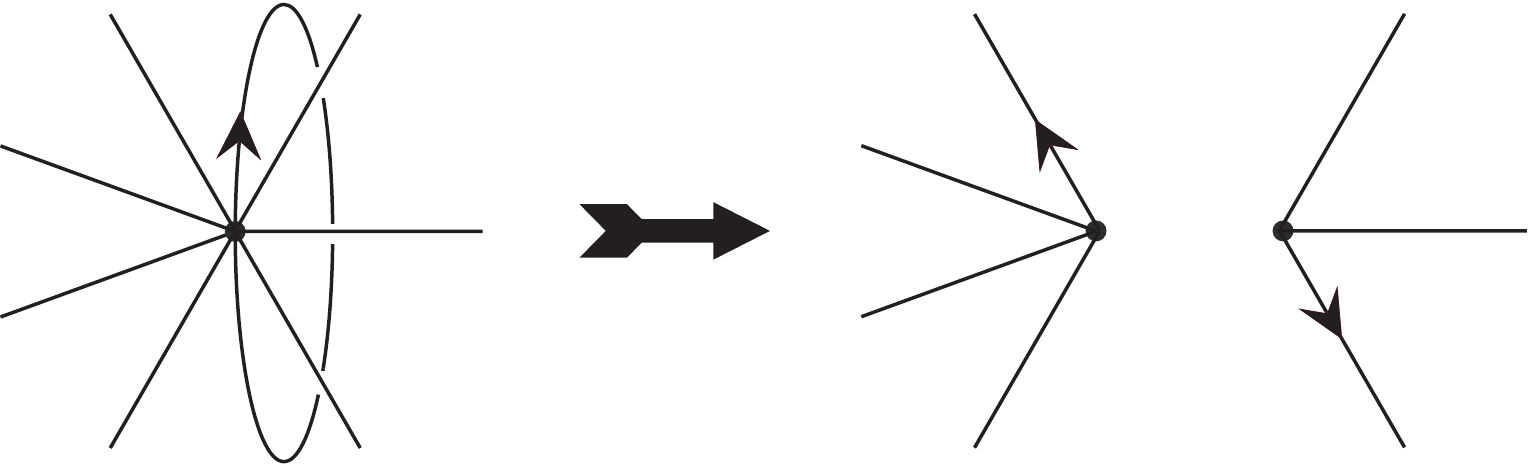, height=1in}}
\caption{The operation that removes a loop, and 
separates the incident vertex into two vertices.}
\label{fig:case2}
\end{figure}

The operation of the
removal of the loop and the separation of the 
vertex into two vertices certainly
increases the number of vertices
from $n$ to $n+1$. This operation also affects
the genus of the dessin. If the resulting dessin is 
 connected, then $g$ goes down to $g-1$.
If the result is the disjoint union of two dessins
of genera $g_1$ and $g_2$, then we have 
$g=g_1+g_2$. Altogether the 
contribution from this case is
$$
\sum_{\a+\b=\mu_1-2}\a\b
\Bigg[
D_{g-1,n+1}(\a,\b,\mu_{[n]\setminus\{1\}})
+
\sum_{\substack{g_1+g_2=g\\
I\sqcup J=\{2,\dots,n\}}}
D_{g_1,|I|+1}(\a,\mu_I)D_{g_2,|J|+1}(\b,\mu_J)
\Bigg].
$$
Note that the outward arrow we place
defines the two groups of incident half-edges
uniquely, since one is after and the other 
before the arrowed half-edge according to the
cyclic order. Thus we do not need to symmetrize
$\a$ and $\b$. Indeed, if the arrow is placed
in the other end of the loop, then $\a$ and $\b$ 
are interchanged.
\end{case}
The right-hand side of
the equation  (\ref{eq:Dgn recursion})
is the sum of the above two contributions.
\end{proof}

\begin{rem}
The equation (\ref{eq:Dgn recursion}) 
is considerably simpler, compared to
the recursion formula
for the number of ribbon graphs
with integral edge lengths that is
proved in \cite[Theorem~3.3]{CMS}.
The edge removal operation of \cite{CMS} is
the dual operation of the edge shrinking operations
of Case~1 and Case~2 above, and
the placement of an arrow corresponds
to the \emph{ciliation}
of \cite{CMS}. In the dual picture, the graphs enumerated
in  \cite{CMS} are 
more restrictive than arbitrary clean dessins,
which makes the equation more
complicated. We also note that
\cite[Theorem~3.3]{CMS} is a recursion 
formula, not just a mere relation like 
what we have in (\ref{eq:Dgn recursion}).
In this regard, 
(\ref{eq:Dgn recursion}) 
is indeed similar to
 the cut-and-join
equation (\ref{eq:caj}) of \cite{GJ,V}. We will 
come back to this point in Section~\ref{sect:Hurwitz}.
\end{rem}

The relation (\ref{eq:Dgn recursion}) 
becomes an effective recursion formula
after taking the Laplace transform.

\section{The Laplace transform of the 
number of dessins and ribbon graphs}
\label{sect:LT}

In this section we derive the 
Eynard-Orantin recursion formula for
the generating functions of the number
of dessins. The key technique is the 
Laplace transform.

Since the projection $x=z+1/z$ 
 of the spectral curve
to the $x$-coordinate plane has two ramification 
points $z=\pm 1$, it is natural to introduce a
coordinate that has these ramification points 
at $0$ and $\infty$. So we define
\begin{equation}
\label{eq:z(t)}
z=\frac{t+1}{t-1}.
\end{equation}

\begin{prop}
\label{prop:F02D}
The Laplace transform of $D_{0,2}(\mu_1,\mu_2)$
is given by
\begin{multline}
\label{eq:F02D}
F_{0,2} ^D (t_1,t_2)
\overset{\rm{def}}{=}
\sum_{\mu_1,\mu_2 > 0}
D_{0,2}(\mu_1,\mu_2)\;e^{-(\mu_1 w_1+\mu_2w_2)}
=
-\log\big(1-z(x_1) z(x_2)\big)
\\
=\log(t_1-1) +\log(t_2-1) -\log(-2(t_1+t_2)),
\end{multline}
where $z(x)$ is the generating function of
the Catalan numbers (\ref{eq:z(x)}),
and the variables
 $t,w,x,z$ are related by (\ref{eq:x=ew}),
(\ref{eq:Dspectral}), and (\ref{eq:z(t)}).
We then  have 
\begin{equation}
\label{eq:W02D}
W_{0,2}^D (t_1,t_2) 
= d_1d_2 F_{0,2} ^D(t_1,t_2)
= \frac{dt_1\cdot dt_2}{(t_1-t_2)^2}
-\frac{dx_1\cdot dx_2}{(x_1-x_2)^2}
=\frac{dt_1\cdot dt_2}{(t_1+t_2)^2}.
\end{equation}
\end{prop}

\begin{proof}
In terms of $x=e^w$, the Laplace transform
(\ref{eq:F02D}) is given by
\begin{multline}
\label{eq:D02inx}
\sum_{\mu_1,\mu_2>0}
D_{0,2}(\mu_1,\mu_2)\;e^{-(\mu_1 w_1+\mu_2w_2)}
\\
=
\frac{1}{4}\sum_{j,k=1}^\infty\frac{1}{j+k}
\binom{2j}{j}\binom{2k}{k}
\frac{1}{x_1 ^{2j}}\;\frac{1}{x_2^{2k}}
+
\sum_{j,k=0}^\infty\frac{1}{j+k+1}
\binom{2j}{j}\binom{2k}{k}
\frac{1}{x_1 ^{2j+1}}\;\frac{1}{x_2^{2k+1}}.
\end{multline}
Since 
\begin{equation}
\label{eq:dx and dz}
dx = \left(1-\frac{1}{z^2}\right)dz,
\end{equation}
we have
\begin{equation}
\label{eq:dx and dz}
x\frac{d}{dx} 
= 
\frac{
z+\frac{1}{z}}{1-\frac{1}{z^2}}\;
\frac{d}{dz}
=
\frac{z(z^2+1)}{z^2-1}\;\frac{d}{dz}.
\end{equation}
To make the computation simpler,
let us introduce
\begin{equation}
\label{eq:xi0D}
\xi_0(x)=\sum_{m=0}^\infty \binom{2m}{m}
\frac{1}{x^{2m+1}}.
\end{equation}
This will also be used in Section~\ref{sect:GWP1}.
In terms of $z$ and $t$ we have
\begin{multline}
\label{eq:xi0inz}
\xi_0(x)=
\half \left(
1-x\frac{d}{dx}
\right)
\sum_{m=0}^\infty \frac{1}{m+1}
\binom{2m}{m}\frac{1}{x^{2m+1}}
\\
=
\half
\left(1-\frac{z(z^2+1)}{z^2-1}\;\frac{d}{dz}\right)z
=-\frac{z}{z^2-1}
=-\frac{t^2-1}{4t}.
\end{multline}
Note that
\begin{multline*}
-\left(
x_1\frac{d}{dx_1}+x_2\frac{d}{dx_2}
\right)
\Bigg(
\frac{1}{4}\sum_{j,k=1}^\infty\frac{1}{j+k}
\binom{2j}{j}\binom{2k}{k}
\frac{1}{x_1 ^{2j}}\;\frac{1}{x_2^{2k}}
\\
+
\sum_{j,k=0}^\infty\frac{1}{j+k+1}
\binom{2j}{j}\binom{2k}{k}
\frac{1}{x_1 ^{2j+1}}\;\frac{1}{x_2^{2k+1}}
\Bigg)
\\
=
\half \;(x_1\xi_0(x_1)-1)(x_2\xi_0(x_2)-1)
+2{\xi_0(x_1)}
{\xi_0(x_2)}
\\
=
2z_1z_2\frac{1+z_1z_2}{(z_1^1-1)(z_2^2-1)}
\\
=-
\left(
\frac{z_1(z_1^2+1)}{z_1^2-1}
\frac{d}{dz_1}
+
\frac{z_2(z_2^2+1)}{z_2^2-1}
\frac{d}{dz_2}
\right)
\left(-\log(1-z_1z_2)\right).
\end{multline*}
In other words, we have a partial differential 
equation
$$
\left(
x_1\frac{d}{dx_1}+x_2\frac{d}{dx_2}
\right)
\left(
F_{0,2}^D(t_1,t_2)+\log(1-z_1z_2)
\right) = 0
$$
for a holomorphic function in
$x_1$ and $x_2$ defined for 
$|x_1|>\!\!>2$ and $|x_2|>\!\!>2$.
Since the first few terms of the 
Laurent expansions of
$-\log\big(1-z(x_1)z(x_2)\big)$ using
(\ref{eq:z(x)}) agree with the 
first few terms of the 
sums of (\ref{eq:D02inx}), we have the
initial condition for the above 
differential equation. By the uniqueness of the
solution to the Euler differential equation
with the initial condition, we obtain (\ref{eq:F02D}).
Equation (\ref{eq:W02D}) follows
from differentiation of (\ref{eq:F02D}).
\end{proof}

In terms of the $t$-coordinate of (\ref{eq:z(t)}),
the  Galois conjugate of $t\in \Sigma$ 
under the projection $x:\Sigma\lrar \bC$ is $-t$.
Therefore, the recursion kernel for  counting of dessins
is given by
\begin{multline}
\label{eq:Dkernel}
K^D(t,t_1) =
\half\; \frac{\int_t ^{-t} W_{0,2}^D(\cdot,t_1)}
{W_{0,1}^D(-t)-W_{0,1}^D(t)}
= \half \; 
\left(
\frac{1}{t+t_1}+\frac{1}{t-t_1}
\right)
\frac{1}{\frac{t+1}{t-1}-\frac{t-1}{t+1}}
\cdot \frac{1}{dx}\cdot dt_1\\
=
-\frac{1}{64} \; 
\left(
\frac{1}{t+t_1}+\frac{1}{t-t_1}
\right)
\frac{(t^2-1)^3}{t^2}\cdot \frac{1}{dt}\cdot dt_1.
\end{multline}
One of the first two stable cases (\ref{eq:W11}) gives us
\begin{multline}
\label{eq:W11D}
W_{1,1}^D (t_1) =\frac{1}{2\pi i}
\int_{\gam}K^D(t,t_1) 
\left[W_{0,2}^D (t, -t)
+\frac{dx\cdot dx_1}{(x-x_1)^2}
\right]
\\
=
-\frac{1}{2\pi i}
\int_{\gam}K^D(t,t_1) 
\frac{dt\cdot dt}{4 t^2}
=
-\frac{1}{128}\; \frac{(t_1^2-1)^3}{t_1^4}dt_1,
\end{multline}
where the integration contour $\gam$ 
consists of two concentric
circles of a small radius  and 
a large radius
centered around $t=0$, with the inner circle 
positively  and the outer circle  
negatively oriented (Figure~\ref{fig:contourD}).
The $(g,n) = (0,3)$ case is given by 
\begin{multline}
\label{eq:W03D}
W_{0,3}^D(t_1,t_2,t_3)
=\frac{1}{2\pi i} 
\int_{\gam}
\frac{W_{0,2}^D(t,t_1)W_{0,2}^D
(t,t_2)W_{0,2}^D(t,t_3)}
{dx(t)\cdot dy(t)}
\\
=
-\frac{1}{16}
\left[
\frac{1}{2\pi i}
\int_{\gam}
\frac{(t^2-1)^2 (t-1)^2}
{(t+t_1)^2(t+t_2)^2(t+t_3)^2}
\cdot \frac{dt}{t}
\right]
dt_1dt_2dt_3
\\
=
-\frac{1}{16}\left(
1-\frac{1}{t_1^2\;t_2^2\;t_3^2}
\right)
dt_1dt_2dt_3.
\end{multline}

\begin{rem}
The general formula (\ref{eq:EO}) for
$(g,n)=(0,3)$ also gives the 
same answer. This is because
$W_{0,2}^D$ acts as the Cauchy
differentiation kernel. 
\begin{multline*}
W_{0,3}^D(t_1,t_2,t_3) 
=
\frac{1}{2\pi i}
\int_{\gam}K^D(t,t_1) 
\bigg[
W_{0,2}^D(t,t_2)W_{0,2}^D(-t,t_3)
+W_{0,2}^D(t,t_3)W_{0,2}^D(-t,t_2)
\bigg]
\\
=
\frac{1}{64} 
\left[
\frac{1}{2\pi i}
\int_{\gam}
\left(
\frac{1}{t+t_1}+\frac{1}{t-t_1}
\right)
\frac{(t^2-1)^3}{t^2}
\left(
\frac{1}{(t+t_2)^2(t-t_3)^2}
+\frac{1}{(t+t_3)^2(t-t_2)^2}
\right)dt
\right]
\\
\cdot
dt_1dt_2dt_3
\\
=
\Bigg[
-\frac{1}{32}
\frac{(t_1^2-1)^3}{t_1^2}
\left(
\frac{1}{(t_1+t_2)^2(t_1-t_3)^2}
+\frac{1}{(t_1+t_3)^2(t_1-t_2)^2}
\right)
\\
-\frac{1}{16}
\frac{\partial}{\partial t_2}
\left(
\frac{t_2}{t_2^2-t_1^2}\; \frac{(t_2^2-1)^3}{t_2^2}
\;\frac{1}{(t_2+t_3)^2}
\right)
\\
-\frac{1}{16}
\frac{\partial}{\partial t_3}
\left(
\frac{t_3}{t_3^2-t_1^2}\; \frac{(t_3^2-1)^3}{t_3^2}
\;\frac{1}{(t_2+t_3)^2}
\right)
\Bigg]
dt_1dt_2dt_3
=
-\frac{1}{16}\left(
1-\frac{1}{t_1^2\;t_2^2\;t_3^2}
\right)
dt_1dt_2dt_3.
\end{multline*}
\end{rem}

\begin{figure}[htb]
\centerline{\epsfig{file=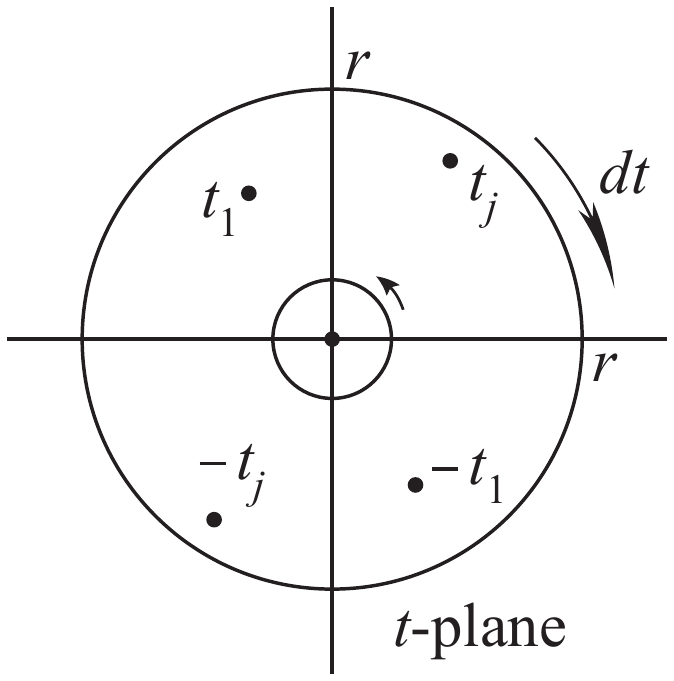, width=1.5in}}
\caption{The integration contour $\gamma$. This contour 
encloses an annulus bounded by two concentric 
circles centered at the origin. The outer one has a 
large radius 
$r>\max_{j\in N} |t_j|$ and the negative orientation,
 and the inner one has an infinitesimally small radius with 
 the positive
 orientation.}
\label{fig:contourD}
\end{figure}

\begin{thm}
\label{thm:LTofD}
Let us define the Laplace transform of the
number of Grothendieck's dessins by
\begin{equation}
\label{eq:LTofD}
F_{g,n}^D(t_1,\dots,t_n)
=\sum_{\mu\in \bZ_{+} ^n}
D_{g,n}(\mu) e^{-(\mu_1 w_1+\cdots+\mu_n w_n)},
\end{equation}
where the coordinate $t_i$ is related to the Laplace
conjugate coordinate $w_j$ by
$$
e^{w_j} = \frac{t_j+1}{t_j-1}+\frac{t_j-1}{t_j+1}.
$$
Then the differential forms
\begin{equation}
\label{eq:WgnD}
W_{g,n}^D (t_1,\dots,t_n) =
d_1\cdots  d_nF_{g,n}^D(t_1,\dots,t_n)
\end{equation}
satisfy the Eynard-Orantin topological recursion
\begin{multline}
\label{eq:DEO}
W_{g,n}^D(t_1,\dots,t_n)
\\
=
-\frac{1}{64} \; 
\frac{1}{2\pi i}\int_\gam
\left(
\frac{1}{t+t_1}+\frac{1}{t-t_1}
\right)
\frac{(t^2-1)^3}{t^2}\cdot \frac{1}{dt}\cdot dt_1
\\
\times
\Bigg[
\sum_{j=2}^n
\bigg(
W_{0,2}^D(t,t_j)W_{g,n-1}(-t,t_2,\dots,\widehat{t_j},
\dots,t_n)
+
W_{0,2}^D(-t,t_j)W_{g,n-1}(t,t_2,\dots,\widehat{t_j},
\dots,t_n)
\bigg)
\\
+
W_{g-1,n+1}^D(t,{-t},t_2,\dots,t_n)
+
\sum^{\text{stable}} _
{\substack{g_1+g_2=g\\I\sqcup J=\{2,3,\dots,n\}}}
W_{g_1,|I|+1}^D(t,t_I) W_{g_2,|J|+1}^D({-t},t_J)
\Bigg].
\end{multline}
The last sum is restricted to the stable 
geometries. In other words, the partition 
should satisfies
$2g_1-1+|I|>0$ and $2g_2-1+|J|$. 
The spectral curve $\Sigma$ 
of the Eynard-Orantin recursion is given by 
$$
\begin{cases}
x = z+\frac{1}{z}\\
y=-z
\end{cases}
$$
with the preferred coordinate $t$ given by 
$$
t = \frac{z+1}{z-1}.
$$
\end{thm}

We give the proof of this theorem in the appendix.

The problem of counting dessins
is  closely related to
the counting problem of the
lattice points of the moduli
space $\cM_{g,n}$ of smooth $n$-pointed
algebraic curves of genus $g$ 
studied in \cite{N1,N2}.
Let us briefly recall
 the combinatorial model for the moduli space
$\cM_{g,n}$ due to Thurston (see for example,
\cite{STT}),
Harer
\cite{Harer},
Mumford \cite{Mumford}, and Strebel
\cite{Strebel},  
following  \cite{MP1998, MP2010}.
For a given ribbon graph $\Gamma$ with $e=e(\Gamma)$
edges, the space of metric
ribbon graphs is $\bR_+ ^{e(\Gamma)}/\Aut (\Gamma)$,
where the
automorphism group acts by permutations of edges
(see \cite[Section~1]{MP1998}).
When we consider \emph{ribbon graph automorphisms},
we restrict ourselves  that
$\Aut (\Gamma)$ fixes each $2$-cell of the cell-decomposition.
We also require that every vertex of a ribbon graph has degree
 $3$ or more.
Using the canonical holomorphic coordinate system on
a topological surface of \cite[Section~4]{MP1998}
corresponding to a metric ribbon graph,
 and the Strebel differentials \cite{Strebel}, we have
an isomorphism of topological orbifolds \cite{Harer,Mumford}
\begin{equation}
\label{eq:M=R}
{\cM}_{g,n}\times \bR_+ ^n \isom R_{g,n}
\end{equation}
for $(g,n)$ in the stable range.
Here
$$
R_{g,n} = \coprod_{\substack{\Gamma {\text{ ribbon graph}}\\
{\text{of type }} (g,n)}}
\frac{\bR_+ ^{e(\Gamma)}}{\Aut (\Gamma)}
$$
is an orbifold consisting of metric ribbon graphs of a given
topological type $(g,n)$.
The gluing of orbi-cells
 is done by making
the length of a non-loop edge tend to $0$. The space
 $R_{g,n}$ is a smooth orbifold
(see \cite[Section~3]{MP1998} and \cite{STT}).
We denote by $\pi:R_{g,n}\longrightarrow
\bR_+ ^n$ the natural projection via (\ref{eq:M=R}), 
which
is the assignment of the  perimeter length
of each boundary to a given metric ribbon graph.

Take a ribbon graph $\Gamma$. Since $\Aut(\Gamma)$
fixes every boundary component of $\Gamma$, 
they are labeled
by $[n]=\{1,2\dots,n\}$. For the moment let us 
give a label to
each edge of $\Gamma$ by an index set 
$[e] = \{1,2,\dots,e\}$.
The edge-face incidence matrix  is defined by
\begin{equation}
\label{eq:incidence}
\begin{aligned}
A_\Gamma &= \big[
a_{i\eta}\big]_{i\in [n],\;\eta\in [e]};\\
a_{i\eta} &= \text{ the number of times edge $\eta$ 
appears in
face $i$}.
\end{aligned}
\end{equation}
Thus $a_{i\eta} = 0, 1,$ or $2$, and the sum of the
entries in each column is
always $2$. The $\Gamma$ contribution of the space
$\pi^{-1}(\mu_1,\dots,\mu_n) = R_{g,n}(\mu)$
 of metric ribbon graphs with
a prescribed perimeter $\mathbf{\mu}=
(\mu_1,\dots,\mu_n)\in \bR_+ ^n$
 is the orbifold
polytope
$$
\frac{\{\mathbf{x}\in \bR_+ ^e\;|\;
A_\Gamma \mathbf{x} = \mathbf{\mu}\}}
{\Aut(\Gamma)},
$$
where $\mathbf{x}=(\ell_1,\dots,\ell_e)$ 
is the collection of
edge lengths of the metric ribbon graph $\Gamma$. 
We have
\begin{equation}
\label{eq:sump}
\sum_{i\in [n]} \mu_i= \sum_{i\in [n]}
\sum_{\eta\in [e]}a_{i\eta}\ell_\eta =
2\sum_{\eta\in [e]}
\ell_\eta.
\end{equation}
Now let $\mu\in\bZ_+ ^n$ be a vector consisting of 
positive integers. The lattice point counting function
we consider is defined by
\begin{equation}
\label{eq:Ngn}
N_{g,n}(\mu) = 
\sum_{\substack{\Gamma {\text{ ribbon graph}}\\
{\text{of type }} (g,n)}}
\frac{\big|\{\bx\in \bZ_+ ^n\;|\;A_\Gam \bx = \mu\}\big|}
{|\Aut(\Gam)|}
\end{equation}
for $(g,n)$ in the stable range (\cite{CMS, MP2010,
N1,N2}).

To find the spectral curve 
for lattice point counting, we  
need to identify the unstable moduli
$\cM_{0,1}$ and
the ribbon graph space $R_{0,1}$.
We recall that 
the orbifold  isomorphism (\ref{eq:M=R}) holds
for  $(g,n)$ in the stable range by 
\emph{defining} $R_{g,n}$
as  the space of metric ribbon graphs
of type $(g,n)$ \emph{without} 
vertices of degrees $1$ and $2$.
For $(g,n) = (0,1)$, there are no ribbon graphs
satisfying these conditions.
Let $v_j$ denote the number of 
degree $j$ vertices in a ribbon graph $\Gam$
of type $(g,n)$. Then we have
$$
\sum_{j\ge 1} jv_j = 2e,\qquad 
\sum_{j\ge 1} v_j = v,
$$
where $v$ is the total number of vertices of $\Gam$.
Hence
\begin{equation}
\label{eq:vertex}
2(2g-2+n) = 2e - 2v = \sum_{j\ge 1} (j-2)v_j
=- v_1 + \sum_{j\ge 3}(j-2)v_j.
\end{equation}
It follows that the number of degree $1$
vertices $v_1$ is positive
when $(g,n)=(0,1)$.
Thus
we conclude that \emph{there is no spectral
curve for this counting problem}.

Still we can consider the Laplace transform of 
the number (\ref{eq:Ngn}) of lattice points of the 
moduli space $\cM_{g,n}$ with a prescribed 
perimeter length. We define for every stable
$(g,n)$
\begin{equation}
\label{eq:FgnL}
F_{g,n}^L(t_1,\dots,t_n) 
= \sum_{\mu\in \bZ_+ ^n} N_{g,n}(\mu) 
\prod_{i=1} ^n \frac{1}
{z_i^{\mu_i}},
\end{equation}
and the Eynard-Orantin differential forms by
\begin{equation}
\label{eq:WgnL}
W_{g,n}^L(t_1,\dots,t_n)
=d_1\cdots  d_nF_{g,n}^L(t_1,\dots,t_n).
\end{equation}
The following result is proved in \cite{CMS},
with inspiration from \cite{N2}.

\begin{thm}[\cite{CMS}]
\label{thm:EOforL}
The differential forms $W_{g,n}^L(t_1,\dots,t_n)$
satisfy the Eynard-Orantin topological recursion
with respect to the same 
spectral curve (\ref{eq:Dspectral}) and 
the recursion kernel (\ref{eq:Dkernel}),
starting with exactly the same first two stable cases
\begin{equation}
\label{eq:W11L}
W_{1,1}^L(t_1) = -\frac{1}{128} 
\frac{(t_1^2-1)^3}{t_1 ^4} dt_1,
\end{equation}
and
\begin{equation}
\label{eq:W03L}
W_{0,3}^L(t_1,t_2,t_3)=-\frac{1}{16}\left(
1-\frac{1}{t_1^2\;t_2^2\;t_3^2}
\right)
dt_1dt_2dt_3.
\end{equation}
\end{thm}

\begin{rem}
It is somewhat surprising, because 
the spectral curve (\ref{eq:Dspectral}) has
nothing to do with the lattice point counting
problem. As we have mentioned, the $(g,n)=(0,1)$
and $(0,2)$ considerations for this problem 
do not produce the spectral curve.
\end{rem}

\begin{cor}
\label{cor:D=N}
For every $(g,n)$ with $2g-2+n>0$, we have 
the identity
\begin{equation}
\label{eq:WgnD=WgnN}
W_{g,n}^D(t_1,\dots,t_n) = W_{g,n}^L(t_1,\dots,t_n).
\end{equation}
The differential form $W_{g,n}^D(t_1,\dots,t_n)$
is a Laurent polynomial in $t_1^2,\dots,t_n^2$
of degree $2(3g-3+n)$,
with a reciprocity property
\begin{equation}
\label{eq:WgnDreciprocity}
W_{g,n}^D(1/t_1,\dots,1/t_n) = 
(-1)^n t_1^2\cdots t_n^2\;W_{g,n}^D(t_1,\dots,t_n).
\end{equation}
The numbers of dessins can be expressed in terms
of the number of lattice
points:
\begin{equation}
\label{eq:DinN}
D_{g,n}(\mu_1,\dots,\mu_n)
=
\sum_{\ell_1>\frac{\mu_1}{2}}
\cdots
\sum_{\ell_n>\frac{\mu_n}{2}}
\prod_{i=1}^n
\frac{2\ell_i-\mu_i}{\mu_i}
\binom{\mu_i}{\ell_i}
N_{g,n}(2\ell_1-\mu_i,\cdots,2\ell_n-\mu_n).
\end{equation}
\end{cor}

\begin{rem}
The relation (\ref{eq:DinN}) appears in 
\cite[Section~2.1]{NS2} for  an abstract setting.
\end{rem}

\begin{proof}
The Eynard-Orantin topological recursion 
uniquely determines the differential forms for
all $(g,n)$. Since $W_{1,1}^D(t)=W_{1,1}^L(t)$
and $W_{0,3}^D(t_1,t_2,t_3) = 
W_{0,3}^L(t_1,t_2,t_3)$, we conclude that
$W_{g,n}^D(t_1,\dots,t_n) =
 W_{g,n}^L(t_1,\dots,t_n)$ for $2g-2+n>0$.

By induction on $2g-2+n$ we can show
that 
$W_{g,n}^D(t_1,\dots,t_n)$ is a Laurent
polynomial in $t_1^2,\dots,t_n^2$. 
The statement is true for the initial cases
(\ref{eq:W11D}) and (\ref{eq:W03D}).
The integral transformation formula
(\ref{eq:DEO}) is a residue calculation
at $t=0$ and $t=\infty$. 
By the induction hypothesis, the right-hand side of
(\ref{eq:DEO}) becomes
\begin{multline*}
-\frac{1}{64} \; 
\frac{1}{2\pi i}\int_\gam
\left(
\frac{1}{t+t_1}+\frac{1}{t-t_1}
\right)
\frac{(t^2-1)^3}{t^2}\cdot \frac{1}{dt}\cdot dt_1
\\
\times
\Bigg[
\sum_{j=2}^n
\bigg(
W_{0,2}^D(t,t_j)W_{g,n-1}(-t,t_2,\dots,\widehat{t_j},
\dots,t_n)
+
W_{0,2}^D(-t,t_j)W_{g,n-1}(t,t_2,\dots,\widehat{t_j},
\dots,t_n)
\bigg)
\\
+
W_{g-1,n+1}^D(t,{-t},t_2,\dots,t_n)
+
\sum^{\text{stable}} _
{\substack{g_1+g_2=g\\I\sqcup J=\{2,3,\dots,n\}}}
W_{g_1,|I|+1}^D(t,t_I) W_{g_2,|J|+1}^D({-t},t_J)
\Bigg]
\\
=
\frac{1}{32} \; 
\frac{1}{2\pi i}\int_\gam
\frac{(t^2-1)^3}{t^2-t_1^2}
\frac{1}{t}\cdot \frac{1}{dt}\cdot dt_1
\Bigg[
\sum_{j=2}^n
\frac{2(t^2+t_j^2)}{(t^2-t_j^2)^2}\;
W_{g,n-1}(t,t_2,\dots,\widehat{t_j},
\dots,t_n)\;dt\cdot dt_j
\\
+
W_{g-1,n+1}^D(t,{t},t_2,\dots,t_n)
+
\sum^{\text{stable}} _
{\substack{g_1+g_2=g\\I\sqcup J=\{2,3,\dots,n\}}}
W_{g_1,|I|+1}^D(t,t_I) W_{g_2,|J|+1}^D({t},t_J)
\Bigg].
\end{multline*}
Clearly the residues at $t=0$ and $t=\infty$ are
Laurent polynomials in $t_1^2,\dots,t_n^2$.

Because of (\ref{eq:WgnD=WgnN}), we have
\begin{equation}
\label{eq:D=N}
\sum_{\mu\in \bZ_+^n}
D_{g,n}(\mu) \prod_{i=1}^n d
\left(\frac{1}{x_i^{\mu_i}}\right)
=
\sum_{\nu\in \bZ_+^n}
N_{g,n}(\nu) \prod_{i=1}^n d
\left(\frac{1}{z_i^{\nu_i}}
\right)
=
(-1)^n
\sum_{\nu\in \bZ_+^n}
N_{g,n}(\nu) \prod_{i=1}^n d
{z_i^{\nu_i}},
\end{equation}
where $x_i=z_i+1/z_i$. 
The Galois conjugation $t\rar -t$ corresponds
to $z\rar 1/z$. Since
$$
W_{g,n}^N(t_1,\dots,t_n)=(-1)^n
W_{g,n}^N(-t_1,\dots,-t_n),
$$
the second equality of (\ref{eq:D=N}) follows.
Take the residue of the left-hand side of
(\ref{eq:D=N}) at $x_i=\infty$ for $i=1,\dots,n$. 
On the right-hand side we take the residue
at $z_i=0$ for every $i$. 
Then for every $(\mu_1,\dots,\mu_n)\in \bZ_+^n$ 
we have
\begin{equation}
\label{eq:Nresidue}
D_{g,n}(\mu_1,\dots,\mu_n) \mu_1\cdots \mu_n
=
\left(\frac{1}{2\pi i}
\right)^n
\int_{|z_1|=\epsilon}\cdots
\int_{|z_n|=\epsilon}
x_1^{\mu_1}\cdots x_n^{\mu_n}
\sum_{\nu\in \bZ_+^n}
N_{g,n}(\nu) \prod_{i=1}^n d
{z_i^{\nu_i}}.
\end{equation}
Since 
$$
\left(z_i+\frac{1}{z_i}\right)^{\mu_i} =
 \sum_{\ell_i=0}^{\mu_i}
\binom{\mu_i}{\ell_i}z_i^{\mu_i-2\ell_i},
$$
the residue of (\ref{eq:Nresidue}) 
comes from the term $\mu_i-2\ell_i+\nu_i=0$,
and we have
\begin{multline*}
D_{g,n}(\mu_1,\dots,\mu_n) \mu_1\cdots \mu_n
\\
=
\sum_{\ell_1>\mu_1/2}\cdots
\sum_{\ell_n>\mu_n/2}
\prod_{i=1}^n
(2\ell_i-\mu_i)\binom{\mu_i}{\ell_i}
N_{g,n}(2\ell_1-\mu_1,\dots,2\ell_n-\mu_n).
\end{multline*}

The reciprocity relation, 
and the degree of the Laurent polynomial,
is the consequence of
the following, which was established in \cite{MP2010}.

\begin{thm}[\cite{MP2010}]
\label{thm:MP2010}
The functions $F_{g,n}^L(t_1,\dots,t_n)$ 
of (\ref{eq:FgnL}) for the stable range 
$2g-2+n>0$
are uniquely determined by the following
differential recursion formula from the
initial values $F_{0,3}^L(t_1,t_2,t_3)$ and
$F_{1,1}^L(t_1)$.
\begin{multline}
\label{eq:FgnL recursion}
F_{g,n}^L(t_1,\dots,t_n)
\\
=
-\frac{1}{16}
\int_{-1} ^{t_1}
\Biggr[
\sum_{j=2} ^n
\frac{t_j}{t^2-t_j^2}
\Bigg(
\frac{(t^2-1)^3}{t^2}\frac{\partial}{\partial t}
F_{g,n-1^L}(t,t_{[n]\setminus\{1,j\}})
-
\frac{(t_j^2-1)^3}{t_j^2}\frac{\partial}{\partial t_j}
F_{g,n-1}^L(t_{[n]\setminus\{1\}})
\Bigg)
\\
+
\sum_{j=2} ^n
\frac{(t^2-1)^2}{t^2}\frac{\partial}{\partial t}
F_{g,n-1}^L(t,t_{[n]\setminus\{1,j\}})
\\
+
\frac{1}{2}\;\frac{(t^2-1)^3}{t^2}
\frac{\partial^2}{\partial u_1\partial u_2}
\Bigg(
F_{g-1,n+1}^L(u_1,u_2,t_{[n]\setminus\{1\}})
\\
+
\sum_{\substack{g_1+g_2=g\\I\sqcup J=[n]
\setminus\{1\}}}
^{\rm{stable}}
F_{g_1,|I|+1}^L(u_1,t_I)F_{g_2,|J|+1}(u_2,t_J)
\Bigg)
\Bigg|_{u_1=u_2=t}\;
\Biggr]\;dt.
\end{multline}
Here $[n]=\{1,2,\dots,n\}$ is an index set,
and the last sum is taken over all partitions 
$g_1+g_2 = g$ and
set partitions $I\sqcup J=[n]\setminus\{1\}$ subject to
the stability conditions $2g_1-1+|I|>0$ and $2g_2-1+|J|>0$.
The initial values are given by
\begin{equation}
\label{eq:F11L}
F_{1,1}^L(t_1) = 
-\frac{1}{384}\frac{(t+1)^4}{t^2}\left(t-4+\frac{1}{t}
\right)
\end{equation}
and 
\begin{equation}
\label{eq:F03L}
F_{0,3}^L(t_1,t_2,t_3)
= -\frac{1}{16}(t_1+1)(t_2+1)(t_3+1)
\left(1+\frac{1}{t_1\;t_2\;t_3}\right).
\end{equation}
In the stable range
 $F_{g,n}^L(t_1,\dots,t_n)$  is a Laurent polynomial of 
degree $3(2n-2+n)$ and satisfies the reciprocity 
relation
\begin{equation}
\label{eq:FgnLreciprocity}
F_{g,n}^L(1/t_1,\dots,1/t_n)
= F_{g,n}^L(t_1,\dots,t_n).
\end{equation}
The leading terms of $F_{g,n}^L(t_1,\dots,t_n)$ 
form a homogeneous polynomial of degree
$3(2g-2+n)$, and is given by
\begin{equation}
\label{eq:FgnK}
{F}_{g,n}^{K}(t_1,\dots,t_n)
\overset{\text{def}}{=}
\frac{(-1)^n}{2^{2g-2+n}}
\sum_{\substack{d_1+\cdots+d_n\\=3g-3+n}}
\la \tau_{d_1}\cdots\tau_{d_n}\ra_{g,n}
\prod_{j=1}^n (2d_j-1)!! \left(
\frac{t_j}{2}
\right)^{2d_j+1},
\end{equation}
where 
$$
\la \tau_{d_1}\cdots\tau_{d_n}\ra_{g,n}
=\int_{\Mbar_{g,n}} \psi_1^{d_1}\cdots
\psi_n^{d_n}
$$
is the $\psi$-class intersection number
(see Section~\ref{sect:Hurwitz} for more
detail about intersection numbers).
The special value at $t_i=1$ gives
\begin{equation}
\label{eq:Euler}
F_{g,n}^L(1,1,\dots,1) = (-1)^n \rchi(\cM_{g,n}).
\end{equation}
\end{thm}

This completes the proof of Corollary~\ref{cor:D=N}.
\end{proof}

\section{The $\psi$-class interaction numbers
on $\Mbar_{g,n}$}
\label{sect:kont}

The crucial discovery of Konstevich
\cite {K1992} is the equality between the
intersection numbers on the compact moduli
space $\Mbar_{g,n}$ and the Euclidean 
volume of the moduli space $\cM_{g,n}$ of
smooth curves using the isomorphism 
(\ref{eq:M=R}). The Feynman diagram
expansion of the Kontsevich matrix integral
relates the Euclidean volume with a $\tau$-function 
of the KdV equations. The Eyanrd-Orantin
recursion for the $\psi$-class
intersection numbers  is precisely
the Dijkgraaf-Verlinde-Verlinde formula
\cite{DVV} of the intersection numbers.
In this section we identify the spectral curve
and the recursion kernel for the $\psi$-class
intersection numbers.

As we have noted, the derivative of the
recursion formula
(\ref{eq:FgnL recursion}) is \emph{not} 
the Eynard-Orantin recursion because the 
spectral curve is not defined by the
unstable geometries. Indeed, we have
$dF_{0,1}^L\equiv 0$. However,
when we associate the number of lattice 
points with the $\psi$-class intersection 
numbers on $\Mbar_{g,n}$, the unstable
geometries do make sense.

Let us recall a computation in \cite[Section~4]{MP2010}.
\begin{multline}
\label{eq:FgnLinA}
\sum_{\mu\in\bZ_+^n}
N_{g,n}(\mu) e^{-\la \mu,w\ra}
=
\sum_{\substack{\Gamma {\text{ ribbon graph}}\\
{\text{of type }} (g,n)}}
\sum_{\mu\in\bZ_+^n}
\frac{1}{|\Aut(\Gamma)|}
\big|\{\bx\in \bZ_+ ^{e(\Gamma)}\;|\;A_\Gamma \bx = \mu\}
\big|
 e^{-\la \mu,w\ra}
 \\
 =
 \sum_{\substack{\Gamma {\text{ ribbon graph}}\\
{\text{of type }} (g,n)}}
\frac{1}{|\Aut(\Gamma)|}
\sum_{\bx\in\bZ_+^{e(\Gamma)}}
e^{-\la A_\Gamma \bx,w\ra}
\\
 =
 \sum_{\substack{\Gamma {\text{ ribbon graph}}\\
{\text{of type }} (g,n)}}
\frac{1}{|\Aut(\Gamma)|}
\prod_{\substack{\eta \text{ edge}\\
\text{of }\Gamma}}
\sum_{\ell_\eta=1}^\infty
e^{-\la a_\eta,w\ra\ell_\eta}
\\
 =
 \sum_{\substack{\Gamma {\text{ ribbon graph}}\\
{\text{of type }} (g,n)}}
\frac{1}{|\Aut(\Gamma)|}
\prod_{\substack{\eta \text{ edge}\\
\text{of }\Gamma}}
\frac{
e^{-\la a_\eta,w\ra}}
{1-e^{-\la a_\eta,w\ra}},
\end{multline}
where $A_\Gam$ is the incidence matrix
of (\ref{eq:incidence}),  $a_\eta$ is  the
$\eta$-th column of $A_\Gam$, and 
$\la \mu,w\ra = \mu_1w_1+\dots+\mu_nw_n$.
By comparing (\ref{eq:FgnL}) and 
(\ref{eq:FgnLinA}), we see that we are substituting
$e^{w_i} = z_i$ in this computation. 
Therefore, we obtain
\begin{equation}
\label{eq:FgnLinz}
F_{g,n}^L(t_1,\dots,t_n)
=
 \sum_{\substack{\Gamma {\text{ ribbon graph}}\\
{\text{of type }} (g,n)}}
\frac{1}{|\Aut(\Gamma)|}
\prod_{\substack{\eta \text{ edge}\\
\text{of }\Gamma}}
\frac{1}
{\prod_{i=1}^n z_i^{a_{i\eta}}-1}.
\end{equation}
Thus the series (\ref{eq:FgnL}) in $z_i$ converges
for $|z_i|>1$. Since $z_i=\frac{t_i+1}{t_1-1}$,
the $t_i\rar \infty$ limit picks up the limit of
(\ref{eq:FgnL}) as $z_i\rar 1$, and hence the
information of $N_{g,n}(\mu)$ as 
$\mu_i\rar \infty$.
Since the orbifold isomorphism (\ref{eq:M=R})
is scale invariant under the action of $\bR_+$,
making the perimeter length $\mu$ large is 
the same as making the mesh small in the
lattice point counting. Hence at the limit we
obtain the Euclidean volume of $\cM_{g,n}$
considered by Kontsevich in \cite{K1992}.
This is why we expect that (\ref{eq:FgnK})  holds.
Let us now consider the limit of the spectral 
curve (\ref{eq:Dspectral}) as $t\rar\infty$.
First we have
\begin{equation*}
\begin{aligned}
x &= z+\frac{1}{z} = 2 +\frac{4}{t^2-1}
\\
y&=-z = -1-\frac{2}{t-1}.
\end{aligned}
\end{equation*}
Ignoring the constant shifts of $x$ and $y$, we obtain
for a large $t$
\begin{equation}
\label{eq:WKspectral}
\begin{cases}
x=\frac{4}{t^2}\\
y=-\frac{2}{t}.
\end{cases}
\end{equation}
Hence the spectral curve is given by
the equation $x=y^2$. We use $t$ as the
preferred coordinate.

We now compare the Eynard-Orantin
recursion with respect to this spectral curve
and the Witten-Kontsevich theory. 
We use (\ref{eq:FgnK}) and define
  \begin{equation}
 \label{eq:WgnK}
 \begin{aligned}
 {W}_{g,n}^{K}(t_1,\dots,t_n)
 &=
 d_1\cdots d_n {F}_{g,n}^{K}(t_1,\dots,t_n)
 \\
 &=
\frac{(-1)^n}{2^{2g-2+n}}
\sum_{\substack{d_1+\cdots+d_n\\=3g-3+n}}
\la \tau_{d_1}\cdots\tau_{d_n}\ra_{g,n}
\prod_{j=1}^n (2d_j+1)!! 
\left(\frac{t_j}{2}\right)^{2d_j} d\left(\frac{t_j}{2}\right)
\\
&=
\frac{(-1)^n}{16^{2g-2+n}}\;
w_{g,n}^K(t_1,\dots,t_n)\;dt_1\cdots dt_n,
\end{aligned}
 \end{equation}
 where $w_{g,n}^K(t_1,\dots,t_n)$ is the coefficient
 of the Eynard-Orantin differential form
 normalized by the constant factor 
 $\frac{(-1)^n}{16^{2g-2+n}}$. Note that  
 $w_{g,n}^K(t_1,\dots,t_n)$ is 
 a polynomial in $t_i^2$'s
 with positive rational coefficients
 for $(g,n)$ in the stable range.
For $(g,n)=(0,1)$ and $(0,2)$, we have 
\begin{align}
\label{eq:tau01}
\la \tau_k\ra_{0,1} &= \delta_{k+2,0}
\\
\label{eq:tau02}
\la\tau_{k_1}\tau_{k_2}\ra_{0,2} &= (-1)^{k_1},
\qquad k_1+k_2 = -1.
\end{align}
Therefore,
\begin{equation}
\label{eq:W01K}
W_{0,1}^K(t) = \frac{-1}{16^{-1}}\la \tau_{-2}\ra
(-3)!! \;t^{-4} dt=
\frac{16}{t^4}\; dt = ydx,
\end{equation}
in agreement with the  spectral curve $x=y^2$
  (\ref{eq:WKspectral}).
  Similarly, we have
  \begin{multline}
  \label{eq:F02K}
  F_{0,2}^K(t_1,t_2)
  =
  \sum_{d=0}^\infty (-1)^d (2d-1)!! (-2d-3)!!
  \left(\frac{t_1}{2}\right)^{2d+1}
   \left(\frac{t_2}{2}\right)^{-2d-1}
  \\
  =
  -
 \sum_{d=0}^\infty \frac{1}{2d+1}
 \left(\frac{t_1}{t_2}\right)^{2d+1}
 =
  \log\left(1-\frac{t_1}{t_2}\right)
 -\half  \log\left(1-\frac{t_1^2}{t_2^2}\right),
  \end{multline}
and hence
\begin{equation}
\label{eq:W02K}
W_{0,2}^K(t_1,t_2)
=\frac{dt_1\cdot dt_2}{(t_1-t_2)^2}
-\half\;\frac{dx_1\cdot dx_2}{(x_1-x_2)^2}.
\end{equation}
As a consequence, the recursion kernel is
given by
\begin{equation}
\label{eq:Kkernel}
K^K(t,t_1)
=
-\half
\left(
\frac{1}{t+t_1}+\frac{1}{t-t_1}
\right)
\frac{t^4}{32}\;\frac{1}{dt}\;dt_1,
\end{equation}
since $\frac{dx_1\cdot dx_2}{(x_1-x_2)^2}$
does not contribute to the kernel.
The Eynard-Orantin recursion for the
Euclidean volume then becomes
\begin{multline}
\label{eq:KEO}
W_{g,n}^K(t_1,\dots,t_n)
\\
=
-\frac{1}{2\pi i}\int_{\gam_\infty}
\left(
\frac{1}{t+t_1}+\frac{1}{t-t_1}
\right)
\frac{t^4}{64}\;\frac{1}{dt}\;dt_1
\Bigg[
W_{g-1,n+1}^K(t,-t,t_2, \dots, t_n)
\\
+
\sum_{j=2}^n
\Bigg(
\frac{dt\cdot dt_j}{(t-t_j)^2}\;
W_{g,n-1}^K(-t,t_2,\dots,\widehat{t_j},\dots,t_n)
\\
-
\frac{dt\cdot dt_j}{(t+t_j)^2}\;
W_{g,n-1}^K(t,t_2,\dots,\widehat{t_j},\dots,t_n)
\Bigg)
\\
+
\sum_{\substack{g_1+g_2=g\\
I\sqcup J=\{2,\dots,n\}}} ^{\text{stable}}
W_{g_1,|I|+1}^K(t,t_I)W_{g_2,|J|+1}^K(-t,t_J)
\Bigg],
\end{multline}
where the integral is taken with respect to 
a large negatively oriented
 circle  $\gam_\infty$ that 
encloses any of $\pm t_1,\dots,\pm t_n$.
This is the larger circle of Figure~\ref{fig:contourD}.
Here again $\frac{dx_1\cdot dx_2}{(x_1-x_2)^2}$
does not contribute in the formula.
Since the coefficients
$w_{g,n}^K(t_1,\dots,t_n)$ in the stable
range are polynomials,
 the   poles
of the integrand of
(\ref{eq:KEO}) in the integration coutour
are at $t=\pm t_i$'s. 
Therefore, we can  perform the 
integral in terms of the residue
calculus at poles  $t=\pm t_i$.
First let us get rid of the factor $1/16^{2g-2+n}$
from (\ref{eq:KEO}).
Since the recursion is
an induction on $2g-2+n$, we have an overall 
factor $16$ adjustment on the right-hand side.
The integration contour is negatively 
oriented, so the residue calculation 
at $t=\pm t_i$ receives universally 
the negative sign. This sign  is
exactly cancelled by the choice
of the sign of $w_{g,n}^K$  in (\ref{eq:WgnK}).
Thus the result of residue evaluation 
of (\ref{eq:KEO}) is
\begin{multline}
\label{eq:KEOinw}
w_{g,n}^K(t_1,\dots,t_n)
=
\half\;t_1^4
w_{g-1,n+1}^K(t_1,t_1,t_2, \dots, t_n)
\\
+
\half\;t_1^4
\sum_{\substack{g_1+g_2=g\\
I\sqcup J=\{2,\dots,n\}}} ^{\text{stable}}
w_{g_1,|I|+1}^K(t_1,t_I)w_{g_2,|J|+1}^K(t_1,t_J)
\\
+
{t_1^4}
\sum_{j=2}^n
\frac{t_1^2+t_j^2}{(t_1^2-t_j^2)^2}\;
w_{g,n-1}^K(t_1,\dots,\widehat{t_j},\dots,t_n)
\\
+
\half
\sum_{j=2}^n
\left(
\left.
\frac{\partial}{\partial t}
\right|_{t=t_j}+
\left.
\frac{\partial}{\partial t}
\right|_{t=-t_j}
\right)
\left(
\frac{1}{t^2-t_1^2}\; t^5
w_{g,n-1}^K(t,t_2,\dots,\widehat{t_j},\dots,t_n)
\right)
\\
=
\half\;t_1^4
\left[
w_{g-1,n+1}^K(t_1,t_1,t_2, \dots, t_n)
+
\sum_{\substack{g_1+g_2=g\\
I\sqcup J=\{2,\dots,n\}}} ^{\text{stable}}
w_{g_1,|I|+1}^K(t_1,t_I)w_{g_2,|J|+1}^K(t_1,t_J)
\right]
\\
+
\sum_{j=2}^n
\frac{\partial}{\partial t_j}
\left[
\frac{t_j}{t_1^2-t_j^2}
\bigg(
t_1^4 w_{g,n-1}^K(t_{[n]\setminus \{j\}})
-t_j^4 w_{g,n-1}^K(t_{[n]\setminus \{1\}})
\bigg)
\right].
\end{multline}
This is the same as \cite[Theorem~5.2]{CMS}.

 Let us adopt the 
normalized notation
\begin{equation}
\label{eq:sigma}
\la \sigma_{d_1}\cdots\sigma_{d_n}\ra_{g,n}
=\la \tau_{d_1}\cdots\tau_{d_n}\ra_{g,n}
\prod_{i=1}^n (2d_i+1)!!
\end{equation}
to make the formula shorter. 
Then 
\begin{equation}
\label{eq:wgnK}
w_{g,n}^K(t_1,\dots,t_n)
=
\sum_{d_1,\dots,d_n}
\la \sigma_{d_1}\cdots\sigma_{d_n}\ra_{g,n}
\prod_{j=1}^nt_j^{2d_j}.
\end{equation}
The DVV formula \cite{DVV}
for the
Virasoro constraint condition on the $\psi$-class
intersection numbers on $\Mbar_{g,n}$  reads 
\begin{multline}
\label{eq:DVV}
 \la \sigma_{k}\prod_{i=2}^n \sigma_{d_i}\ra _{g,n}
=
\frac{1}{2} \sum_{a+b=k-2}
\la \sigma_a\sigma_b\prod_{i=2}^n \sigma_{d_i}
\ra _{g-1,n+1}
\\
+
\frac{1}{2} \sum_{a+b=k-2}
\sum_{\substack{g_1+g_2=g\\
I\sqcup J= \{2,\dots,n\}}} ^{
\text{stable}}
\la \sigma_a\prod_{i\in I}\sigma_{d_i}\ra _{g_1,|I|+1}\cdot
\la \sigma_b\prod_{j\in J}\sigma_{d_j}\ra _{g_2,|J|+1}
\\
+
\sum_{j=2}^n (2d_j+1)  
 \la \sigma_{k+d_j-1}\prod_{i\ne1, j}
 \sigma_{d_i}
\ra _{g,n-1}.
\end{multline}
We thus recover the discovery of \cite{EO1}:

\begin{thm}
\label{thm:EO=WK}
The Eynard-Orantin recursion formula for the
spectral curve $x=y^2$ is the Dijkgraaf-Verlinde-Verlinde
formula \cite{DVV} for the intersection numbers
$\la \tau_{d_1}\cdots\tau_{d_n}\ra_{g,n}$
on the moduli space $\Mbar_{g,n}$ of 
pointed stable curves. 
\end{thm}

\begin{proof}
We  extract the coefficient of 
\begin{equation}
\label{eq:kterm}
t_1^{2k}\prod_{j=2}^n
t_j^{2d_j}
\end{equation}
 in (\ref{eq:KEOinw}) and compare the result 
 with 
 (\ref{eq:DVV}). 
It is obvious that the fifth line
of (\ref{eq:KEOinw}) produces the first and second
lines of
(\ref{eq:DVV}).

To compare the last
lines of (\ref{eq:KEOinw}) and 
 (\ref{eq:DVV}),
we consider the case $|t_j|<|t_1|$ for all
$j\ge 2$ in (\ref{eq:KEOinw}). 
We then have the expansion
$$
\frac{1}{t_1^2-t_j^2}
=
\frac{1}{t_1^2}\;
\frac{1}{1-\frac{t_j^2}{t_1^2}}
=
\frac{1}{t_1^2}\sum_{m=0} ^\infty 
\left(
\frac{t_j^2}{t_1^2}\right)^m.
$$
The (\ref{eq:kterm})-term of the last line
of (\ref{eq:KEOinw}) has two contributions. 
The first one comes from
$$
\frac{\partial}{\partial t_j}\left(
{t_1^2}t_j\sum_{m=0} ^\infty 
\left(
\frac{t_j^2}{t_1^2}\right)^m
w_{g,n-1}^K(t_1,t_2,\dots,\widehat{t_j},\dots,t_n)
\right).
$$
Since 
$w_{g,n-1}^K(t_1,t_2,\dots,\widehat{t_j},\dots,t_n)$
does not contain $t_j$,
we set $m=d_j$
to produce the right power $2d_j$ of $t_j$.
The power of $t_1$ has to be
$2k$. Thus from $w_{g,n-1}^K$
we take the term of $t_1^{2k+2d_j-2}$,
whose coefficient is
$\la \sigma_{k+d_j-1}\prod_{i\ne 1, j}
\sigma_{d_i}\ra$.
The total contribution from the first kind comes from
the differentiation, which gives
$2m+1=2d_j+1$.

The second possible contribution for
the (\ref{eq:kterm})-term may come from 
$$
-\frac{\partial}{\partial t_j}\left(
\frac{t_j^5}{t_1^2}\sum_{m=0} ^\infty 
\left(
\frac{t_j^2}{t_1^2}\right)^m
w_{g,n-1}^K(t_2,\dots,t_n)
\right).
$$
However, this term does not 
produce $t_1^{2k}$, and hence does not
contribute to the (\ref{eq:kterm})-term.
This completes the proof of Theorem~\ref{thm:EO=WK}.
\end{proof}

\section{Single Hurwitz numbers}
\label{sect:Hurwitz}

\emph{What is the mirror dual of the 
number of trees?} The answer we wish to 
present in this section is that it is the 
\emph{Lambert curve}.
This analytic curve serves as the 
spectral curve for the Hurwitz counting
problem, and comes up from the  
the unstable  geometries
 $(g,n)=(0,1)$ and $(0,2)$ via Laplace transform.

A \emph{Hurwitz} cover is a holomorphic mapping 
$f:C\rightarrow \mathbb{P}^1$ from 
a connected nonsingular projective 
algebraic curve $C$ of genus $g$ to 
the projective line $\mathbb{P}^1$ with only simple 
ramifications except for $\infty \in\mathbb{P}^1$. 
Such a cover is further refined by specifying its
\emph{profile}, which is a partition 
$\mu = (\mu_1\ge \mu_2 \ge
\cdots \ge \mu_n >0)$ of  the degree of the covering 
$d = |\mu| =
\mu_1 + \cdots + \mu_n$.  The length  $\ell(\mu) = n$
 of this partition
is the number of points in the inverse image  
$f^{-1}(\infty) 
= \{p_1,\dots,p_n\}$ of 
$\infty$. 
Each part $\mu_i$ gives a local 
description of the map $f$,
which is given by $u\longmapsto u^{-\mu_i}$ 
in terms of a local
coordinate $u$ of $C$ around $p_i$. 
The number 
$h_{g,\mu}$ of the topological
types of Hurwitz covers of a given genus $g$ and 
a profile 
$\mu$, counted with the weight factor $1/|\Aut f|$, 
is the \emph{single Hurwitz
number} we shall deal with in this section.

The deformations of a Hurwitz cover $f$ are
obtained by moving the branch points 
(i.e., the image of the ramification points)
on $\bP^1\setminus \{\infty\}$. Thus $h_{g,\mu}$ counts
the number of Hurwitz covers with prescribed
(i.e., fixed) and \emph{labeled} branch points. 
On the other hand, the preimages of
$\infty$ on $C$ are labeled only by 
the parts of $\mu$. Therefore, 
a more natural count of Hurwitz cover is
\begin{equation}
\label{eq:Hg(mu)}
H_g(\mu)= 
\frac{|\Aut(\mu)|}{(2g-2+n+
|\mu|)!}\cdot h_{g,\mu}.
  \end{equation}
Here, 
\begin{equation}
\label{eq:r}
r=r(g,\mu)\overset{\text{def}}{=}2g-2+n+|\mu| 
\end{equation}
is the number of simple ramification points
of $f$ by the Riemann-Hurwitz formula,
and 
$\Aut(\mu)$ is the group of permutations
of equal parts of the partition $\mu$.

One reason that explains why single 
Hurwitz numbers are interesting is 
a remarkable formula due to 
Ekedahl, Lando, Shapiro and Vainshtein
\cite{ELSV, GV1, Liu, OP1} that
relates Hurwitz numbers and Gromov-Witten
invariants.  For genus
$g\ge 0$ and a partition $\mu$ 
of length $\ell(\mu)=n$ subject to the 
stability condition
 $2g - 2 +n >0$, the ELSV formula states that
\begin{multline}
\label{eq:ELSV}
H_g(\mu)=
	 \prod_{i=1}^{n}
	 \frac{\mu_i ^{\mu_i}}{\mu_i!}
	 \int_{\overline{\mathcal{M}}_{g,n}} 
	 \frac{\Lambda_g^{\vee}(1)}
	 {\prod_{i=1}^{n}\big( 1-\mu_i \psi_i\big)}
	 \\
	 =
	 \sum_{j=0}^g (-1)^j
	 \sum_{k_1,\dots,k_n\ge 0}
	 \la \tau_{k_1}\cdots \tau_{k_{n}}c_{j}(\bE)\ra 
	   \prod_{i=1}^{n}
	 \frac{\mu_i ^{\mu_i+k_i}}{\mu_i!},
\end{multline}
where $\overline{\mathcal{M}}_{g,n}$ is 
the Deligne-Mumford moduli stack of stable algebraic
curves of genus $g$ with $n$ distinct smooth
marked points, 
$\Lambda_g^{\vee}(1)=1-c_{1}(\bE)+\cdots +(-1)^g
c_{g}(\bE)$ is the alternating sum of the
Chern classes of the Hodge bundle $\bE$  
 on $\overline{\mathcal{M}}_{g,n}$, 
$\psi_i$ is the $i$-th tautological cotangent class, 
and 
\begin{equation}
\label{eq:LHI}
\la \tau_{k_1}\cdots \tau_{k_n}c_{j}(\bE)\ra= \int_{\overline{\mathcal{M}}_{g,n}}\psi_1^{k_1}\cdots\psi_\ell ^{k_n}c_j(\bE)
\end{equation}
is the linear Hodge integral,
which is $0$ unless $k_1+\cdots+k_n +j=3g-3+n$.

The Deligne-Mumford stack 
$\Mbar_{g,n}$
is defined as the moduli space of \emph{stable} curves satisfying the
stability condition
$2-2g-n <0$.  However, single Hurwitz numbers 
are well defined for \emph{unstable} geometries
 $(g,n) = (0,1)$ and $(0,2)$, and their values are
\begin{equation}
\label{eq:unstableHurwitz}
H_0((d)) = \frac{d^{d-3}}{(d-1)!}
=\frac{d^{d-2}}{d!}
\qquad \text{and}\qquad
H_0((\mu_1,\mu_2)) =
 \frac{1}{\mu_1+\mu_2}\cdot
\frac{\mu_1^{\mu_1}}{\mu_1!}\cdot \frac{\mu_2^{\mu_2}}{\mu_2!}.
\end{equation}
The ELSV formula remains valid for unstable cases
by \emph{defining}
\begin{align}
\label{eq:01Hodge}
&\int_{\overline{\cM}_{0,1}} \frac{\Lambda_0 ^\vee (1)}{1-d\psi}
=\frac{1}{d^2},\\
\label{eq:02Hodge}
&\int_{\overline{\cM}_{0,2}} 
\frac{\Lambda_0 ^\vee (1)}{(1-\mu_1\psi_1)(1-\mu_2\psi_2)}
=\frac{1}{\mu_1+\mu_2}.
\end{align}

Let us examine the $(g,n)=(0,1)$ case. 
We wish to count the number of Hurwitz covers
$f:\bP^1\lrar\bP^1$ of degree $d$
with profile $\mu = (d)$. If $d=2$, then 
$f(u) = u^2$ is the only map, since $r=1$ and 
the two ramification points can be placed at
$u=0$ and $u=\infty$. The automorphism of 
this map is $\bZ/2\bZ$.
We now consider the case when $d\ge 3$.
First we label all 
branch points. One is $\infty$, so let us
place all others,
the images of simple ramification points,
at the $r$-th roots of unity. Here $r=d-1$. 
We label these points with indices 
$[r]=\{1,2,\dots,r\}$.
Connect each
$r$-th root of unity with the origin by a 
straight line (see Figure~\ref{fig:H01}).
Let $*$ denote this star-like shape, which 
has one vertex at the center and $r$ half-edges.
Then the inverse image $f^{-1}(*)$
is a tree-like shape with $d$ vertices and
$rd$ half-edges. Here we call 
each inverse image of $0$ a vertex
of $f^{-1}(*)$. If $f$ is simply ramified at
$p$, then two half-edges are connected
at $p$ and form a real edge that is incident
to two vertices. Since $f(p)$ is one
of the $r$-th root of unity, we give the 
same label to $p$. Thus all simple 
ramification points are labeled with 
the index set $[r]$. 
Now we remove all
half-edges from $f^{-1}(*)$
that are not made into an edge, and denote it by $T$. 
It is  a tree on $\bP^1$ that has
$d$ vertices and $r=d-1$ edges.
Note that except for the case $d=2$, 
the edge labeling gives a labeling of vertices.
For example, if a vertex $x$ is incident to 
edges $i_1<i_2<\dots<i_k$, then $x$ is
labeled by $i_1i_2\cdots i_k$.

\begin{figure}[htb]
\centerline{\epsfig{file=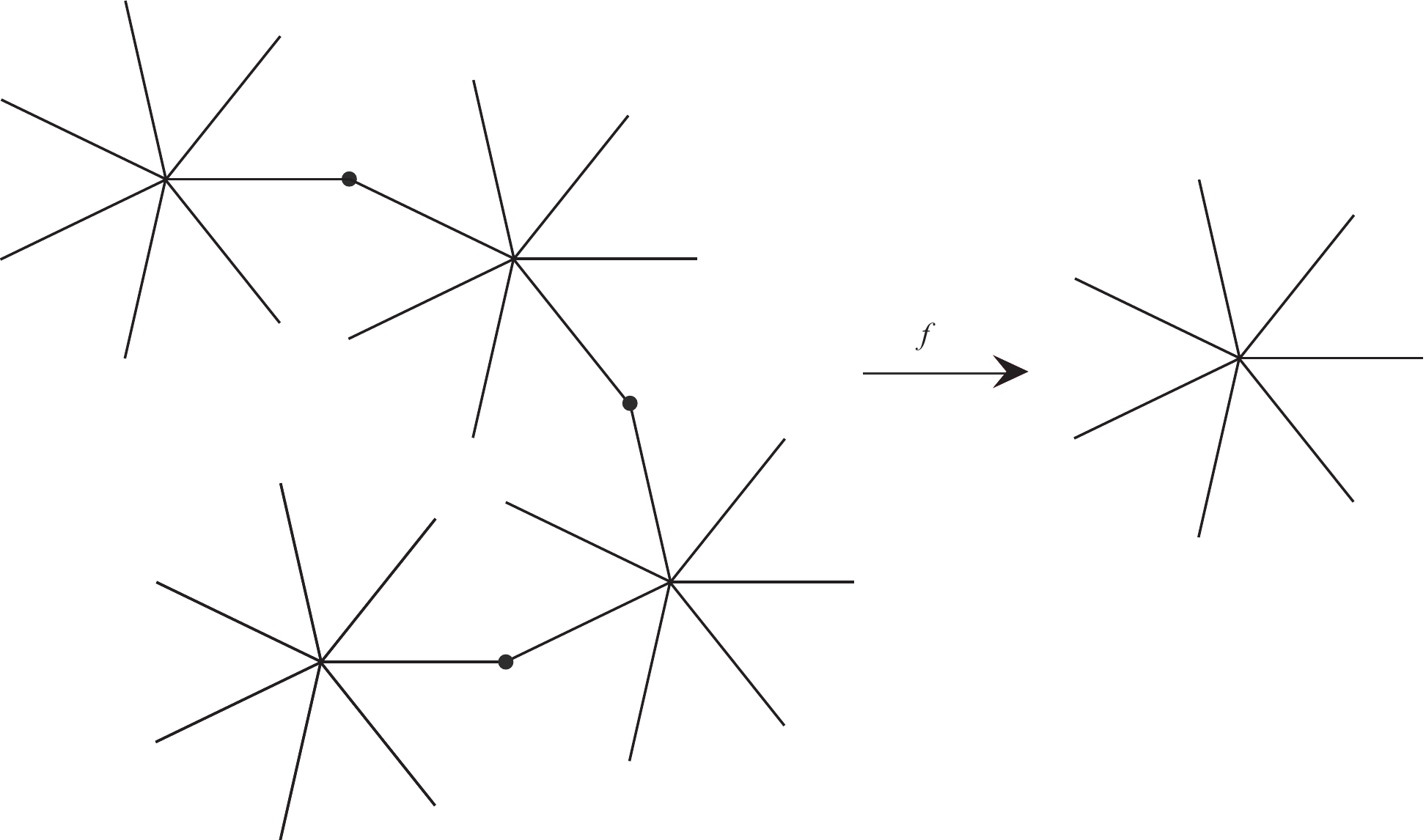, width=4in}}
\caption{Counting the genus $0$ single
Hurwitz numbers with the total ramification
at $\infty$.}
\label{fig:H01}
\end{figure}

Conversely, suppose we are given a tree
 with $d$ labeled vertices 
 by the index set $[d]=\{1,2,\dots,d\}$ 
 and $r=d-1$ edges.
At each vertex we can give a cyclic 
order to incident edges by aligning them
in the increasing order
of the labels of the other ends
of the edges. Thus the tree becomes
a \emph{ribbon graph}
(see Section~\ref{sect:dessin}), and hence it
can be placed on $\bP^1$.
Then 
by choosing the midpoint of each edge
as a simple ramification point and each vertex
as a zero of $f$, we can construct a Hurwitz cover.
Recall that the number of trees with $d$ 
labeled vertices is $d^{d-2}$. 
Therefore, 
$$
H_0((d)) = \frac{d^{d-2}}{d!}
$$
is the number of trees with 
$d$ \emph{unlabeled}
 vertices.

Fix an $n\ge 1$, and consider a partition
$\mu$ of length $n$ as an $n$-dimensional vector
$$
\mu = (\mu_1,\dots,\mu_n)\in \bZ_+^n
$$
consisting of positive integers. 
The Laplace transform of 
$H_g(\mu)$ as a function in $\mu$, 
\begin{equation}
\label{eq:Hgn}
H_{g,n}(w_1,\dots,w_n)
=
\sum_{\mu\in\bZ_+^n} H_g(\mu)
e^{-\left(
\mu_1(w_1+1)+\cdots+\mu_n(w_n+1)
\right)},
\end{equation}
is the function we wish to compute. 
Note that the automorphism group 
$\Aut (\mu)$ acts trivially on the function
$e^{-\left(
\mu_1(w_1+1)+\cdots+\mu_n(w_n+1)
\right)}$, which explains its appearance in 
(\ref{eq:Hg(mu)}).
The reason for  shifting the variables
$w_i\longmapsto w_i+1$ is due to the 
asymptotic behavior
$$
\frac{\mu ^{\mu+k}}{\mu!} e^{-\mu}\sim
\frac{1}{\sqrt{2\pi}}\;\mu^{k-\half}
$$
as $\mu$ approaches to $\infty$. This asymptotics
also suggests that the holomorphic
function $H_{g,n}(w_1,\dots,w_n)$ is actually 
defined on a double-sheeted overing on the
$w_i$-plane, since
$\sqrt{w_i}$  behaves better as a 
holomorphic coordinate.

Following \cite{EMS, MZ},
we  introduce a series of polynomials $\hxi_n(t)$ of
degree $2n+1$ in $t$ for $n\ge 0$ by the
recursion formula
\begin{equation}
\label{eq:xihatn}
\hxi_{n}(t) = t^2(t-1)\frac{d}{dt}\hat{\xi}_{n-1}(t)
\end{equation}
with the initial condition
$\hxi_{0}(t) =t-1$. 
This differential operator appears in \cite{GJV2}.
The functions $\hxi_{-1}(t)$ and $\hxi_0(t)$
also appear as
 the two fundamental functions in \cite{Zvonkine}
 that generate his algebra $\cA$.
These polynomials are introduced to 
make the computation 
 of the Laplace transform (\ref{eq:Hgn}) easier.

\begin{prop}[\cite{CGHJK, EMS}]
\label{prop:twxy}
Let us introduce new coordinates
\begin{equation}
\label{eq:xyt}
x=e^{-w},\qquad 
z=\sum_{\mu=1}^\infty
\frac{\mu^{\mu-1}}{\mu!}\;e^{-\mu}\;x^\mu,\qquad
t-1=\sum_{\mu=1}^\infty
\frac{\mu^{\mu}}{\mu!}\;e^{-\mu}\;x^\mu.
\end{equation}
Then the inverse function of $z=z(x)$ is given by
\begin{equation}
\label{eq:Lambert}
x=z e^{1-z},
\end{equation}
and the variables $z$ and $t$ are related by
\begin{equation}
\label{eq:z in t for H}
z=\frac{t-1}{t}.
\end{equation}
Moreover, we have 
\begin{equation}
\label{eq:xihatninx}
\hxi_n(t) = 
\sum_{\mu=1}^\infty 
\frac{\mu^{\mu+n}}{\mu!}\;e^{-\mu(w+1)}
=
\sum_{\mu=1}^\infty 
\frac{\mu^{\mu+n}}{\mu!}\;e^{-\mu}\;x^\mu
\end{equation}
for $n\ge 0$.
\end{prop}

\begin{proof}
The infinite series (\ref{eq:xihatninx}) has the
radius of convergence $1$, and for $|x|<1$, 
we can apply the Lagrange inversion formula
to obtain  (\ref{eq:Lambert}). 
Since the application of
$$
-\frac{d}{dw} = x\;\frac{d}{dx} = 
t^2(t-1)\frac{d}{dt}
$$
$n$-times to 
$
\sum_{\mu=1}^\infty 
\frac{\mu^{\mu}}{\mu!}\;e^{-\mu(w+1)}
$
produces 
$
\sum_{\mu=1}^\infty 
\frac{\mu^{\mu+n}}{\mu!}\;e^{-\mu(w+1)}
$, we obtain (\ref{eq:xihatn}). 
If we extend (\ref{eq:xihatninx}) formally to $n=-1$,
then we have $z=\hxi_{-1}(t)$.
To obtain the expression of $z$ as a function of $t$, we
need to solve the differential equation 
$$
t^2(t-1)\frac{d}{dt}\cdot z=
t-1 .
$$
Its solution is $z= c-\frac{1}{t}$.
Since $x=0\Longleftrightarrow z=0$ and 
$x=0\Longrightarrow t=1$, we conclude that 
the constant of integration is $c=1$. Thus
$z = 1-1/t$.
\end{proof}

\begin{rem}
The relation between our $z$ as a function 
in $x$ and 
the classical Lambert W-function
(see for example, \cite{CGHJK}) is
$$
z(x) = -W(-x/e).
$$
\end{rem}

Because of the ELSV formula (\ref{eq:Hg(mu)}),
the Laplace transform of $H_g(\mu)$ becomes
a polynomial in $t_i,\dots,t_n$ for
 $(g,n)$ in the stable range. 
 The result is
 \begin{multline}
 \label{eq:FgnH}
 F_{g,n}^H(t_1,\dots,t_n)
 =H_{g,n}(w(t_1),\dots,w(t_n))
 \\
 =
 \sum_{\mu\in\bZ_+^n} H_g(\mu)
e^{-\left(
\mu_1(w_1+1)+\cdots+\mu_n(w_n+1)
\right)}
\\
=
\sum_{\mu\in\bZ_+^n}
\sum_{k_1+\cdots+k_n \le 3g-3+n} \la 
\tau_{k_1}\cdots \tau_{k_n}\Lambda_g^{\vee}(1)\ra
\,\,\prod_{i=1}^{n}
\frac{\mu_i^{\mu_i+k_i}}{\mu_i!}
e^{-\left(
\mu_1(w_1+1)+\cdots+\mu_n(w_n+1)
\right)}
\\
=
\sum_{k_1+\cdots+k_n \le 3g-3+n} \la 
\tau_{k_1}\cdots \tau_{k_n}\Lambda_g^{\vee}(1)\ra
\,\,\prod_{i=1}^{n}
\hxi_{k_i}(t_i).
 \end{multline}
The Laplace transform (\ref{eq:FgnH}) 
is no longer a polynomial for the unstable geometries
$(g,n) = (0,1)$ and $(0,2)$. Wel use
(\ref{eq:unstableHurwitz}) to calculate
$F_{0,1}^H$ and $F_{0,2}^H$.

\begin{thm}
\label{thm:F01H}
The Laplace transform of the unstable 
 cases $(g,n) = (0,1)$ and $(0,2)$ are given by
 \begin{equation}
 \label{eq:F01H}
 F_{0,1}^H(t) 
= 
\half\left(1-\frac{1}{t^2}\right)
\end{equation}
and
\begin{equation}
\label{eq:F02H}
F_{0,2}^H(t_1,t_2) 
=\log\left(
\frac{z_1-z_2}
{x_1-x_2}
\right)
-(z_1+z_2) +1,
\end{equation}
where $t_i,x_i,z_i$ are related by
(\ref{eq:Lambert}) and (\ref{eq:z in t for H}).
\end{thm}

\begin{proof}
The $(0,1)$ case is a straightforward computation.
$$
F_{0,1}^H(t) 
= 
\sum_{k=d} ^\infty
H_0((d))\;e^{-d}x^d
=
\sum_{d=1} ^\infty \frac{d^{d-2}}{d!}\;e^{-d}x^d
=\hxi_{-2}(t).
$$
This is a solution to the differential equation
$$
t^2(t-1)\frac{d}{dt}\;\hxi_{-2}(t) = 
\hxi_{-1}(t) = z=\frac{t-1}{t}.
$$
Therefore, $\hxi_{-2}(t) = c-\half\;\frac{1}{t^2}$
for a constant of integration $c$.
Here again 
we note 
$$
t=1\Longrightarrow z=0\Longrightarrow
x=0\Longrightarrow \hxi_{-2}(t)=0.
$$
This determines that $c=\half$.
Thus we have established (\ref{eq:F01H}).

Since
$$
F_{0,2}^H(t_1,t_2) =
\sum_{\mu_1,\mu_2\ge 1}
\frac{1}{\mu_1+\mu_2}\cdot\frac{\mu_1 ^{\mu_1}}
{\mu_1 !}\;e^{-\mu_1}\cdot\frac{\mu_2 ^{\mu_2}}{\mu_2 !}\;e^{-\mu_2}\cdot
x_1^{\mu_1}\;x_2 ^{\mu_2}
$$
and since $z=\hxi_{-1}(t)$,
(\ref{eq:F02H}) is equivalent 
to 
\begin{equation}
\label{eq:F02-inx}
\sum_{\substack{\mu_1,\mu_2\ge 0
\\(\mu_1,\mu_2)\ne (0,0)}}
\frac{1}{\mu_1+\mu_2}\cdot\frac{\mu_1 ^{\mu_1}}
{\mu_1 !}\;e^{-\mu_1}
\cdot\frac{\mu_2 ^{\mu_2}}{\mu_2 !}
\;e^{-\mu_2}\cdot
x_1 ^{\mu_1}x_2 ^{\mu_2}
=\log\left( e
\sum_{k=1} ^\infty \frac{k^{k-1}}{k!}\;e^{-k}\cdot
\frac{x_1 ^k - x_2 ^k}{x_1-x_2}
\right),
\end{equation}
where $|x_1|< 1, |x_2|<1$, and $0<|x_1-x_2|<1$
so that the formula is an equation of holomorphic functions 
in $x_1$ and $x_2$.
Define
\begin{multline*}
\phi(x_1,x_2)
\\
\overset{{\rm{def}}}{=}\sum_{\substack{\mu_1,\mu_2\ge 0
\\(\mu_1,\mu_2)\ne (0,0)}}
\frac{1}{\mu_1+\mu_2}\cdot\frac{\mu_1 ^{\mu_1}}
{\mu_1 !}\;e^{-\mu_1}
\cdot\frac{\mu_2 ^{\mu_2}}{\mu_2 !}
\;e^{-\mu_2}\cdot
x_1 ^{\mu_1}x_2 ^{\mu_2}
-\log\left(
\sum_{k=1} ^\infty \frac{k^{k-1}}{k!}\;e^{1-k}\cdot
\frac{x_1 ^k - x_2 ^k}{x_1-x_2}
\right).
\end{multline*}
Then
\begin{multline*}
\phi(x,0)=\sum_{\mu_1\ge 1}
\frac{\mu_1 ^{\mu_1 -1}}{\mu_1 !}
\;e^{-\mu_1}
x ^{\mu_1}
-\log\left(
\sum_{k=1} ^\infty \frac{k^{k-1}}{k!}
\;e^{-k}\cdot
x ^{k-1}
\right)-1
\\
=
\hxi_{-1}(t) - \log\left(\frac{\hxi_{-1}(t)}{x}\right)-1
=
1-\frac{1}{t} -\log\left(1-\frac{1}{t}\right) +\log x-1
\\
=
-\frac{1}{t} -\log\left(1-\frac{1}{t}\right) - w = 0
\end{multline*}
because 
$$
x=e^{-w}=ze^{1-z}=\left(1-\frac{1}{t}
\right)e^{\frac{1}{t}}. 
$$
Here $t$ is restricted on the 
domain $Re(t)>1$. 
Since
\begin{multline*}
x_1 \frac{\partial}{\partial x_1}
\log\left(e
\sum_{k=1} ^\infty \frac{k^{k-1}}{k!}
\; e^{-k}\cdot
\frac{x_1 ^k - x_2 ^k}{x_1-x_2}
\right)
\\
=
t_1 ^2 (t_1-1)\frac{\partial}{\partial t_1}
\log\left(\hxi_{-1}(t_1)-\hxi_{-1}(t_2)\right)
-x_1 \frac{\partial}{\partial x_1}\log(x_1-x_2)
\\
=
t_1 ^2 (t_1-1)\frac{\partial}{\partial t_1}
\log\left(-\frac{1}{t_1}+\frac{1}{t_2}\right)
- \frac{x_1}{x_1-x_2}
\\
=
\frac{t_1t_2(t_1-1)}{t_1-t_2}-\frac{x_1}{x_1-x_2},
\end{multline*}
we have 
\begin{multline*}
\left(x_1 \frac{\partial}{\partial x_1}+
x_2 \frac{\partial}{\partial x_2}
\right)
\log\left(e
\sum_{k=1} ^\infty \frac{k^{k-1}}{k!}
\;e^{-k}\cdot
\frac{x_1 ^k - x_2 ^k}{x_1-x_2}
\right)
\\
=
\frac{t_1t_2(t_1-1)-t_1t_2(t_2-1)}{t_1-t_2}
-\frac{x_1-x_2}{x_1-x_2}
\\
=
t_1t_2 - 1 = \hxi_0(t_1)\hxi_0(t_2)+\hxi_0(t_1)+\hxi_0(t_2).
\end{multline*}
On the other hand, we also have
\begin{multline*}
\left(x_1 \frac{\partial}{\partial x_1}+
x_2 \frac{\partial}{\partial x_2}
\right)
\sum_{\substack{\mu_1,\mu_2\ge 0\\
(\mu_1,\mu_2)\ne (0,0)}}
\frac{1}{\mu_1+\mu_2}\cdot\frac{\mu_1 ^{\mu_1}}{\mu_1 !}\;e^{-\mu_1}
\cdot\frac{\mu_2 ^{\mu_2}}{\mu_2 !}
\;e^{-\mu_2}\cdot
x_1 ^{\mu_1}x_2 ^{\mu_2}
\\
=
\sum_{\substack{\mu_1,\mu_2\ge 0\\
(\mu_1,\mu_2)\ne (0,0)}}
\frac{\mu_1+\mu_2}{\mu_1+\mu_2}\cdot\frac{\mu_1 ^{\mu_1}}{\mu_1 !}\;e^{-\mu_1}
\cdot\frac{\mu_2 ^{\mu_2}}{\mu_2 !}
\;e^{-\mu_2}\cdot
x_1 ^{\mu_1}x_2 ^{\mu_2}
\\
=
\hxi_0(t_1)\hxi_0(t_2)+\hxi_0(t_1)+\hxi_0(t_2).
\end{multline*}
Therefore, 
\begin{equation}
\label{eq:Eulerequation}
\left(x_1 \frac{\partial}{\partial x_1}+
x_2 \frac{\partial}{\partial x_2}
\right)\phi(x_1,x_2) = 0.
\end{equation}
Note that $\phi(x_1,x_2)$ is a holomorphic function 
in $x_1$ and $x_2$. Therefore, it has a series
expansion in homogeneous polynomials around $(0,0)$. 
Since a homogeneous polynomial in $x_1$ and $x_2$ of
degree $n$ is an eigenvector of the differential operator 
$x_1 \frac{\partial}{\partial x_1}+
x_2 \frac{\partial}{\partial x_2}$ belonging to the 
eigenvalue $n$, the only holomorphic
solution to the Euler differential
equation (\ref{eq:Eulerequation}) is a constant. But 
since $\phi(x_1,0)=0$, we conclude that 
$\phi(x_1,x_2)=0$. This completes the proof of 
(\ref{eq:F02-inx}), and hence the proposition.
\end{proof}

\begin{Def}
\label{def:WgnH}
We define the symmetric differential forms
for all $g\ge 0$ and $n>0$ by
\begin{equation}
\label{eq:WgnH}
W_{g,n}^H(t_1,\dots,t_n) = d_1\cdots d_n
F_{g,n}^H(t_1,\dots,t_n),
\end{equation}
and call them the Hurwitz differential forms. 
\end{Def}

The unstable cases are given by
\begin{equation}
\label{eq:W01H}
W_{0,1}^H(t_1) =
d_1 F_{0,1}^H(t_1) = \frac{1}{t_1^3}\;dt_1
=  \frac{z}{x}\;dx,
\end{equation}
and 
\begin{multline}
\label{eq:W02H}
W_{0,2}^H(t_1, t_2) =
d_1 d_2F_{0,2}^H(t_1, t_2) = 
d_1 d_2 \left[
\log\left( z_1-z_2\right)
-\log(x_1-x_2)
\right]
\\
=
d_1 d_2 \left[
\log\left( \frac{1}{t_2}-\frac{1}{t_1}\right)
-\log(x_1-x_2)
\right]
=
d_1 d_2 \left[
\log\left( t_1-t_2\right)
-\log(x_1-x_2)
\right]
\\
=\frac{dt_1\cdot dt_2}{(t_1-t_2)^2}
-\frac{dx_1\cdot dx_2}{(x_1-x_2)^2}.
\end{multline}
We note that all quantities are expressible in terms
of $z$, or equivalently, in $t$. 
Now Definition~\ref{def:EO}
tells us that the spectral curve 
$\Sigma$ of the single Hurwitz
number is
\begin{equation}
\label{eq:Hspectral}
\begin{cases}
x=z e^{1-z}\\
y=\frac{z}{x} = e^{z-1}.
\end{cases}
\end{equation}
The \emph{Lambert curve} 
$\Sigma$ defined by 
$x=z e^{1-z}$, which is 
obtained by the 
Laplace transform of the
number of trees,  is an analytic curve and its 
$x$-projection has a simple
ramification point at $z=1$,
since
$$
dx = (1-z)e^{1-z}\;dz.
$$
The $t$-coordinate brings this ramification point
to $t=\infty$. Let $\bar{z}$ (resp.\ $\bar{t}$)
denote the unique 
local Galois conjugate of $z$ (reps.\ $t$).
We also use 
\begin{equation}
\label{eq:s(t)}
\bar{t} = s(t),
\end{equation}
which is defined by the functional equation
\begin{equation}
\label{eq:s(t)equation}
\left(1-\frac{1}{t}\right) e^{\frac{1}{t}}
= 
\left(1-\frac{1}{s(t)}\right) e^{\frac{1}{s(t)}}.
\end{equation}
Although the Galois conjugate is only locally defined
near the branched point $t=\infty$, we consider
$s(t)$ as a global holomorphic function via
analytic continuation. 
For $Re(t)>1$, (\ref{eq:s(t)equation}) implies
$$
w(t)=-\log x = -\left(\frac{1}{t}-\sum_{n=1} ^\infty 
\frac{1}{n}\;
\frac{1}{t^n}\right)
 =\sum_{n=2}^\infty \frac{1}{t^n}.
$$
When considered as a functional equation, 
(\ref{eq:s(t)equation}) has exactly two solutions:
$t$ and
\begin{equation}
\label{eq:s(t)expansion}
s(t) =  -t + \frac{2}{3} + 
\frac{4}{135}t^{-2}+\frac{8}{405} t^{-3}+\frac{8}{567}t^{-4}
+\cdots .
\end{equation}
This is the deck-transformation of 
 the projection $\pi:\Sigma\rightarrow \bC$ near $t=\infty$ 
 and satisfies the involution equation
 $s\big(s(t)\big)= t$.
  It is analytic on
 $\bC\setminus [0,1]$ and has logarithmic singularities
 at $0$ and $1$.

Let us calculate the recursion kernel.
Since
$$
\frac{dx}{x} = \frac{1-z}{z}\;dz
=\frac{dt}{t^2(t-1)}
=\frac{s'(t)dt}{s(t)^2(s(t)-1)} ,
$$
we have
\begin{multline}
\label{eq:Hkernel}
K^H(t,t_1) = 
\half\;\frac{\int_t ^{s(t)} W_{0,2}^H(\cdot, t_1)}
{W_{0,1}(s(t))-W_{0,1}(t)}
=
\half 
\left(
\frac{1}{t-t_1} -\frac{1}{s(t)-t_1}
\right)
\frac{t^2(t-1)}{\frac{1}{t}-\frac{1}{s(t)}}
\cdot \frac{1}{dt}\cdot dt_1
\\
=
\half 
\left(
\frac{1}{t-t_1} -\frac{1}{s(t)-t_1}
\right)
\frac{t s(t)}{s(t)-t}
\cdot \frac{t^2(t-1)}{dt}\cdot dt_1.
\end{multline}

\begin{thm}[\cite{EMS, MZ}]
\label{thm:BM conjecture}
The Hurwitz differential forms (\ref{eq:WgnH})
for $2g-2+n>0$ satisfy the
Eynard-Orantin recursion:
\begin{multline}
\label{eq:HEO}
W_{g,n}^H(t_1,\dots,t_n)
=\frac{1}{2\pi i}\oint_{\gam_\infty}
K^H(t,t_1)
\Bigg[
W_{g-1,n+1}^H(t,s(t),t_2, \dots, t_n)
\\
+
\sum_{\substack{g_1+g_2=g\\
I\sqcup J=\{2,\dots,n\}}} ^{\text{No $(0,1)$-terms}}
W_{g_1,|I|+1}^H(t,t_I)W_{g_2,|J|+1}^H(s(t),t_J)
\Bigg],
\end{multline}
where $\gam_\infty$ is a negatively oriented 
circle around $\infty$ whose radius
is larger than any of $|t_j|$'s and $|s(t_j)|$'s.
\end{thm}

\begin{rem}
The recursion formula (\ref{eq:HEO}) was first
conjectured by Bouchard and Mari\~no in 
\cite{BM}. Its proofs appear in 
\cite{BEMS, EMS, MZ}. The method of \cite{BEMS}
is to use a matrix integral expression of the single
Hurwitz numbers. The idea of \cite{EMS,MZ}
is that the Laplace transform of the 
\emph{cut-and-join
equation} of \cite{GJ,V} is the Eynard-Orantin
recursion. The cut-and-join equation takes the 
following form:
\begin{multline}
\label{eq:caj}
r(g,\mu){H}_g  (\mu)
=
\sum_{i< j}
(\mu_i+\mu_j)
{H}_g \big(\mu(\hat{i},\hat{j}),\mu_i+\mu_j\big)
 \\
 + \frac{1}{2} \sum_{i=1} ^n
 \sum_{\alpha+\beta = \mu_i}\alpha\beta 
\left({H}_{g-1} \big(\mu(\hat{i}),\a,\b\big)
+
\sum_{\substack{g_1+g_2 = g\\
\nu_1\sqcup \nu_2 = \mu(\hat{i})}}
{H}_{g_1} (\nu_1,\a)
{H}_{g_2} (\nu_2,\b)
\right).
\end{multline}
Here $\mu$ is a partition of length $n$, and
 $\mu(\hat{i})$ and $\mu(\hat{i},\hat{j})$
 indicate the partition obtained by deleting 
 parts of $\mu$. 
\end{rem}

\begin{rem}
As we have seen above,
Hurwitz numbers in (\ref{eq:unstableHurwitz})
determine the shape of the recursion formula
(\ref{eq:HEO}). Since the recursion 
gives the Hurwitz numbers for all $(g,n)$, 
we have thus established that unstable
$(g,n) = (0,1)$ and $(0,2)$ Hurwitz numbers
determine all other single Hurwitz numbers.
\end{rem}

It is important to check if the formulas
(\ref{eq:W11}) and
(\ref{eq:W03})  agree
with the geometry. 
From the definition (\ref{eq:FgnH})
we calculate
$$
F_{0,3}^H(t_1,t_2,t_3)
=\la \tau_0\tau_0\tau_0\ra_{0,3} \hxi_0(t_1)
\hxi_0(t_2)\hxi_0(t_3)
=(t_1-1)(t_2-1)(t_3-1),
$$
which yields 
\begin{equation}
\label{eq:W03H}
W_{0,3}^H(t_1,t_2,t_3) 
= dt_1dt_2dt_3.
\end{equation}
Since 
$$
dx(z)\cdot dy(z) = (1-z)dz\cdot dz =\frac{dt \cdot dt}{t^5} 
$$
from (\ref{eq:Hspectral})
and (\ref{eq:Lambert}), 
the general formula
 (\ref{eq:W03}) yields
\begin{multline*}
W_{0,3}^H(t_1,t_2,t_3)
=
-\frac{1}{2\pi i} 
\oint_{\gam_\infty}
\frac{W_{0,2}^H(t,t_1)W_{0,2}^H(t,t_2)
W_{0,2}^H(t,t_3)}
{dx(t)\cdot dy(t)}
\\
=
-\left[
\frac{1}{2\pi i} 
\oint_{\gam_\infty}
\frac{t^5}{(t-t_1)^2(t-t_2)^2(t-t_3)^2}dt
\right]
dt_1dt_2dt_3
=dt_1dt_2dt_3,
\end{multline*}
in agreement with geometry. Here we calculate
the residue at $t=\infty$. Although
$$
W_{0,2}^H(t,t_i)=\frac{dt\cdot dt_i}{(t-t_i)^2}
-\frac{dx\cdot dx_i}{(x-x_i)^2},
$$
the second term does not contribute to the integral.
This is because as $t\rar\infty$, we have $x\rar 1$, and 
$dx\cdot dx_i/{(x-x_i)^2}$ has no pole at $x=1$.

Similarly, 
$$
F_{1,1}^H(t_1)=\la\tau_1\ra_{1,1} \hxi_1(t_1)
-\la \tau_0\lam_1\ra_{1,1}\hxi_0(t_1)
=\frac{1}{24}(t_1^2-1)(t_1-1),
$$
and thus we have
\begin{equation}
\label{eq:W11H}
W_{1,1}^H(t_1) = \frac{1}{24}(t_1-1)(3t_1+1)dt_1.
\end{equation}
On the other hand, the general formula (\ref{eq:W11})
gives
\begin{multline*}
W_{1,1}^H(t_1) = 
\frac{1}{2\pi i} \oint_{\gam_\infty}
K^H(t,t_1)
\left.
\left[
W_{0,2}^H(u,v)+\frac{dx(u)\cdot dx(v)}
{(x(u)-x(v))^2}
\right]
\right|_{\substack{u=t\\v=s(t)}}
\\
=
\frac{1}{2\pi i} \oint_{\gam_\infty}
K^H(t,t_1)
\frac{dt\cdot s'(t)dt}{(t-s(t))^2}
\\
=
\left[
\frac{1}{2\pi i} \oint_{\gam_\infty}
\half 
\left(
\frac{1}{t-t_1} -\frac{1}{s(t)-t_1}
\right)
\frac{t s(t)}{s(t)-t}
t^2(t-1)
\frac{s'(t)dt}{(t-s(t))^2}
\right]\;dt_1
\\
=
\frac{t_1 s(t_1)}{(t_1-s(t_1))^3}s(t_1)^2 (s(t_1)-1)\;dt_1
\\
-
\left[
\frac{1}{2\pi i} \oint_{\gam_{[0,1]}}
\half 
\left(
\frac{1}{t-t_1} -\frac{1}{s(t)-t_1}
\right)
\frac{t s(t)}{s(t)-t}
t^2(t-1)
\frac{s'(t)dt}{(t-s(t))^2}
\right]\;dt_1,
\end{multline*}
where $\gam_{[0,1]}$ is a contour circling
around the slit $[0,1]$ in the $t$-plane in the 
positive direction.

\begin{figure}[htb]
\centerline{\epsfig{file=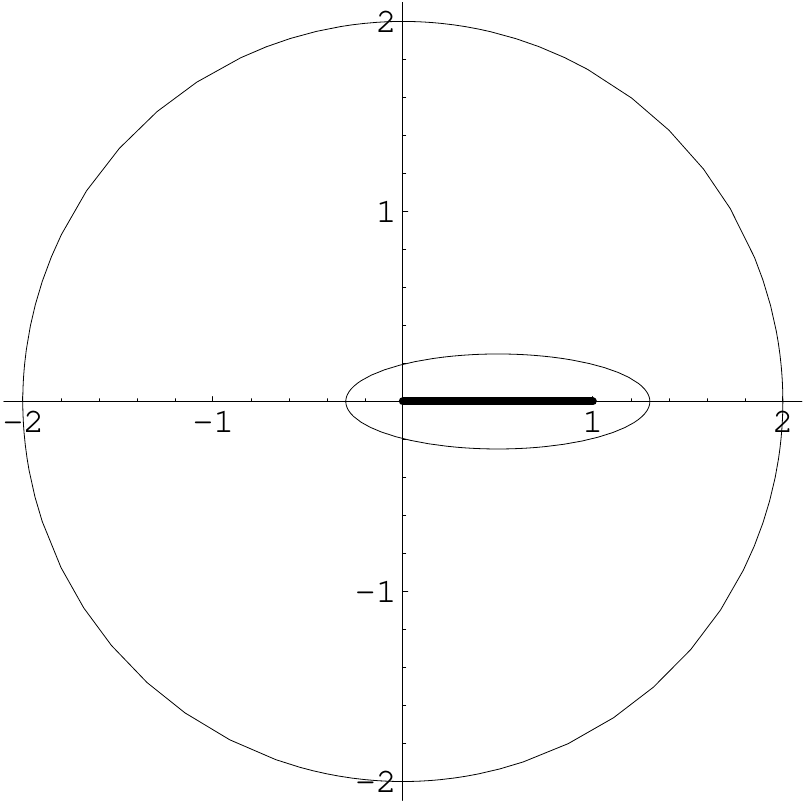, width=2in}}
\caption{The contours
 of integration. The outer loop $\gamma_\infty$ is the
 circle of a large radius oriented
 clock wise, and $\gamma_{[0,1]}$ is the thin
loop surrounding the closed interval $[0,1]$
in the positive direction. }
\label{fig:contourH}
\end{figure}

Note that the integrand of the last integral is a holomorphic
function in $t$ on $\gam_{[0,1]}$, hence it has
a finite value. It is also clear that as $t_1\rar\infty$, 
this integral tends to $0$, because $\gam_{[0,1]}$
is a compact space.
Therefore, we conclude that
\begin{multline*}
W_{1,1}^H(t_1)
= \frac{t_1 s(t_1)}{(t_1-s(t_1))^3}s(t_1)^2 (s(t_1)-1)\;dt_1
+O(1/t_1)
\\
=
\left(
\frac{1}{8}\;t_1^2 -\frac{1}{12}\; t_1-\frac{1}{24}
\right)
\;dt_1
+O(1/t_1),
\end{multline*}
since $s(t) = -t +2/3 +O(1/t^2)$.
It agrees with (\ref{eq:W11H}) because of
the following

\begin{lem}
\label{lem:Hpolynomial}
A solution to the topological recursion 
(\ref{eq:HEO}) is a polynomial in $t_1$.
\end{lem}

\begin{proof}
The $t_1$-dependence of $W_{g,n}^H(t_1,\dots,t_n)$
only comes from the factor
\begin{multline*}
\left(
\frac{1}{t-t_1} -\frac{1}{s(t)-t_1}
\right)
\\
= \frac{1}{t} +
+\frac{1}{3}\;\frac{1}{t^2}
+
\left(
t_1^2 -\frac{2}{3} t_1 +\frac{2}{9}
\right)\frac{1}{t^3}
+\left(
t_1^2-\frac{2}{3}t_1+\frac{22}{135}
\right)\frac{1}{t^4}+\cdots
\end{multline*}
in the recursion kernel (\ref{eq:Hkernel}). 
Since each
coefficient of the  $t$-expansion of $K^H(t,t_1)$ 
is a polynomial in $t_1$, Lemma follows.
\end{proof}

\section{The stationary Gromov-Witten
invariants of $\bP^1$}
\label{sect:GWP1}

In this section we study the generating functions
of stationary Gromov-Witten invariants of $\bP^1$.
The conjectural relation between these invariants and the
Eynard-Orantin
topological recursion was first formulated
in \cite{NS2}. We identify the spectral curve
and the recursion kernel using the unstable 
geometries.

Morally speaking, the space $\bP^1$ we are
considering here
 appears as the zero section of a Calabi-Yau
threefold known as the \emph{resolved conifold} 
$X$, which is the total 
space of the rank $2$ vector bundle
$\cO_{\bP^1}(-1)\dsum\cO_{\bP^1}(-1)$
over ${\bP^1}$. Let  
$L\subset X$ be a special Lagrangian submanifold
\cite{Hitchin, Joyce}. Then the intersection $L\cap \bP^1$
of the special Lagrangian and the zero
section is a circle on $\bP^1$. 
If we holomorphically
embed a bordered Riemann surface with
$n$ boundary components into $X$ in a way that
each boundary is mapped to a distinct circle on
$\bP^1$, then the whole Riemann surface is
necessarily mapped to $\bP^1$. Thus we are 
 considering \emph{open} Gromov-Witten invariants
of $\bP^1$. And if we make these circles on $\bP^1$ 
small and centered around 
$n$ distinct points of $\bP^1$,
then we are naturally led to the stationary 
Gromov-Witten
invariants of $\bP^1$. 

So our main object of this section is the 
Laplace transform of the stationary Gromov-Witten
invariants
\begin{equation}
\label{eq:FgnP}
F_{g,n}^{\bP^1}(x_1,\dots,x_n)
= \sum_{\mu_1,\dots,\mu_n=0}^\infty
\la \tau_{\mu_1}(\omega)
\cdots
 \tau_{\mu_n}(\omega)\ra_{g,n}
 \prod_{i=1}^n \mu_i !
\prod_{i=1}^n \frac{1}{x^{\mu_i+1}},
\end{equation}
where $\omega\in A_0(\bP^1)$ is the point class
generator, and
\begin{equation}
\label{eq:GWP1}
\la \tau_{\mu_1}(\omega)
\cdots
 \tau_{\mu_n}(\omega)\ra_{g,n}
 =
 \int_{[\Mbar_{g,n}(\bP^1,d)]^{\text{virt}}}
 \psi_1^{\mu_1}ev_1^*(\omega)
 \cdots
  \psi_1^{\mu_n}ev_n^*(\omega)
\end{equation}
is a stationary Gromov-Witten invariant of $\bP^1$.
More precisely, $\Mbar_{g,n}(\bP^1,d)$ is the 
moduli stack of stable morphisms
from a connected $n$-pointed curve 
$(C,p_1,\dots,p_n)$ into
$\bP^1$ 
of degree $d$ such that $f(p_i)$, $i=1,\dots,n$,
are distinct,
and $ev_i$ is the natural evaluation morphism
$$
ev_i:\Mbar_{g,n}(\bP^1,d) \owns
[f,(C,p_1,\dots,p_n)]\longmapsto f(p_i) \in\bP^1.
$$
The Gromov-Witten invariant (\ref{eq:GWP1}) 
vanishes unless 
\begin{equation}
\label{eq:P1degree}
2g-2+2d = \mu_1+\cdots +\mu_n.
\end{equation}
The sum in (\ref{eq:FgnP})
 is the Laplace transform if we identity
\begin{equation}
\label{eq:x=ewP}
x = e^w.
\end{equation}
The extra numerical factor $\prod_{i=1}^n \mu_i !$
is included in (\ref{eq:FgnP}) because of the 
polynomial growth order of
\begin{equation}
\label{eq:mu!}
\la \tau_{\mu_1}(\omega)
\cdots
 \tau_{\mu_n}(\omega)\ra_{g,n}
 \prod_{i=1}^n \mu_i !
 \end{equation} 
 for large $\mu$
 that is established in \cite{OP2}.
 Indeed (\ref{eq:mu!}) is essentially
 a special type of \emph{Hurwitz numbers} that
counts the number of certain coverings of $\bP^1$.

To determine the spectral curve and 
the annulus amplitude, 
 we need to consider unstable geometries
 $(g,n)=(0,1)$ and $(0,2)$. From \cite{OP2}
 we learn
 \begin{equation}
 \label{eq:P01}
 \la \tau_{\mu_1}(\omega)\ra_{0,1}
 =
 \la \tau_{2d-2}(\omega)\ra_{0,1}
 =\left(\frac{1}{d!}\right)^2.
 \end{equation}
  To compute a closed formula for 
$$
 F_{0,1}^{\bP^1}(x)
 =\sum_{\mu_1=0}^\infty
 \la \tau_{\mu_1}(\omega)\ra_{0,1}\;\mu_1!
 \frac{1}{x^{\mu_1+1}}
 =
 \sum_{d=1}^\infty
 \frac{(2d-2)!}{d!d!}\;
 \frac{1}{x^{2d-1}},
 $$ 
we notice that the generating function
 of Catalan numbers (\ref{eq:Catalan})
$$
 z(x) = \sum_{m=0}^\infty C_m \frac{1}{x^{2m+1}}
$$
 provides again an effective tool. Thus we have
 \begin{multline}
 \label{eq:F01Pequation}
 \left(
 x\frac{d}{dx}-1
 \right)F_{0,1}^{\bP^1}(x)
 =
- 2\sum_{d=1}^\infty
 \frac{(2d-2)!}{(d-1)!d!}\;\frac{1}{x^{2d-1}}
 \\
 =
 - 2\sum_{m=0}^\infty
 \frac{(2m)!}{(m+1)!m!}\;\frac{1}{x^{2m+1}}
 =
 -2z(x).
 \end{multline}
 The advantage of using the Catalan series $z(x)$ is
 that we know its inverse function (\ref{eq:x=x(z)}).
Using (\ref{eq:dx and dz}), we see that 
(\ref{eq:F01Pequation}) is equivalent to
\begin{equation}
\label{eq:F01Pequation in z}
\left(
\frac{z^3+z}{z^2-1}\;\frac{d}{dz}-1
\right)
F_{0,1}^{\bp^1}(z)
=-2z.
\end{equation}
The solution of (\ref{eq:F01Pequation in z}) is given by
$$
F_{0,1}^{\bp^1}(z) =
-\frac{2}{z}-\left(z+\frac{1}{z}\right)
\log(1+z^2) +c\left(z+\frac{1}{z}\right),
$$
with   a constant of integration $c$.
Since 
$$
z\rar 0\Longrightarrow x\rar \infty
\Longrightarrow F_{0,1}^{\bP^1}\rar 0,
$$
we conclude that $c=2$. 
We  thus obtain
\begin{equation}
\label{eq:F01Pinz}
F_{0,1}^{\bP^1}(z) = 2z-\left(z+\frac{1}{z}\right)
\log(1+z^2),
\end{equation}
and therefore,
\begin{equation}
\label{eq:W01P}
W_{0,1}^{\bP^1}(z) =
d F_{0,1}^{\bP^1}(z)
= -\log(1+z^2)\; d\left(
z+\frac{1}{z}\right).
\end{equation}

\begin{thm}
\label{thm:Pspectral}
The spectral curve for the stationary 
Gromov-Witten  invariants of $\bP^1$ is
given by
\begin{equation}
\label{eq:Pspectral}
\begin{cases}
x=z+\frac{1}{z}\\
y=-\log(1+z^2).
\end{cases}
\end{equation}
\end{thm}

\begin{rem}
Since $dx=0$ has two zeros at $z=\pm 1$, 
we also use as 
our preferred coordinate 
\begin{equation}
\label{eq:t}
t=\frac{z+1}{z-1} \Longleftrightarrow
z=\frac{t+1}{t-1}.
\end{equation}
The $\log$ singularity on the $t$-plane is
the right semicircle of radius $1$ connecting
$i$ to $-i$ (see Figure~\ref{fig:semicircle}).
The expression of $W_{0,1}^{\bP^1}$
in terms of the preferred coordinate is
\begin{equation}
\label{eq:W01Pint}
W_{0,1}^{\bP^1}(t)=
\frac{8t}{(t^2-1)^2} \log\left(
\frac{2(t^2+1)}{(t-1)^2}
\right)dt.
\end{equation}
\end{rem}

\begin{figure}[htb]
\centerline{\epsfig{file=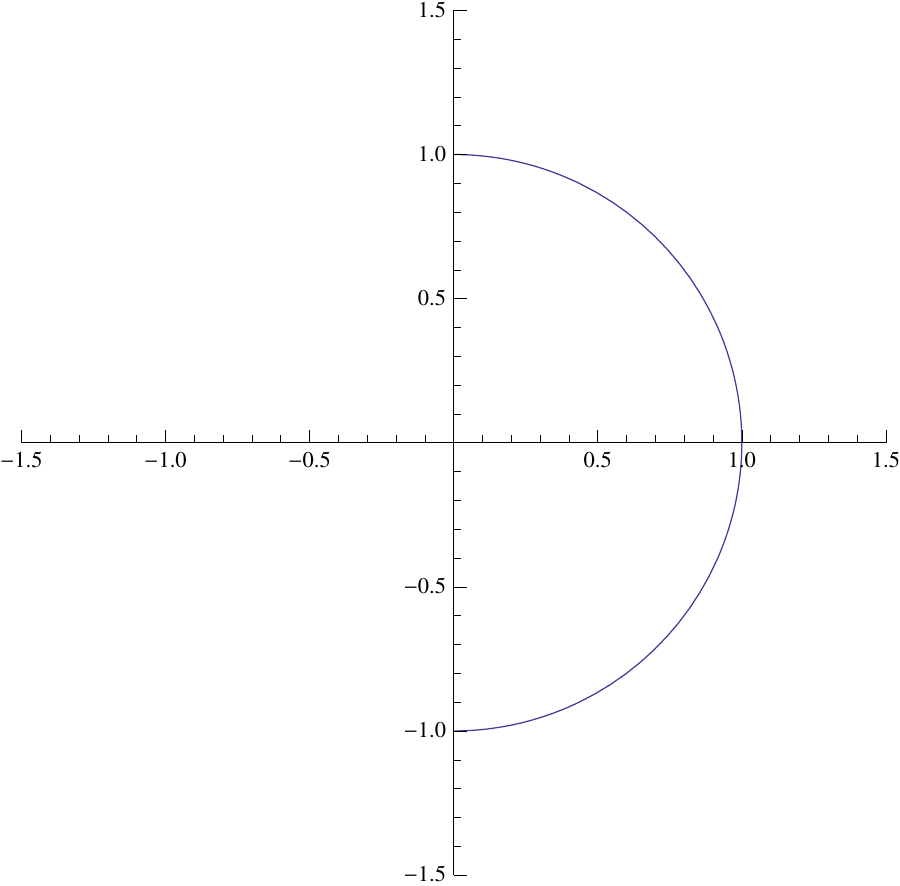, width=2in}}
\caption{The spectral curve for the stationary
Gromov-Witten invariants of $\bP^1$ is 
the complex $t$-plane minus the semicircle.}
\label{fig:semicircle}
\end{figure}

\begin{rem}
The function $x=z+\frac{1}{z}$ is expected here,
since it is the Landau-Ginzburg model that is 
homologically
mirror dual to $\bP^1$ \cite{Ballard}.
\end{rem}

\begin{rem}
The Galois conjugate of $x=z+\frac{1}{z}$ is
 globally defined, and is given by
 \begin{equation}
 \label{eq:Galois}
 t\longmapsto \bar{t}=-t.
 \end{equation}
\end{rem}

\begin{rem}
Since 
\begin{equation}
\label{eq:1/(1-z)}
\frac{1}{1-z(x)}=\sum_{k=0}^\infty z(x)^k
=1+\sum_{n=0} ^\infty \binom{n}{\lfloor 
\frac{n}{2}\rfloor}\frac{1}{x^{n+1}},
\end{equation}
we can express $t$ in the  branch near $t=-1$
as a function in $x$. The result
is 
\begin{equation}
\label{eq:t in x}
t+1 = \frac{z(x)+1}{z(x)-1}+1= 2-\frac{2}{1-z(x)}
=-\sum_{n=0}^\infty 2 \binom{n}{\lfloor 
\frac{n}{2}\rfloor}\frac{1}{x^{n+1}},
\end{equation}
which is also absolutely convergent for $|x|>2$.
\end{rem}

\begin{rem}
We are using the normalized Gromov-Witten invariants
(\ref{eq:mu!}) to compute the Laplace transform
(\ref{eq:FgnP}). If we did not include the $\mu!$ 
factor in our computation of the spectral curve,
then we would have encountered
with  the \emph{modified 
Bessel function}
$$
I_0(2x) = \sum_{m=1}^\infty \frac{1}{(m!)^2}
x^{2m},
$$
instead of  $z(x)$, in computing (\ref{eq:F01Pinz}). 
We note that $I_0(2x)$ appears in \cite{DZ2}
in the exact same context of computing the
Gromov-Witten invariants of $\bP^1$. 
We prefer the Catalan number series $z(x)$
over the modified Bessel function
mainly  because  the  inverse function of $z(x)$
takes a simple form $x=z+\frac{1}{z}$.
\end{rem}

Motivated by the technique developed in 
\cite{BM,EMS,MZ}
 for single Hurwitz numbers,
let us define
\begin{equation}
\label{eq:xin}
\xi_n(t)=\sum_{k=0} ^\infty \binom{2k}{k}
k^n \frac{1}{x^{2k+1}}, \qquad n\ge 0,
\end{equation}
and
\begin{equation}
\label{eq:etan}
\eta_n(t)=\sum_{k=0} ^\infty \binom{2k+1}{k}
k^n \frac{1}{x^{2k+2}}, \qquad n\ge 0.
\end{equation}
We then have
\begin{equation}
\label{eq:xinrecursion}
\xi_{n+1}(t) =
-\half
\left(x\frac{d}{dx}+1\right) \xi_n(t)
=
\left(
\frac{t^4-1}{8t}\frac{d}{dt}-\half
\right)\xi_n(t)
\end{equation}
and
\begin{equation}
\label{eq:etanrecursion}
\eta_{n+1}(t) =
-\half
\left(x\frac{d}{dx}+2\right) \eta_n(t)
=
\left(
\frac{t^4-1}{8t}\frac{d}{dt}-1
\right)\eta_n(t).
\end{equation}
The initial values are computed as follows:
\begin{multline}
\label{eq:xi0inz}
\xi_0(t)=
\half \left(
1-x\frac{d}{dx}
\right)
\sum_{m=0}^\infty \frac{1}{m+1}
\binom{2m}{m}\frac{1}{x^{2m+1}}
\\
=
\half
\left(1-\frac{z(z^2+1)}{z^2-1}\;\frac{d}{dz}\right)z
=-\frac{z}{z^2-1}
=-\frac{t^2-1}{4t},
\end{multline}
and similarly 
\begin{equation}
\label{eq:eta0}
\eta_0(t) = -\frac{(t+1)^2}{4t}.
\end{equation}
We note that $\xi_n(t)$ and $\eta_n(t)$ are
Laurent polynomials 
of degree $2n+1$
for every $n\ge 0$. Since they
are defined as functions in $x$, we have the 
reciprocity property
\begin{equation}
\label{eq:reciprocity}
\begin{aligned}
\xi_n(1/t) &= -\xi_n(t)
\\
\eta_n(1/t) &=\eta_n(t).
\end{aligned}
\end{equation}
This follows from 
$$
t\longmapsto 
\frac{1}{t}
\Longrightarrow 
x\longmapsto -x.
$$

The annulus amplitude requires 
$(g,n)=(0,2)$ Gromov-Witten invariants.
They can be calculated from the 
$(g,n)=(0,1)$ invariants using the 
Topological Recursion Relation
\cite{Getzler}. The results are
\begin{equation}
\label{eq:P02}
 \la \tau_{\mu_1}(\omega)
 \tau_{\mu_2}(\omega)\ra_{0,2}
 =
 \begin{cases}
 \frac{1}{(m_1!)^2 (m_2!)^2}
 \;\frac{1}{ (m_1+m_2+1)}
 \qquad \mu_1=2m_1, \mu_2=2m_2\\
 \frac{1}{(m_1!)^2 (m_2!)^2}
 \;\frac{1}{ (m_1+m_2+2)}
  \qquad \mu_1=2m_1+1, \mu_2=2m_2+1.
 \end{cases}
\end{equation}

\begin{thm}
\label{thm:F02P}
The annulus amplitude is given by
\begin{equation}
\label{eq:F02P}
F_{0,2}^{\bP^1}(z_1,z_2)
=-\log(1-z_1z_2).
\end{equation}
Hence we have
\begin{equation}
\label{eq:W02P}
W_{0,2}^{\bP^1}(t_1,t_2)
= \frac{dt_1\cdot dt_2}{(t_1-t_2)^2}-
\frac{dx_1\cdot dx_2}{(x_1-x_2)^2}
=
\frac{dt_1\cdot dt_2}{(t_1+t_2)^2}.
\end{equation}
\end{thm}

\begin{proof}
From  (\ref{eq:P02}) we calculate
\begin{multline*}
F_{0,2}^{\bP^1}(z_1,z_2)
=
\sum_{\mu_1,\mu_2=0}^\infty 
\la \tau_{\mu_1}(\omega)
\tau_{\mu_2}(\omega)\ra_{0,2}\;
\mu_1 !\mu_2!\;\frac{1}{x_1^{\mu_1+1}}\;
\frac{1}{x_2^{\mu_2+1}}
\\
=
\sum_{m_1,m_2=0}^\infty
\frac{1}{ (m_1+m_2+1)}
\binom{2m_1}{m_1}\binom{2m_2}{m_2}
\frac{1}{x_1^{2m_1+1}}\;
\frac{1}{x_2^{2m_2+1}}
\\
+
\sum_{m_1,m_2=0}^\infty
\frac{1}{ (m_1+m_2+2)}(2m_1+1)(2m_2+1)
\binom{2m_1}{m_1}\binom{2m_2}{m_2}
\frac{1}{x_1^{2m_1+2}}\;
\frac{1}{x_2^{2m_2+2}}.
\end{multline*}
Thus we have
\begin{multline*}
\left(
x_1\frac{d}{dx_1}+x_2\frac{d}{dx_2}
\right)F_{0,2}^{\bP^1}(z_1,z_2)
\\
=
-2\sum_{m_1,m_2=0}^\infty
\binom{2m_1}{m_1}\binom{2m_2}{m_2}
\frac{1}{x_1^{2m_1+1}}\;
\frac{1}{x_2^{2m_2+1}}
\\
-
2\sum_{m_1,m_2=0}^\infty
(2m_1+1)(2m_2+1)
\binom{2m_1}{m_1}\binom{2m_2}{m_2}
\frac{1}{x_1^{2m_1+2}}\;
\frac{1}{x_2^{2m_2+2}}
\\
=
-2\xi_0(x_1)\xi_0(x_2)
-2z'(x_1)z'(x_2)
\\
=
-2\frac{z_1}{z_1^2-1}\;
\frac{z_2}{z_2^2-1}
-2 \frac{z_1^2}{z_1^2-1}\;\frac{z_2^2}{z_2^2-1}
=-2 \frac{z_1z_2(1+z_1z_2)}{(z_1^2-1)(z_2^2-1)},
\end{multline*}
where $\xi_0(x)$ is calculated in (\ref{eq:xi0inz}),
and from  (\ref{eq:dx and dz}) we know
$$
z'(x) = \frac{dz}{dx}=\frac{z^2}{z^2-1}.
$$
On the other hand, 
\begin{multline*}
\left(
x_1\frac{d}{dx_1}+x_2\frac{d}{dx_2}
\right)
\left(-\log(1-z_1z_2)\right)
\\
=
\left(
\frac{z_1(z_1^2+1)}{z_1^2-1}\;\frac{d}{dz_1}
+
\frac{z_2(z_2^2+1)}{z_2^2-1}\;\frac{d}{dz_2}
\right)\left(-\log(1-z_1z_2)\right)
\\
=
\left(
\frac{(z_1^2+1)}{z_1^2-1}
+
\frac{(z_2^2+1)}{z_2^2-1}
\right)\frac{z_1z_2}{1-z_1z_2}
=
-2\;\frac{z_1z_2(1+z_1z_2)}{(z_1^2-1)(z_2^2-1)}.
\end{multline*}
Therefore, 
\begin{multline*}
\left(
x_1\frac{d}{dx_1}+x_2\frac{d}{dx_2}
\right)
\left(
F_{0,2}^{\bP^1}(z_1,z_2)+\log(1-z_1z_2)
\right)
\\
=
\left(
x_1\frac{d}{dx_1}+x_2\frac{d}{dx_2}
\right)
\Bigg(
\sum_{\mu_1,\mu_2=0}^\infty 
\la \tau_{\mu_1}(\omega)
\tau_{\mu_2}(\omega)\ra_{0,2}\;
\mu_1 !\mu_2!\;\frac{1}{x_1^{\mu_1+1}}\;
\frac{1}{x_2^{\mu_2+1}}
\\
-\sum_{n=1}^\infty
\frac{1}{n}\left(
\sum_{m=0}^\infty C_m\frac{1}{x_1^{2m+1}}
\sum_{m=0}^\infty C_m\frac{1}{x_2^{2m+1}}
\right)^n
\Bigg) =0.
\end{multline*}
Since the kernel of the Euler differential operator 
is the constants, and since actual
computation shows that the first few expansion
terms of the  Laurent series 
$$
\sum_{\mu_1,\mu_2=0}^\infty 
\la \tau_{\mu_1}(\omega)
\tau_{\mu_2}(\omega)\ra_{0,2}\;
\mu_1 !\mu_2!\;\frac{1}{x_1^{\mu_1+1}}\;
\frac{1}{x_2^{\mu_2+1}}
-\sum_{n=1}^\infty
\frac{1}{n}\left(
\sum_{m=0}^\infty C_m\frac{1}{x_1^{2m+1}}
\sum_{m=0}^\infty C_m\frac{1}{x_2^{2m+1}}
\right)^n
$$
are $0$, we
complete the proof of (\ref{eq:F02P}).
\end{proof}

Using $\xi_n(t)$ and $\eta_n(t)$ of
(\ref{eq:xin}) and (\ref{eq:etan})
and the classical topological recursion relation
\cite{Getzler}, we can systematically
calculate  the Laplace transform of stationary 
Gromov-Witten invariants. 
First let us consider 
 $(g,n)=(0,3)$. Since the sum of the 
 descendant indices 
of 
$$
\la \tau_{\mu_1}(\omega)
\tau_{\mu_2}(\omega)
\tau_{\mu_3}(\omega)\ra_{0,3}
$$ 
is even,
we have
\begin{equation}
\label{eq:GW03}
\begin{aligned}
\la \tau_{2m_1}(\omega)
\tau_{2m_2}(\omega)
\tau_{2m_3}(\omega)\ra_{0,3}
&=
\frac{1}{m_1^2m_2^2m_3^2},
\\
\la \tau_{2m_1}(\omega)
\tau_{2m_2+1}(\omega)
\tau_{2m_3+1}(\omega)\ra_{0,3}
&=
\frac{(m_2+1)(m_3+1)}
{m_1^2(m_2+1)^2(m_3+1)^2}.
\end{aligned}
\end{equation}
The Laplace transform is therefore 
\begin{multline}
\label{eq:F03P}
F_{0,3}^{\bP^1}(t_1,t_2,t_3)
=
\sum_{\mu_1\mu_2\mu_3\ge 0}
\la \tau_{\mu_1}(\omega)
\tau_{\mu_2}(\omega)
\tau_{\mu_3}(\omega)\ra_{0,3}
\mu_1!\mu_2!\mu_3!
\frac{1}{x_1^{\mu_1+1}}\cdot
\frac{1}{x_2^{\mu_2+1}}\cdot
\frac{1}{x_3^{\mu_3+1}}
\\
=
\sum_{m_1,m_2,m_3\ge 0}
\binom{2m_1}{m_1}
\binom{2m_2}{m_2}
\binom{2m_3}{m_3}
\frac{1}{x_1^{2m_1+1}}\cdot
\frac{1}{x_2^{2m_2+1}}\cdot
\frac{1}{x_3^{2m_3+1}}
\\
+
\sum_{m_1,m_2,m_3\ge 0}
\binom{2m_1}{m_1}
\binom{2m_2+1}{m_2}
\binom{2m_3+1}{m_3}
\frac{1}{x_1^{2m_1+1}}\cdot
\frac{1}{x_2^{2m_2+2}}\cdot
\frac{1}{x_3^{2m_3+2}}
\\
+
\sum_{m_1,m_2,m_3\ge 0}
\binom{2m_1+1}{m_1}
\binom{2m_2}{m_2}
\binom{2m_3+1}{m_3}
\frac{1}{x_1^{2m_1+2}}\cdot
\frac{1}{x_2^{2m_2+1}}\cdot
\frac{1}{x_3^{2m_3+2}}
\\
+
\sum_{m_1,m_2,m_3\ge 0}
\binom{2m_1+1}{m_1}
\binom{2m_2+1}{m_2}
\binom{2m_3}{m_3}
\frac{1}{x_1^{2m_1+2}}\cdot
\frac{1}{x_2^{2m_2+2}}\cdot
\frac{1}{x_3^{2m_3+1}}
\\
=
\xi_0(t_1)\xi_0(t_2)\xi_0(t_3)
+\xi_0(t_1)\eta_0(t_2)\eta_0(t_3)
+\eta_0(t_1)\xi_0(t_2)\eta_0(t_3)
+\eta_0(t_1)\eta_0(t_2)\xi_0(t_3)
\\
=
-\frac{1}{16}(t_1+1)(t_2+1)(t_3+1)
\left(1-\frac{1}{t_1t_2t_3}
\right),
\end{multline}
which is indeed a Laurent polynomial. 
Since it is an odd degree polynomial in 
$\xi_n(t)$'s, we have the reciprocity 
$$
F_{0,3}^{\bP^1}(1/t_1,1/t_2,1/t_3)
=
-F_{0,3}^{\bP^1}(t_1,t_2,t_3).
$$

The $n=1$ stationary invariants
are concretely calculated in 
\cite{OP2}. We have
\begin{equation}
\label{eq:n=1}
\begin{aligned}
\la \tau_{2d}\ra_{1,1} &= \frac{1}{24}\left(
\frac{1}{d!}\right)^2 (2d-1)
\\
\la \tau_{2d+2}\ra_{2,1} &= \left(
\frac{1}{d!}\right)^2 
\left(
\frac{1}{5! \;4^2}(2d-1)+\frac{1}{24^2}
\binom{2d-1}{2}\right)
\\
\la \tau_{2g-2+2d}\ra_{g,1} &= \left(
\frac{1}{d!}\right)^2 
\sum_{\ell=1}^g \binom{2d-1}{\ell}
\sum_{\substack{k_i>0\\
k_1+\cdots+k_\ell = g}}
\prod_{i=1}^\ell \frac{1}{(2k_i+1)!\; 4^{k_i}}.
\end{aligned}
\end{equation}
We thus obtain
\begin{multline}
\label{eq:F11P}
F_{1,1}^{\bP^1}(t_1)
=\frac{1}{24}\sum_{d=0}^\infty 
\binom{2d}{d}(2d-1)\frac{1}{x_1^{2d+1}}
=
\frac{1}{24}
\left(
2\xi_1(t_1)-\xi_0(t_1)
\right)
\\
=
-\frac{1}{384}
\left(
t_1^3-7t_1+\frac{7}{t_1}-\frac{1}{t_1^3}
\right).
\end{multline}
To calculate the $g=2$ case we need to do the following.
\begin{multline}
\label{eq:F21P}
F_{2,1}^{\bP^1}(t_1)
=
\sum_{d=0}^\infty 
\frac{(2d+2)!}{d!\;d!}
\left(
\frac{1}{5! \;4^2}(2d-1)+\frac{1}{24^2}
\binom{2d-1}{2}\right)
\frac{1}{x_1^{2d+3}}
\\
=
\left(\frac{d}{dx_1}\right)^2
\sum_{d=0}^\infty 
\binom{2d}{d}
\left(
\frac{1}{5! \;4^2}(2d-1)+\frac{1}{24^2}\;
(2d^2-3d+1)\right)
\frac{1}{x_1^{2d+1}}
\\
=
\left(-\frac{(t^2-1)^2}{8t}\;\frac{d}{dt}\right)^2
\left[
\frac{1}{5! \;4^2}(2\xi_1(t_1)-\xi_0(t_1))
+
\frac{1}{24^2}
\left(
2\xi_2(t_1)-3\xi_1(t_1)+\xi_0(t_1)
\right)
\right]
\\
=
-\frac{1}{2^{19}\cdot 3^2\cdot 5}
\;\frac{(t^2-1)^3}{t^9}
\big(
525 t_1^{12}-1470t_1^{10}
+1107t_1^8+527t_1^6+1107t_1^4-1470t_1^2+525
\big).
\end{multline}

\begin{prop}
\label{prop:Fg1P}
$F_{g,1}^{\bP^1}(t_1)$ is a 
Laurent polynomial of degree $6g-3$ with the 
reciprocity 
$$
F_{g,1}^{\bP^1}(1/t_1)
=
-F_{g,1}^{\bP^1}(t_1).
$$
\end{prop}

\begin{proof}
First we calculate the binomial coefficient
$$
\binom{2d-1}{\ell}=\frac{1}{\ell!}
(2d-1)(2d-2)\cdots(2d-\ell)
=
\frac{1}{\ell!}
\left(
2^\ell d^\ell -\frac{\ell(\ell+1)}{2}d^{\ell-1}
+\cdots+(-1)^\ell\ell!
\right)
$$
as a polynomial in $d$, and then replace each $d^i$
with $\xi_i(t_1)$. The result is a linear combination 
of $\xi_0(t_1),\dots,\xi_\ell(t_1)$. Let $\Xi_\ell(t_1)$
denote the resulting Laurent polynomial of
degree $2\ell+1$.
Then we have an expression
\begin{multline}
\label{eq:Fg1P}
F_{g,1}^{\bP^1}(t_1)
=
\left(\frac{d}{dx_1}
\right)^{2g-2}
\sum_{d=0}^\infty
\binom{2d}{d} 
\sum_{\ell=1}^g \binom{2d-1}{\ell}
\sum_{\substack{k_i>0\\
k_1+\cdots+k_\ell = g}}
\prod_{i=1}^\ell \frac{1}{(2k_i+1)!\; 4^{k_i}}\;
\frac{1}{x_1^{2d+1}}
\\
=
\left(-\frac{(t^2-1)^2}{8t}\;\frac{d}{dt}\right)^{2g-2}
\left[
\sum_{\ell=1}^g \;
\Xi_\ell(t_1)
\sum_{\substack{k_i>0\\
k_1+\cdots+k_\ell = g}}
\prod_{i=1}^\ell \frac{1}{(2k_i+1)!\; 4^{k_i}}
\right],
\end{multline}
which is a Laurent polynomial of degree 
$2(2g-2)+2g+1=6g-3$. 
The reciprocity property follows from 
(\ref{eq:reciprocity}) and the
$x_1$ expression of (\ref{eq:Fg1P}), where
$x_1$ changes to $-x_1$. In particular, the
even order 
differentiation in $x_1$ is not affected by this change.
\end{proof}

The Eynard-Orantin recursion is for the differential 
forms $W_{g,n}^{\bP^1}(t_1,\dots,t_n)$. 
We need the recursion kernel. 
From (\ref{eq:W01Pint}) and (\ref{eq:W02P}),
we compute 
\begin{multline}
\label{eq:Pkernel}
K^{\bP^1}(t,t_1)=
\half\;
\frac{\int_t ^{-t}W_{0,2}^{\bP^1}(\cdot,t_1)}
{W_{0,1}^{\bP^1}(-t)-W_{0,1}^{\bP^1}(t)}
\\
=
\half
\left(
\frac{1}{t+t_1}+\frac{1}{t-t_1}
\right)
\frac{1}{\log\left(\frac{2(t^2+1)}{(t+1)^2}\right)
-\log\left(\frac{2(t^+1)}{(t-1)^2}\right)}
\frac{(t^2-1)^2}{-8tdt}\cdot dt_1
\\
=
\frac{1}{16}\left(
\frac{1}{t+t_1}+\frac{1}{t-t_1}
\right)
\frac{1}{\log\left(\frac{(t-1)^2}{(t+1)^2}\right)}
\frac{(t^2-1)^2}{tdt}\cdot dt_1.
\end{multline}
We note the reciprocity property of the kernel
\begin{equation}
\label{eq:Kreciprocity}
K^{\bP^1}(1/t,1/t_1) =-K^{\bP^1}(t,t_1).
\end{equation}
The topological recursion (\ref{eq:EO})  
becomes
\begin{multline}
\label{eq:EOP1}
W_{g,n}^{\bP^1}(t_1,t_2,\dots,t_n)
= \frac{1}{2\pi i} 
\oint_{\gam} K^{\bP^1}(t,t_1)
\Bigg[
W_{g-1,n+1}^{\bP^1}(t,{-t},t_2,\dots,t_n)\\
+
\sum^{\text{No $(0,1)$ terns}} _
{\substack{g_1+g_2=g\\I\sqcup J=\{2,3,\dots,n\}}}
W_{g_1,|I|+1}^{\bP^1}(t,t_I) W_{g_2,|J|+1}^{\bP^1}
({-t},t_J)
\Bigg],
\end{multline}
where the residue calculation is taken along
the integration contour $\gam$ 
(see Figure~\ref{fig:contourD})
consisting of two concentric
circles of radius $\epsilon$ and $1/\epsilon$ for 
a small $\epsilon$
centered around $t=0$, with the inner circle 
positively oriented and the outer circle  
negatively oriented. 
Since there is a log singularity in the complex
$t$-plane, we cannot use the residue calculus
method to evaluate the integral
at $t=t_1$ and $t=-t_1$.
Thus the residue calculation of
(\ref{eq:EOP1}) is performed around the
neighborhood of $t=0$ and $t=\infty$.

So let us provide two expansion formulas
for the kernel $K^{\bP^1}(t,t_1)$, assuming 
that $t_1\in\bC^*$ is away from the log 
singularity of Figure~\ref{fig:semicircle}.
The transcendental factor of $K^{\bP^1}(t,t_1)$
has an expansion
\begin{equation}
\label{eq:log}
\frac{4}{t\log\left(\frac{(t-1)^2}{(t+1)^2}\right)}
=
-\frac{1}{t^2}+\frac{1}{3}+\frac{4}{45}t^2
+\frac{44}{945}t^4+\frac{428}{14175}t^6
+\frac{10196}{467775} t^8+\cdots
\end{equation}
around $t=0$. The denominator of the coefficient
of $t^{2k-2}$ is given by
$$
\prod_{q =3, \text{  prime}} ^{2k+1}
q^{\left\lfloor \frac{2k}{q-1}\right\rfloor} 
=
3^{\left\lfloor k\right\rfloor}\cdot
5^{\left\lfloor \frac{k}{2}\right\rfloor}
\cdot
7^{\left\lfloor \frac{k}{3}\right\rfloor}
\cdots,
$$
which is the same as $\mu(L_k)$ of
\cite[Lemma 1.5.2]{Hirzebruch}.
The expansion of 
$\frac{1}{t+t_1}+\frac{1}{t-t_1}$
at $t=0$ is given by
$$
\frac{1}{t+t_1}+\frac{1}{t-t_1}
=-2t\;\frac{1}{t_1^2}\frac{1}{1-\frac{t^2}{t_1^2}}
=-2\sum_{n=0}^\infty
\frac{t^{2n+1}}{t_1^{2n+2}}.
$$
From the expression (\ref{eq:Pkernel})
and the above consideration, 
we know that 
around $t=0$, $K^{\bP^1}(t,t_1)$
starts from $t^{-1}$, and that
the coefficient of $t^{2n-1}$ is 
a Laurent polynomial
in $t_1^2$  starting 
from $\frac{1}{32}t_1^{-(2n+2)}$
up to $t_1^{-2}$ with  rational coefficients.
More concretely, 
 we have
\begin{multline}
\label{eq:Kexpansion0}
K^{\bP^1}(t,t_1) 
= 
\Bigg[
\frac{1}{t}
\left(
\frac{1}{32}\frac{1}{t_1^2}
\right)
+
t\left(
\frac{1}{32}\frac{1}{t_1^4}-\frac{7}{96}
\frac{1}{t_1^2}
\right)
+
t^3\left(
\frac{1}{32}\frac{1}{t_1^6}-\frac{7}{96}
\frac{1}{t_1^4}+\frac{71}{1440}\frac{1}{t_1^2}
\right)
\\
+
t^5\left(
\frac{1}{32}\frac{1}{t_1^8}-\frac{7}{96}
\frac{1}{t_1^6}+\frac{71}{1440}\frac{1}{t_1^4}
-\frac{191}{30240}\frac{1}{t_1^2}
\right)
\\
+
t^7\left(
\frac{1}{32}\frac{1}{t_1^{10}}-\frac{7}{96}
\frac{1}{t_1^8}+\frac{71}{1440}\frac{1}{t_1^6}
-\frac{191}{30240}\frac{1}{t_1^4}-
\frac{23}{28350}\frac{1}{t_1^2}
\right)
\\
+
t^9 
\left(
\frac{1}{32}\frac{1}{t_1^{12}}-\frac{7}{96}
\frac{1}{t_1^{10}}+\frac{71}{1440}\frac{1}{t_1^8}
-\frac{191}{30240}\frac{1}{t_1^6}
-\frac{23}{28350}\frac{1}{t_1^4}
-\frac{233}{935550}\frac{1}{t_1^2}
\right)
+\cdots
\Bigg] \frac{1}{dt}\cdot dt_1.
\end{multline}
Similarly,  around $t=\infty$ we have
\begin{multline}
\label{eq:Kexpansioninfinity}
K^{\bP^1}(t,t_1) 
=
\Bigg[
- t^3\frac{1}{32}+
t\left(-
\frac{1}{32}t_1^2+\frac{7}{96}
\right)
+
\frac{1}{t}\left(-
\frac{1}{32}{t_1^{4}}+\frac{7}{96}
{t_1^{2}}-\frac{71}{1440}
\right)
\\
+
\frac{1}{t^3}\left(-
\frac{1}{32}{t_1^{6}}+\frac{7}{96}
{t_1^{4}}-\frac{71}{1440}{t_1^2}
+\frac{191}{30240}
\right)
\\
+
\frac{1}{t^5}\left(-
\frac{1}{32}{t_1^{8}}+\frac{7}{96}
{t_1^{6}}-\frac{71}{1440}{t_1^4}
+\frac{191}{30240}{t_1^2}+
\frac{23}{28350}
\right)
\\
+
\frac{1}{t^7}\left(-
\frac{1}{32}{t_1^{10}}+\frac{7}{96}
{t_1^{8}}-\frac{71}{1440}{t_1^6}
+\frac{191}{30240}{t_1^4}+
\frac{23}{28350}{t_1^2}+\frac{233}{935550}
\right)
+\cdots
\Bigg] \frac{1}{dt}\cdot dt_1.
\end{multline}

\begin{thm}
\label{thm:polynomial}
The Eynard-Orantin differential form
$W_{g,n}^{\bP^1}(t_1,\dots,t_n)$
is a Laurent polynomial in $t_1^2,t_2^2,\dots, t_n^2$
of degree $2(3g-3+n)$
in the stable range $2g-2+n>0$. It satisfies the
reciprocity property
\begin{equation}
\label{eq:Wreciprocity}
W_{g,n}^{\bP^1}(1/t_1,\dots,1/t_n)
=(-1)^n W_{g,n}^{\bP^1}(t_1,\dots,t_n)
\end{equation}
as a meromorphic symmetric $n$-form.
The highest degree terms form a homogeneous 
polynomial of degree $2(3g-3+n)$, which
is given by
\begin{equation}
\label{eq:leading}
\widehat{W}_{g,n}^{\bP^1}(t_1,\dots,t_n)
=\frac{(-1)^n}{2^{2g-2+2n}}
\sum_{k_1,\dots,k_n\ge 0}
\la \tau_{k_1}\cdots\tau_{k_n}\ra_{g,n}
\prod_{i=1}^n
\left[
 (2k_1+1)!!
\left(\frac{t_i}{2}\right)^{2k_i}dt_i
\right].
\end{equation}
Indeed it is the same as the generating function
of the $\psi$-class intersection numbers
(\ref{eq:WgnK}).
\end{thm}

\begin{proof}
The statement is proved 
by induction on $2g-2+n$ using
the recursion (\ref{eq:EOP1}).
The initial cases 
$(g,n)=(1,1)$ and $(g,n)=(0,3)$
are easily verified from the 
concrete calculations below.
Since we are expanding 
$\frac{1}{t+t_1}+\frac{1}{t-t_1}$
around $t=0$ and $t=\infty$, 
it is obvious that the recursion produces
a Laurent polynomial in 
$t_1^2,t_2^2,\dots, t_n^2$ as the result. 

The expression of (\ref{eq:Kexpansioninfinity})
tells us that the residue calculation at infinity
increases the degree by $4$. This is because 
the leading term of the coefficient of 
$t^{-(2n+1)}$ is $t_1^{2n+4}$, and the 
residue calculation picks up the term $t^{2n}$. 
By the induction hypothesis, the right-hand side
of (\ref{eq:EOP1}) without the kernel 
term has homogenous degree 
$2(3g-3+n) -4$. 
The reciprocity property also follows by
induction using (\ref{eq:Kreciprocity}).

The leading terms of $W_{g,n}^{\bP^1}(t_1,\dots,t_n)$
satisfy a topological recursion themselves. 
We can extract the terms in the kernel that
produce the leading terms of the differential forms
from (\ref{eq:log}) or (\ref{eq:Kexpansioninfinity}).
The result is 
\begin{equation}
\label{eq:WKkernel}
K^{\text{WK}}(t,t_1)=-\frac{1}{32}\; t^3
\sum_{k=0}^\infty \frac{t_1^{2n}}{t^{2n}}
\frac{1}{dt}\cdot dt_1
=
-\half
\left(
\frac{1}{t-t_1}+\frac{1}{t+t_1}
\right)
\frac{1}{32}\; t^4\cdot 
\frac{1}{dt}\cdot dt_1,
\end{equation}
which is identical to \cite[Theorem~7.4]{CMS},
and also to (\ref{eq:Kkernel}).
Since the topological recursion uniquely
determines all the differential forms from 
the initial condition, and again since the
$(g,n)=(0,3)$ and $(1,1)$ cases satisfy
(\ref{eq:leading}), by induction we obtain
(\ref{eq:leading}) for all stable values of $(g,n)$. 
\end{proof}

The $(g,n)=(1,1)$ Eynard-Orantin differential 
form is computed using 
(\ref{eq:W11}).
\begin{multline}
\label{eq:W11Pa}
W_{1,1}^{\bP^1} (t_1) =\frac{1}{2\pi i}
\int_{\gam}K^{\bP^1}(t,t_1) 
\left[W_{0,2}^{\bP^1} (t, -t)
+\frac{dx\cdot dx_1}{(x-x_1)^2}
\right]
=
-\frac{1}{2\pi i}
\int_{\gam}K^{\bP^1}(t,t_1) 
\frac{dt\cdot dt}{4 t^2}
\\
=
-\frac{1}{64}
\left(
\frac{1}{2\pi i}
\int_{\gam}\left(
\frac{1}{t+t_1}+\frac{1}{t-t_1}
\right)
\frac{1}{\log\left(\frac{(t-1)^2}{(t+1)^2}\right)}
\frac{(t^2-1)^2}{t^3}
dt
\right) dt_1
\\
=
\left(
-\frac{1}{128}t_1^2+\frac{7}{384}+\frac{7}{384}
\frac{1}{t_1^2}-\frac{1}{128}\frac{1}{t_1^4}
\right) dt_1.
\end{multline}
This is in agreement of $W_{1,1}^{\bP^1}(t_1)
=dF_{1,1}^{\bP^1}(t_1)$ and (\ref{eq:F11P}).
From (\ref{eq:EOP1}) we have
\begin{multline}
\label{eq:W03P}
W_{0,3}^{\bP^1}(t_1,t_2,t_3)
\\
=
\frac{1}{2\pi i}
\int_{\gam}K^{\bP^1}(t,t_1) 
\left[
W_{0,2}^{\bP^1}(t,t_2)W_{0,2}^{\bP^1}(-t,t_3)
+
W_{0,2}^{\bP^1}(t,t_3)W_{0,2}^{\bP^1}(-t,t_2)
\right]
\\
=-\frac{1}{16}
\left(1+\frac{1}{t_1^2\;t_2^2\;t_3^2}
\right)dt_1dt_2dt_3.
\end{multline}
It is also in agreement with (\ref{eq:F03P}).

Norbury and Scott conjecture the following

\begin{conj}[Norbury-Scott Conjecture \cite{NS2}]
For $(g,n)$ in the stable range we have
\begin{equation}
\label{eq:NS}
W_{g,n}^{\bP^1}(t_1,\dots,t_n)
= d_1\cdots d_n F_{g,n}^{\bP^1}(t_1,\dots,t_n).
\end{equation}
\end{conj}

The conjecture is verified for $g=0$ and $g=1$ cases
in \cite{NS2}. We recall that the Eyanrd-Orantin 
recursion for simple Hurwitz numbers is essentially
the Laplace transform of the cut-and-join
equation \cite{EMS}. For the case of the 
counting problem
of clean Belyi morphisms
 the recursion is the Laplace transform
of the edge-contraction operation of
Theorem~\ref{thm:Dgn}.

\begin{quest}
What is the equation among the stationary 
Gromov-Witten invariants of $\bP^1$
whose Laplace
transform is the Eynard-Orantin recursion
(\ref{eq:EOP1})?
\end{quest}

\begin{appendix}

\section{Calculation of the Laplace transform}
\label{app:LTProof}
\setcounter{section}{1}

In this appendix we give the proof of 
Theorem~\ref{thm:LTofD}.

\begin{prop}
\label{prop:LT of Dgn}
Let us use the $x_j$-variables defined
by $x_j=e^{w_j}$, and write
$$
W_{g,n}^D(t_1,\dots,t_n) =
w_{g,n}(x_1,\dots,x_n)\; dx_1\cdots dx_n.
$$
Then
the Laplace transform of the recursion formula
(\ref{eq:Dgn recursion}) is the following 
differential recursion:
\begin{multline}
\label{eq:LT of Dgn}
-x_1\;w_{g,n}(x_1,\dots,x_n)
\\
=
\sum_{j=2}^n
 \frac{\partial}{\partial x_j}
     \left(
    \frac{1}{x_j-x_1}
    \left(w_{g,n-1}(x_2,\dots,x_n)
    -
    w_{g,n-1}(x_1,x_2,\dots,\widehat{x_j},\dots,x_n)
    \right)
    \right)
    \\
  +
 w_{g-1,n+1}(x_1,x_1,x_2,\dots,x_n)
 +
 \sum_{\substack{g_1+g_2=g\\
 I\sqcup J=\{2,\dots,n\}}}
 w_{g_1,|I|+1}(x_1,x_I)
  w_{g_2,|J|+1}(x_1,x_J).
\end{multline}
\end{prop}

\begin{proof}
The operation we wish to do  is to apply
 $$
 (-1)^n
 \sum_{\mu_1,\dots,\mu_n>0} \mu_2\cdots\mu_n
 \prod_{i=1}^n \frac{1}{x_i^{\mu_i+1}}
 $$
 to  each side of  (\ref{eq:Dgn recursion}).
Then by (\ref{eq:WgnD}), the left-hand side
becomes $w_{g,n}(x_1,\dots,x_n)$.

The second line of  (\ref{eq:Dgn recursion}) is
straightforward. Let us just consider the first term,
since the computation of  the second term is the same.
\begin{multline*}
 (-1)^n
 \sum_{\mu_1,\dots,\mu_n>0} \mu_2\cdots\mu_n
 \prod_{i=1}^n \frac{1}{x_i^{\mu_i+1}}
 \sum_{\a+\b=\mu_1-2}
 \a\b D_{g-1,n+1}(\a,\b,\mu_2,\dots,\mu_n)
 \\
 =
 -\frac{1}{x_1}
  (-1)^{n+1}
 \sum_{\mu_2,\dots,\mu_n>0}\sum_{\a,\b>0}
 \a\b\mu_2\cdots\mu_n
 D_{g-1,n+1}(\a,\b,\mu_2\dots,\mu_n)
 \frac{1}{x_1^{\a+1}}\cdot
\frac{1}{x_1^{\b+1}}
\prod_{i=2}^n\frac{1}{x_i^{\mu_i+1}}
\\
=
 -\frac{1}{x_1}
 w_{g-1,n+1}(x_1,x_1,x_2,\dots,x_n).
\end{multline*}
Thus the second line of (\ref{eq:Dgn recursion})
produces
$$
-\frac{1}{x_1}
\Bigg(
 w_{g-1,n+1}(x_1,x_1,x_2,\dots,x_n)
 +
 \sum_{\substack{g_1+g_2=g\\
 I\sqcup J=\{2,\dots,n\}}}
 w_{g_1,|I|+1}(x_1,x_I)
  w_{g_2,|J|+1}(x_1,x_J)
  \Bigg).
$$

To calculate the operation on
the first line of (\ref{eq:Dgn recursion}),
let us fix $j>1$ and
 set $\nu=\mu_1+\mu_j-2\ge 0$. Then 
\begin{multline}
\label{eq:LT1}
(-1)^n
 \sum_{\mu_1,\dots,\mu_n>0} \mu_2\cdots\mu_n
 (\mu_1+\mu_j-2)
 \\
 \times
 D_{g,n-1}(\mu_1+\mu_j-2,\mu_2,\dots,
 \widehat{\mu_j},\dots,\mu_n)
 \prod_{i=1}^n \frac{1}{x_i^{\mu_i+1}}
 \\
 =
 -\sum_{\nu=0}^\infty
 \sum_{\mu_2,\dots,
 \widehat{\mu_j},\dots,\mu_n>0}
 (-1)^{n-1}\nu\mu_2\cdots\widehat{\mu_j}\cdots\mu_n
 \\
 \times
 D_{g,n-1}(\nu,\mu_2,\dots,\widehat{\mu_j},
 \dots,\mu_n) 
 \frac{1}{x_1^{\nu+1}}\prod_{i\ne 1,j}
  \frac{1}{x_i^{\mu_i+1}}
   \sum_{\mu_j=1}^{\nu+1}
 \mu_j x_1^{\mu_j-2}\frac{1}{x_j^{\mu_j+1}}.
 \end{multline}
Assuming $|x_1|<|x_j|$, we calculate
\begin{multline}
\label{eq:muj sum}
 \sum_{\mu_j=1}^{\nu+1}
 \mu_j x_1^{\mu_j-2}\frac{1}{x_j^{\mu_j+1}}
 =
 -\frac{1}{x_1^2}\frac{\partial}{\partial x_j}
 \sum_{\mu_j=0}^{\nu+1}
 \left(\frac{x_1}{x_j}\right)^{\mu_j}
 =
  -\frac{1}{x_1^2}\frac{\partial}{\partial x_j}
  \left(
  \frac{1}{1-\frac{x_1}{x_j}}
  -\frac{\left(\frac{x_1}{x_j}\right)^{\nu+2}}
  {1-\frac{x_1}{x_j}}
  \right)
  \\
  =
   -\frac{1}{x_1^2}\frac{\partial}{\partial x_j}
   \left(
    \frac{1}{1-\frac{x_1}{x_j}}
    \right)
    +
    x_1^\nu\frac{\partial}{\partial x_j}
    \left(
    \frac{1}{x_j-x_1}\;\frac{1}{x_j^{\nu+1}}
    \right).
\end{multline}
We then substitute (\ref{eq:muj sum}) in (\ref{eq:LT1})
and obtain
\begin{multline}
\text{(A.2)} =
w_{g,n-1}(x_1,x_2,\dots,\widehat{x_j},\dots,x_n)
\frac{1}{x_1^2}\frac{\partial}{\partial x_j}
   \left(
    \frac{1}{1-\frac{x_1}{x_j}}
    \right)
    \\
    -\frac{1}{x_1}
    \frac{\partial}{\partial x_j}
    \left(
    \frac{1}{x_j-x_1}\;w_{g,n-1}(x_2,\dots,x_j,\dots,x_n)
    \right)
    \\
    =
    -\frac{1}{x_1} \frac{\partial}{\partial x_j}
     \left(
    \frac{1}{x_j-x_1}
    \left(w_{g,n-1}(x_2,\dots,x_j,\dots,x_n)
    -
    w_{g,n-1}(x_1,x_2,\dots,\widehat{x_j},\dots,x_n)
    \right)
    \right).
\end{multline}
This completes the proof.
\end{proof}

\begin{proof}[Proof of Theorem~\ref{thm:LTofD}]
When the curve is split into two pieces,
the second term of the third line of 
(\ref{eq:LT of Dgn}) contains contributions from
unstable geometries $(g,n)=(0,1)$ and $(0,2)$. 
We first separate them out. 
For $g_1=0$ and $I=\emptyset$, or $g_2=0$ and
$J=\emptyset$,  we have
a contribution of 
$$
2 w_{0,1}(x_1)w_{g,n}(x_1,x_2,\dots,x_n).
$$
Similarly, for $g_1=0$ and $I=\{j\}$, 
or $g_2=0$ and $J=\{j\}$, 
we have
$$
2\sum_{j=2}^n w_{0,2}(x_1,x_j)
w_{g,n-1}(x_1,\dots,\widehat{x_j},
\dots,x_n).
$$
Since $W_{0,1}^D$ and $W_{0,2}^D$ are
defined on the spectral curve, it is time for us
to switch to the preferred coordinate $t$ of
(\ref{eq:z(t)}) now.
We thus introduce
\begin{equation}
\label{eq:wgnD}
W_{g,n}^D(t_1,\dots,t_n)=w_{g,n}^D(t_1,\dots,t_n)\;
dt_1\cdots dt_n
=
w_{g,n}(x_1,\dots,x_n)\;dx_1\cdots dx_n.
\end{equation}
Since $w_{0,1}(x)=-z(x)$, we have
\begin{align*}
w_{0,1}(x)&=-\frac{t+1}{t-1}\\
w_{0,2}(x_1,x_2)&=
\frac{1}{(t_1+t_2)^2}\;\frac{(t_1^2-1)^2}{8t_1}
\;\frac{(t_2^2-1)^2}{8t_2}\\
w_{g,n}(x_1,\dots,x_n)&=
(-1)^n w_{g,n}^D(t_1,\dots,t_n)\prod_{i=1}^n
\frac{(t_i^2-1)^2}{8t_i}.
\end{align*}
Thus (\ref{eq:LT of Dgn}) is equivalent to 
\begin{multline*}
 2\left(\frac{t_1^2+1}{t_1^2-1}-\frac{t_1+1}{t_1-1}
 \right)
w_{g,n}^D(t_1,\dots,t_n)
\\
=
\sum_{j=2}^n
\Bigg(
\frac{(t_1^2-1)^2(t_j^2-1)^2}{16 (t_1^2-t_j^2)^2}\;
    \frac{8t_j}{(t_j^2-1)^2}\;
    w_{g,n-1}^D(t_1,\dots,\widehat{t_j},\dots,t_n)
\\
    +
    \frac{\partial}{\partial t_j}
\left(
    \frac{(t_1^2-1)(t_j^2-1)}{4 (t_1^2-t_j^2)}
     \frac{8t_1}{(t_1^2-1)^2}
     \frac{(t_j^2-1)^2}{8t_j}
         w_{g,n-1}^D(t_2,\dots,t_n)
         \right)
     \Bigg)
\\
+
\frac{(t_1^2-1)^2}{8t_1}
\left(
w_{g-1,n+1}^D(t_1,t_1,t_2,\dots,t_n)
+\sum_{\substack{g_1+g_2=g\\
I\sqcup J=\{2,\dots,n\}}} ^{\text{stable}}
w_{g_1,|I|+1}^D(t_1,t_I)
w_{g_2,|J|+1}^D(t_1,t_J)
\right)
\\
+
2\sum_{j=2}^n
\frac{1}{(t_1+t_j)^2}\;\frac{(t_1^2-1)^2}{8t_1}
w_{g,n-1}^D(t_1,\dots,\widehat{t_j},\dots,t_n)
\\
=
\sum_{j=2}^n
\Bigg(
\left(
\frac{ t_j(t_1^2-1)^2}{2 (t_1^2-t_j^2)^2}
     +  
\frac{1}{(t_1+t_j)^2}\;\frac{(t_1^2-1)^2}{4t_1}
\right)
w_{g,n-1}^D(t_1,\dots,\widehat{t_j},\dots,t_n)
\\
    +
    \frac{t_1}{t_1^2-1}\;
    \frac{\partial}{\partial t_j}
\left(
    \frac{(t_j^2-1)^3}{4t_j (t_1^2-t_j^2)}\;
         w_{g,n-1}^D(t_2,\dots,t_n)
         \right)
     \Bigg)
     \\
+
\frac{(t_1^2-1)^2}{8t_1}
\left(
w_{g-1,n+1}^D(t_1,t_1,t_2,\dots,t_n)
+\sum_{\substack{g_1+g_2=g\\
I\sqcup J=\{2,\dots,n\}}} ^{\text{stable}}
w_{g_1,|I|+1}^D(t_1,t_I)
w_{g_2,|J|+1}^D(t_1,t_J)
\right).
\end{multline*}
Since
$$
 2\left(\frac{t_1^2+1}{t_1^2-1}-\frac{t_1+1}{t_1-1}
 \right)
=
-\frac{4t_1}{t_1^2-1},
$$
we obtain
\begin{multline}
\label{eq:D-diff}
w_{g,n}^D(t_1,\dots,t_n)
=
-\sum_{j=2}^n
\Bigg(
    \frac{\partial}{\partial t_j}
\left(
    \frac{(t_j^2-1)^3}{16t_j (t_1^2-t_j^2)}\;
         w_{g,n-1}^D(t_2,\dots,t_n)
         \right)
     \\
      +
  \frac{(t_1^2-1)^3}  {16 t_1^2}
  \;
  \frac{t_1^2+t_j^2}{(t_1^2-t_j^2)^2}
w_{g,n-1}^D(t_1,\dots,\widehat{t_j},\dots,t_n)
     \Bigg)
     \\
-
\frac{(t_1^2-1)^3}{32 t_1^2}
\left(
w_{g-1,n+1}^D(t_1,t_1,t_2,\dots,t_n)
+\sum_{\substack{g_1+g_2=g\\
I\sqcup J=\{2,\dots,n\}}} ^{\text{stable}}
w_{g_1,|I|+1}^D(t_1,t_I)
w_{g_2,|J|+1}^D(t_1,t_J)
\right).
\end{multline}
Now let us compute the integral  
\begin{multline}
\label{eq:DEO-A}
W_{g,n}^D(t_1,\dots,t_n)
=
-\frac{1}{64} \; 
\frac{1}{2\pi i}\int_\gam
\left(
\frac{1}{t+t_1}+\frac{1}{t-t_1}
\right)
\frac{(t^2-1)^3}{t^2}\cdot \frac{1}{dt}\cdot dt_1
\\
\times
\Bigg[
\sum_{j=2}^n
\bigg(
W_{0,2}^D(t,t_j)W_{g,n-1}(-t,t_2,\dots,\widehat{t_j},
\dots,t_n)
+
W_{0,2}^D(-t,t_j)W_{g,n-1}(t,t_2,\dots,\widehat{t_j},
\dots,t_n)
\bigg)
\\
+
W_{g-1,n+1}^D(t,{-t},t_2,\dots,t_n)
+
\sum^{\text{stable}} _
{\substack{g_1+g_2=g\\I\sqcup J=\{2,3,\dots,n\}}}
W_{g_1,|I|+1}^D(t,t_I) W_{g_2,|J|+1}^D({-t},t_J)
\Bigg].
\end{multline}
Recall that
For $2g-2+n>0$, $w_{g,n}^D(t_1,\dots,t_n)$
is a Laurent polynomial in $t_1^2,\dots,t_n^2$.
Thus the third line of (\ref{eq:DEO-A}) is
immediately calculated because
the integration contour
$\gam$ of Figure~\ref{fig:contourD} encloses
 $\pm t_1$ and 
 contributes residues
 with the negative sign.
 The result is 
 exactly the last line of (\ref{eq:D-diff}).
Similarly, since
\begin{multline*}
W_{0,2}^D(t,t_j)W_{g,n-1}(-t,t_2,\dots,\widehat{t_j},
\dots,t_n)
+
W_{0,2}^D(-t,t_j)W_{g,n-1}(t,t_2,\dots,\widehat{t_j},
\dots,t_n)
\\
=
-\left(
\frac{1}{(t+t_j)^2}+\frac{1}{(t-t_j)^2}
\right)
w_{g,n-1}^D(t,t_2,\dots,\widehat{t_j},\dots,t_n)
\;dt\;dt\;dt_2\cdots\widehat{dt_j}\cdots dt_n,
\end{multline*}
the residues at $\pm t_1$ contributes 
$$
-\frac{(t_1^2-1)^3(t_1^2+t_j^2)}
{16 t_1^2(t_1^2-t_j^2)^2}\;
w_{g,n-1}^D(t_1,\dots,\widehat{t_j},\dots,t_n).
$$
This is the same as the second line of the 
right-hand side of (\ref{eq:D-diff}).

Within the contour $\gam$, there are second order
poles at $\pm t_j$ for each $j\ge 2$ that come from
$W_{0,2}^D(\pm t,t_j)$. Note that 
$W_{0,2}^D(t,t_j)$ acts as the
Cauchy differentiation kernel. We calculate
\begin{multline*}
\frac{1}{64} \; 
\frac{1}{2\pi i}\int_\gam
\left(
\frac{1}{t+t_1}+\frac{1}{t-t_1}
\right)
\frac{(t^2-1)^3}{t^2}
\sum_{j=2}^n
\bigg(
w_{0,2}^D(t,t_j)w_{g,n-1}^D(-t,t_2,\dots,\widehat{t_j},
\dots,t_n)
\\
+
w_{0,2}^D(-t,t_j)w_{g,n-1}(t,t_2,\dots,\widehat{t_j},
\dots,t_n)
\bigg)
\\
=
-\frac{1}{32} 
\frac{\partial}{\partial t_j}
\left(
\left(
\frac{1}{t_j+t_1}+\frac{1}{t_j-t_1}
\right)
\frac{(t_j^2-1)^3}{t_j^2}
w_{g,n-1}^D(t_j,t_2,\dots,\widehat{t_j},
\dots,t_n)
\right)
\\
=
-\frac{1}{16}
\frac{\partial}{\partial t_j}
\left(
\frac{1}{t_j^2-t_1^2}\;
\frac{(t_j^2-1)^3}{t_j}
w_{g,n-1}^D(t_j,t_2,\dots,\widehat{t_j},
\dots,t_n)
\right).
\end{multline*}
This gives the first line of the right-hand side
of (\ref{eq:D-diff}).
We have thus
 completes the proof of Theorem~\ref{thm:LTofD}.
\end{proof}

\end{appendix}

\begin{ack}
The authors thank 
G.~Borot, V.~Bouchard, A.~Brini, 
K.~Chapman, B.~Eynard,
D.~Hern\'andez Serrano, 
G.~Gliner, M.~Mari\~no,
P.~Norbury, R.~Ohkawa,
M.~Penkava, G.~Shabat,
S.~Shadrin, R.~Vakil, and D.~Zagier
for stimulating and useful discussions.
During the preparation of this paper, the
authors received support from   NSF
grants DMS-0905981, DMS-1104734 and
 DMS-1104751, and from the Banff International 
Research Station. 
The research of M.M.\ was also 
partially supported by the University of 
Geneva, the University
of Grenoble,   the University of Salamanca,
 the University of Amsterdam, and
 the Max-Planck Institute for  Mathematics
 in Bonn. B.S.\ also received a special research
 support from the Central Michigan University.
 \end{ack}


\providecommand{\bysame}{\leavevmode\hbox to3em{\hrulefill}\thinspace}

\bibliographystyle{amsplain}

\begin{thebibliography}{10}


\bibitem{AMM}
A.~Alexandrov, A.~Mironov and A.~Morozov, 
\emph{Unified description of correlators
in non-Gaussian phases of Hermitean matrix 
model},
arXiv:hep-th/0412099 (2004).

\bibitem{Ballard}
M.~R.~Ballard,
\emph{Meet homological mirror symmetry},
in ``Modular forms and string duality,'' 
Fields Inst.\ Commun., \textbf{54}, 191--224
 (2008).



\bibitem{Belyi}
G.~V.~Belyi, \emph{On galois extensions of a maximal cyclotomic fields}, Math.\
  U.S.S.R.\ Izvestija \textbf{14}, 247--256  (1980).





\bibitem{BEMS} 
G.~Borot, B.~Eynard,
M.~Mulase and B.~Safnuk, 
\emph{Hurwitz numbers, matrix models and 
topological recursion}, 	arXiv:0906.1206 [math.Ph] (2009).


\bibitem{BCMS}
V.~Bouchard, A.~Catuneanu, O.~Marchal
and P.~Su\l kowski,
\emph{The remodeling conjecture and the 
Faber-Pandharipande formula},
arXiv:1108.2689  (2011).


\bibitem{BKMP}
V.~Bouchard, A.~Klemm, M.~Mari\~no, and S.~Pasquetti,
\emph{Remodeling the {B}-model},
Commun.\ Math.\ Phys.
 \textbf{287}, 117--178 (2008).



\bibitem{BM}
V.~ Bouchard and M.~ Mari\~no, 
\emph{Hurwitz numbers, matrix models and enumerative geometry},
Proc.\ Symposia Pure Math.\ 
 \textbf{78}, 263--283 (2008).
 


\bibitem{Brini}
A.~Brini, 
\emph{The local Gromov-Witten theory of $\bC\bP^1$ and 
integrable hierarchies}, arXiv:1002.0582v1 [math-ph].



\bibitem{BEM}
A.~Brini, B.~Eynard, and M.~Mari\~no,
\emph{Torus knots and mirror symmetry}
arXiv:1105.2012.





\bibitem{CMS}
K.~Chapman, M.~Mulase, and B.~Safnuk,
\emph{
Topological recursion and the Kontsevich constants
for the volume of the moduli of curves}, Preprint (2010).



\bibitem{Chen}
L.~Chen,
\emph{Bouchard-Klemm-Marino-Pasquetti 
Conjecture for $\mathbb{C}^3$},
arXiv:0910.3739 (2009).




\bibitem{CGHJK}
R.M.~Corless, G.H.~Gonnet, D.E.G.~Hare,
D.J.~Jeffrey and D.E.~Knuth,
\emph{On the Lambert W-function},
Adv.\ Computational Math.\ \textbf{5},
329--359 (1996).



\bibitem{DFM}
R.~Dijkgraaf, H.~Fuji and M.~Manabe,
\emph{The volume conjecture, perturbative
knot invariants, and recursion relations for 
topological strings},
arXiv:1010.4542 [hep-th] (2010).

\bibitem{DV}
  R.~Dijkgraaf and C.~Vafa,
  \emph{Two Dimensional Kodaira-Spencer 
  Theory and Three Dimensional Chern-Simons
  Gravity},
  arXiv:0711.1932 [hep-th].




\bibitem{DVV}
R.~Dijkgraaf, E.~Verlinde, and H.~Verlinde,
\emph{Loop equations and {V}irasoro constraints in non-perturbative two-dimensional quantum gravity}, Nucl. Phys.
\textbf{B348}, 435--456 (1991).



\bibitem{D}
B.~Dubrovin, 
\emph{Geometry of 2D topological field theories}, 
in ``Integrable systems and quantum groups,'' 
Lecture Notes in Math.\ \textbf{1620}, 120--348 (1994). 

\bibitem{DZ}
B.~Dubrovin and Y.~Zhang,
\emph{Frobenius manifolds and Virasoro constraints},
Selecta Mathematica, New Ser.\  \textbf{5}, 423--466 (1999).

\bibitem{DZ2}
B.~Dubrovin and Y.~Zhang,
\emph{Virasoro Symmetries of the Extended Toda 
Hierarchy}, arXive:math/0308152 (2003).





\bibitem{DMMSS}
P.~Dunin-Barkowski, A.~Mironov, A.~Morozov, 
A.~Sleptsov, and A.~Smirnov.
\emph{Superpolynomials for toric knots from evolution induced by cut-and-join operators
}, arXiv:1106.4305 [hep-th].





\bibitem{ELSV} T.~Ekedahl, S.~Lando, M.~Shapiro, A.~Vainshtein,
{\em Hurwitz numbers and intersections on moduli spaces of curves},
Invent. Math. {\bf 146}, 297--327 (2001).







\bibitem{E2004} B.~Eynard, {\em Topological expansion for the
1-hermitian matrix model 
correlation functions}, arXiv:hep-th/0407261.


\bibitem{E2007} B.~Eynard, {\em Recursion between 
volumes of moduli spaces}, arXiv:0706.4403 [math-ph].



\bibitem{EMO}
  B.~Eynard, M.~Mari\~no and N.~Orantin,
  \emph{Holomorphic anomaly and matrix models},
  Journal of High Energy Physics {\bf 06} 058, (2007)
  [arXiv:hep-th/0702110].





  \bibitem{EMS} B.~Eynard, M.~Mulase
  and B.~Safnuk,
  \emph{The
{L}aplace transform  of the cut-and-join equation
and the {B}ouchard-{M}ari\~no conjecture on
{H}urwitz numbers}, arXiv:0907.5224 math.AG (2009).





\bibitem{EO1}
  B.~Eynard and N.~Orantin,
\emph{Invariants of algebraic curves and topological 
expansion},
Communications in Number Theory
and Physics {\bf 1},  347--452 (2007).


\bibitem{EO2} B.~Eynard and N.~Orantin, 
\emph{Weil-Petersson volume of moduli spaces, 
Mirzakhani's recursion and matrix models},
  arXiv:0705.3600 [math-ph] (2007).
  
  \bibitem{EO3}
  B.~Eynard and N.~Orantin,
\emph{Algebraic methods in random
matrices and enumerative geometry},
 arXiv:0811.3531 [math-ph] (2008).

 
 \bibitem{Getzler}
 E.~Getzler,
 \emph{Topological recursion relations in
 genus $2$}, arXiv:math/9801003 (1998).




  \bibitem{GJ}
  I.P.~Goulden and D.M.~Jackson,
  {\em Transitive factorisations into transpositions and
  holomorphic mappings on the sphere},
  Proc. A.M.S., \textbf{125}, 51--60 (1997).


 \bibitem{GJV1} I.P.~Goulden, D.M.~Jackson and A.~Vainshtein,
{\em The number of ramified coverings of the sphere by the torus and surfaces of higher genera},
Ann. of Comb. {\textbf{4}}, 27--46 (2000).

 \bibitem{GJV2} I.P.~Goulden, D.M.~Jackson and R.~Vakil, 
{\em The Gromov-Witten potential of a point, Hurwitz numbers,
and Hodge integrals},
Proc.\  London Math.\ Soc.\  \textbf{83:3}, 563--581 (2001).

 \bibitem{GJV3} I.P.~Goulden, D.M.~Jackson and R.~Vakil, 
{\em A short proof of the $\lambda_g$-conjecture without Gromov-Witten theory: Hurwitz theory and the moduli of curves},
arXiv:math/0604297v1 [math.AG] (2006).

\bibitem{GV1} T.~Graber and R.~Vakil,
{\em Hodge integrals and Hurwitz numbers via virtual localization},
Compositio Math. {\bf 135}, 25--36 (2003).

\bibitem{GV2} T.~Graber and R.~Vakil,
{\em Relative virtual localization and vanishing of tautological classes on moduli spaces of curves},
Duke Math.\ J.\ \textbf{130}, 1--37 (2005).



\bibitem{GS} S.~Gukov and P.~Su\l kowski,
\emph{A-polynomial, B-model, and quantization},
arXiv:1108.0002v1 [hep-th].


 
 

\bibitem{Harer}
J.~L. Harer, \emph{The cohomology of the moduli space of curves},
in Theory of
  Moduli, Montecatini Terme, 1985 (Edoardo Sernesi, ed.), Springer-Verlag,
  1988, pp.~138--221.

\bibitem{HZ}
J.~L. Harer and D.~Zagier, \emph{The {Euler} characteristic of the moduli
  space of curves}, Inventiones Mathematicae \textbf{85}, 457--485
   (1986).








\bibitem{Hirzebruch}
F.~Hirzebruch, 
\emph{Topological methods in algebraic geometry},
Third Edition, Springer -Verlag, 232 pages, 1966.


\bibitem{Hitchin}
N.~Hitchin,
\emph{Lectures on Special Lagrangian Submanifolds}, 
math.DG/9907034 (1999).

\bibitem{Mirror}
K.~Hori, S.~Katz,
A.~Klemm, R.~Pandharipande,
R.~Thomas, C.~Vafa,
R.~Vakil, and E.~Zaslow,
\emph{Mirror symmetry},
Clay Mathematics Monograph \textbf{1}, 
929 pages, American Mathematical Society, 2003.

\bibitem{H}
A.~Hurwitz,
\emph{\"Uber Riemann'sche Fl\"achen mit gegebene
Verzweigungspunkten},
Mathematische Annalen \textbf{39}, 1--66 (1891).




\bibitem{Joyce}
D.~Joyce,
\emph{Lectures on special Lagrangian geometry},
math.DG/0111111 (2001).


\bibitem{KMZ}
R.~Kaufmann, Yu.~Manin, and D.~Zagier,
\emph{Higher Weil-Petersson volume of moduli
space of stable $n$-pointed curves},
ArXiv:alg-geom/9604001 (1996).



\bibitem{Kodama}
Y.~Kodama, 
\emph{Combinatorics of
the dispersionless Toda hierarchy},
Lecture delivered at the International Workshop on
Nonlinear and Modern Mathematical Physics,
Beijing Xiedao Group, Beijing, China, July 15--21, 2009.


\bibitem{KP}
Y.~Kodama and V.U.~Pierce,
\emph{Combinatorics of dispersionless 
integrable systems and universality in random 
matrix theory},
arXiv:0811.0351 (2008).


\bibitem{K1992}
M.~Kontsevich,
\emph{Intersection theory on
the moduli space of
curves and the
matrix {Airy} function},
  Communications in Mathematical Physics
 \textbf{147}, 1--23  (1992).
 
 
 \bibitem{K1994}
M.~Kontsevich,
\emph{Homological algebra of mirror symmetry},
 arXiv:alg-geom/9411018  (1994).


\bibitem{Liu}C.-C.~M.~Liu,
\emph{Lectures on the ELSV formula},
arXiv:1004.0853. In ``Transformation Groups 
and Moduli Spaces of Curves,''
  Adv. Lect. Math. (ALM) \textbf{16}, 
  195--216,
  Higher Education Press and International Press, 
  Beijing-Boston, 2010.

\bibitem{LLZ} C.-C.~M.~Liu, K.~Liu, J.~Zhou,
\emph{A proof of a conjecture of Mari\~no-Vafa on 
Hodge Integrals},
J. Differential Geom.  \textbf{65},  no. 2,  
289--340 (2003). 




 
\bibitem{M1} M.~Mari\~no, {\em Chern-Simons theory, matrix models, and topological strings}, Oxford University Press, 2005.

\bibitem{M2}
  M.~Mari\~no,
  \emph{Open string amplitudes and large order behavior in topological string
  theory},
  JHEP {\bf 0803}, 060 (2008).



\bibitem{MirMor}
A.~Mironov and A.~Morozov,
\emph{Virasoro constraints for Kontsevich-Hurwitz partition function},
arXiv:0807.2843 [hep-th] (2008).




\bibitem{M1995}
M.~Mulase,
\emph{Asymptotic analysis of a Hermitian matrix integral}, International Journal of Mathematics \textbf{6}, 881--892  (1995).




 \bibitem{MP1998}
 M.~Mulase and M.~Penkava, \emph{
 Ribbon graphs, quadratic differentials on Riemann surfaces, and algebraic curves defined over $\overline{\mathbb{Q}}$}, The Asian Journal of Mathematics  \textbf{2} (4), 875--920 (1998).


   \bibitem{MP2010}
 M.~Mulase and M.~Penkava, 
 \emph{Topological recursion for the Poincar\'e 
 polynomial of the combinatorial moduli space of curves}, 
 Preprint arXiv:1009.2135 math.AG (2010).
 


\bibitem{MS} M. Mulase and B. Safnuk, {\em Mirzakhani's recursion relations, Virasoro constraints and the KdV hierarchy},
Indian J. Math. \textbf{50}, 189--228 (2008).




\bibitem{MZ} M.~Mulase and N.~Zhang,
\emph{
Polynomial recursion formula for linear Hodge integrals},  Communications in Number Theory
and Physics {\bf 4}, (2010).

  \bibitem{Mumford} D.~Mumford,
  \emph{
  Towards an enumerative geometry of the moduli space of curves}
  (1983),
  in ``Selected Papers of David Mumford,'' 235--292 (2004).



  \bibitem{N1} P.~Norbury,
  \emph{Counting lattice points in the moduli space of curves},
  arXiv:0801.4590  (2008).


  \bibitem{N2} P.~Norbury,
  \emph{String and dilaton equations for counting lattice points in the moduli space of curves},
  arXiv:0905.4141  (2009).


\bibitem{NS1}
P.~Norbury and N.~Scott,
\emph{Polynomials representing Eynard-Orantin
invariants},
arXiv:1001.0449 (2010).


\bibitem{NS2}
P.~Norbury and N.~Scott,
\emph{Gromov-Witten invariants of $\bP^1$ 
and Eynard-Orantin invariants},
arXiv:1106.1337 (2011).

\bibitem{O1} A.~Okounkov, {\em Random matrices and random perputations}, 
International Mathematics Research Notices \textbf{2000},
1043--1095  (2000).




\bibitem{O2} A.~Okounkov, {\em Toda equations for Hurwitz numbers}, Math. Res. Lett. {\bf 7}, 447 (2000) 
[arXiv:math.AG/0004128].





\bibitem{OP1}
A.~Okounkov and R.~Pandharipande,
\emph{Gromov-Witten theory, Hurwitz numbers, and matrix 
models, I}, 
Proc.\ Symposia Pure Math.\ 
 \textbf{80}, 325--414 (2009).
 
 \bibitem{OP2}
A.~Okounkov and R.~Pandharipande,
\emph{Gromov-Witten theory, Hurwitz theory, 
and completed cycles}, 
Ann.\ Math.\ 
 \textbf{163}, 517--560 (2006).



\bibitem{OP3}
A.~Okounkov and R.~Pandharipande,
\emph{The equivariant Gromov-Witten theory of 
$\mathbb{P}^1$}, 
arXiv:math/0207233 [math.AG] (2002).

\bibitem{OP4}
A.~Okounkov and R.~Pandharipande,
\emph{Virasoro constraints for target curves}, 
arXiv:math/0308097 [math.AG] (2003).


\bibitem{OP5}
A.~Okounkov and R.~Pandharipande,
\emph{Hodge integrals and invariants of the unknot}, 
Geom.\ Topol.\ \textbf{8}, 675--699  (2004).




\bibitem{OSY}
H.~Ooguri, P.~Sulkowski, M.~Yamazaki,
\emph{Wall Crossing As Seen By Matrix Models},
arXiv:1005.1293 (2010).



\bibitem{Schneps}
L.~Schneps, Editor,
\emph{The Grothendieck theory of 
dessins d'enfants},
London Mathematical Society Lecture
Notes Series \textbf{200}, 368 pages, 1994.





\bibitem{SL}
L.~Schneps and P.~Lochak, Editors,
\emph{
Geometric Galois actions} \textbf{1},
London Mathematical Society Lecture
Notes Series \textbf{242}, 293 pages,1997.


\bibitem{STT}
D.~D. Sleator, R.~E. Tarjan, and W.~P. Thurston, \emph{Rotation
  distance, triangulations, and hyperbolic geometry}, Journal of the American
  Mathematical Society \textbf{1}, 647--681 (1988).


\bibitem{Stanley}
R.~P.~Stanley, 
\emph{Enumerative combinatorics}
volume 2, Cambridge University Press, 2001.


\bibitem{Strebel}
K.~ Strebel, \emph{Quadratic differentials}, Springer-Verlag, 1984.




\bibitem{tH}
G.~'t Hooft,
\emph{
A planer diagram theory for
strong interactions},
Nuclear Physics \textbf{B
72},
461--473 (1974).



\bibitem{V} R.~Vakil, Harvard Thesis 1997.




\bibitem{W1991}
E.~Witten, \emph{Two dimensional gravity and
intersection
theory on moduli space}, Surveys in
Differential Geometry \textbf{1},
243--310 (1991).


\bibitem{Zhou1} J.~Zhou,
{\em Hodge integrals, Hurwitz numbers, and symmetric groups}, preprint, arXiv:math/0308024 [math.AG] (2003).

\bibitem{Zhou2}
J.~Zhou,
\emph{On computations of Hurwitz-Hodge integrals},
arXiv:0710.1679  (2007).


\bibitem{Zhou3}
J.~Zhou,
\emph{Local Mirror Symmetry for One-Legged 
Topological Vertex},
arXiv:0910.4320  (2009).

\bibitem{Zhou4}
J.~Zhou,
\emph{Local Mirror Symmetry for the Topological 
Vertex}
arXiv:0911.2343   (2009).

\bibitem{Zhu}
S.~Zhu,
\emph{On a proof of the Bouchard-Su\l kowski 
conjecture},
arXiv:1108.2831 (2011).

\bibitem{Zvonkine}
D.~Zvonkine,
\emph{An algebra of power series 
arising in the intersection theory of moduli spaces 
of curves and in the enumeration of ramified 
coverings of the sphere},
arXiv:math.AG/0403092 (2004).


\end{thebibliography}

\end{document}